\documentclass[11pt,letterpaper]{amsart}
\usepackage{amssymb}
\usepackage{stmaryrd}
\usepackage[all]{xy}
\usepackage{mathrsfs} 
\usepackage{color}
\usepackage[marginratio=1:1,height=584pt,width=430pt,tmargin=117pt]{geometry}
\usepackage[linktocpage=true]{hyperref}
\usepackage{enumitem}
\usepackage{color}
\usepackage{tikz-cd}

\usetikzlibrary{decorations.markings}
\tikzset{double line with arrow/.style args={#1,#2}{decorate,decoration={markings,%
mark=at position 0 with {\coordinate (ta-base-1) at (0,1pt);
\coordinate (ta-base-2) at (0,-1pt);},
mark=at position 1 with {\draw[#1] (ta-base-1) -- (0,1pt);
\draw[#2] (ta-base-2) -- (0,-1pt);
}}}}
\tikzset{Equal/.style={-,double line with arrow={-,-}}}

\numberwithin{itemcounter}{subsection}

\makeatletter
\newcommand{\nolisttopbreak}{\vspace{\topsep}\nobreak\@afterheading}
\makeatother

\makeatletter
\newcommand{\rellab}[2]{%
   \protected@write \@auxout {}{\string \newlabel {#1}{{#2}{\thepage}{#2}{#1}{}} }%
   \hypertarget{#1}{#2}
}
\makeatother

\def\C{\mathbb{C}}

\def\Hb{\mathbf{H}}

\def\Gm{\mathbb{G}_m}
\def\N{\mathbb{N}}
\def\Z{\mathbb{Z}}
\def\Q{\mathbb{Q}}
\def\Coh{\mathcal{C}oh}
\def\oS{{\overline{S}}}
\def\RCoh{\mathfrak{Coh}}
\def\RHilb{\mathfrak{Hilb}}
\def\Hom{\operatorname{Hom}\nolimits}
\def\RHom{\operatorname{RHom}\nolimits}
\def\Ext{\operatorname{Ext}\nolimits}
\def\Hbar{\overline{H}}

\def\r{\mathbf{r}}

\def\ch{\mathrm{ch}}
\def\red{\mathrm{red}}

\def\supp{\mathrm{supp}}
\def\Hilb{\mathfrak{Hilb}}
\def\Hilbrm{\mathrm{Hilb}}
\def\Higgs{\mathfrak{Higgs}}
\def\geq{{\geqslant}}
\def\leq{{\leqslant}}
\def\Gr{\operatorname{Gr}\nolimits}
\def\Sym{\operatorname{Sym}\nolimits}
\def\Spec{\operatorname{Spec}\nolimits}
\def\Exp{\operatorname{Exp}\nolimits}
\def\End{\operatorname{End}\nolimits}

\def\Ad{\operatorname{Ad}\nolimits}
\def\dim{\operatorname{dim}\nolimits}
\def\codim{\operatorname{codim}\nolimits}

\def\eu{\mathrm{eu}}
\def\ev{\operatorname{ev}\nolimits}
\def\rk{\operatorname{rk}\nolimits}
\def\Ec{\mathcal{E}}

\def\tilD{\widetilde{D}}
\def\c{\mathbf{c}}
\def\op{\mathrm{op}}
\def\taut{\mathrm{taut}}
\def\vtaut{\mathrm{vtaut}}

\def\Pic{\mathfrak{Pic}}
\def\pt{[\mathrm{pt}]}
\def\Heis{\mathfrak{h}}
\def\Vir{\mathrm{Vir}}
\def\len{\mathrm{lg}}
\def\Hpat{\mathfrak{H}}
\def\oHpat{\overline{\mathfrak{H}}}
\def\Tot{\mathrm{Tot}}
\def\Id{\mathrm{Id}}
\def\Td{\mathrm{Td}}

\def\Wc{W_\downarrow}
\def\Wn{W_\uparrow}
\def\wnn{W_{\uparrow\mathrel{\mspace{-2mu}}\uparrow}}
\def\wcn{W_{\downarrow\mathrel{\mspace{-2mu}}\uparrow}}
\def\wnc{W_{\uparrow\mathrel{\mspace{-2mu}}\downarrow}}
\def\wcc{W_{\downarrow\mathrel{\mspace{-2mu}}\downarrow}}
\def\Wnn#1{W_{\uparrow\mathrel{\mspace{-2mu}}\uparrow}^{(#1)}}
\def\Wnc#1{W_{\uparrow\mathrel{\mspace{-2mu}}\downarrow}^{(#1)}}
\def\Wcn#1{W_{\downarrow\mathrel{\mspace{-2mu}}\uparrow}^{(#1)}}
\def\Wcc#1{W_{\downarrow\mathrel{\mspace{-2mu}}\downarrow}^{(#1)}}
\def\DW#1{W^{(#1)}}

\def\cI{\mathcal{I}}
\def\cJ{\mathcal{J}}
\def\uch{\underline{\mathrm{ch}}}
\def\Span{\mathrm{Span}}
\def\Mukai{\mathfrak{Mukai}}

\theoremstyle{plain}

\newtheorem*{corollary*}{Corollary}

\newtheorem{theorem}{Theorem}[section]
\newtheorem{lemma}[theorem]{Lemma}
\newtheorem{lemma-definition}[theorem]{Lemma-Definition}
\newtheorem{definition-lemma}[theorem]{Definition-Lemma}

\newtheorem{proposition}[theorem]{Proposition}
\newtheorem{conjecture}[theorem]{Conjecture}
\newtheorem{corollary}[theorem]{Corollary}

\newcounter{thmintros}

\newtheorem{thmintro}[thmintros]{Theorem}
\newtheorem{thmintroalt}{Theorem}[thmintros]
\newenvironment{thmintrop}[1]{
  \renewcommand\thethmintroalt{#1}
  \thmintroalt
}{\endthmintroalt}

\theoremstyle{definition}
\newtheorem{definition}[theorem]{Definition}
\theoremstyle{remark}
\newtheorem{remark}[theorem]{Remark}
\newtheorem{example}[theorem]{Example}
\numberwithin{equation}{section}

\author[A.~Mellit, A.~Minets, O.~Schiffmann, E.~Vasserot]{Anton Mellit, Alexandre Minets, Olivier Schiffmann, Eric Vasserot}
\address[Anton Mellit]{Faculty of Mathematics, University of Vienna, Oskar-Morgenstern-Platz 1, 1020 Vienna, Austria}
\email{anton.mellit@univie.ac.at}
\address[Alexandre Minets]{Max Planck Institute for Mathematics, Vivatsgasse 7, 53111 Bonn, Germany \\ and Mathematisches Institut, University of Bonn, Endenicher Allee 60, 53115 Bonn, Germany}
\email{minets@mpim-bonn.mpg.de}
\address[Olivier Schiffmann]{D\'epartement de Math\'ematiques, Universit\'e de Paris-Sud Paris-Saclay, Bat. 307, 91405 Orsay Cedex, France \\ and Simion Stoilow Institute of Mathematics
P.O. Box 1-764, RO-014700 Bucharest, Romania }
\email{olivier.schiffmann@universite-paris-saclay.fr}
\address[Eric Vasserot]{Université Paris Cité, 75013 Paris, France, UMR7586 (CNRS), Institut Universitaire de France (IUF)}
\email{eric.vasserot@imj-prg.fr}

\title[Coherent sheaves, COHAs and deformed $W_{1+\infty}$-algebras]{Coherent sheaves on surfaces, COHAs and \\deformed $W_{1+\infty}$-algebras} 

\begin{document}

\setcounter{section}{-1}

\begin{abstract}We compute the COHA of zero-dimensional sheaves on an arbitrary smooth quasi-projective surface $S$ with pure cohomology, deriving an explicit presentation by generators and relations. When $S$ has trivial canonical bundle, this COHA is isomorphic to the enveloping algebra of deformed trigonometric $W_{1+\infty}$-algebra associated to the ring $H^*(S,\Q)$. We also define a double of this COHA, show that it acts on the homology of various moduli stacks of sheaves on $S$ and explicitly describe this action on the products of tautological classes. 
Examples include Hilbert schemes of points on surfaces, the moduli stack of Higgs bundles on a smooth projective curve and the moduli stack of $1$-dimensional sheaves on a $K3$ surface in an ample class. The double COHA is shown to contain Nakajima's Heisenberg algebra, as well as a copy of the Virasoro algebra.

\end{abstract}

\maketitle

\setcounter{tocdepth}{1}
\tableofcontents

\section{Introduction}

Let $S$ be a smooth quasi-projective complex surface. In the pioneering work \cite{NakLectures}, Nakajima constructed an action of a Heisenberg algebra $\mathfrak{h}_S$ on the direct sum $\mathbf{V}(S)=\bigoplus_{n \geq 0} H^*(\Hilbrm_n(S),\Q)$ of cohomology groups of the Hilbert schemes of points on $S$. Here, $\mathfrak{h}_S$ is modeled on the cohomology ring $H^*(S,\Q)$. What's more, Nakajima identified $\mathbf{V}(S)$ with the Fock space representation of $\mathfrak{h}_S$, thereby providing a very fruitful bridge between the enumerative geometry of $\Hilbrm(S)$ and the representation theory of Heisenberg algebras.
This has led to a flurry of remarkable results on the topology of Hilbert schemes of points on surfaces or of instanton spaces, e.g., \cite{Lehn}, \cite{Lehnsorger}, \cite{Vasserot}, \cite{SVV}, ~\ldots, and has served as model for the theory of quiver varieties. Similar constructions exist also in the $K$-theoretic context, see e.g.~\cite{NegutShuffle}, and may be upgraded to the $T$-equivariant setting in the presence of a torus action on $S$. 

\medskip

Nakajima operators arise from the correspondences
$$\xymatrix{\Hilbrm_n(S) \times S & \Hilbrm_{n,n + k}(S) \ar[r]^-{p} \ar[l]_-{q} & \Hilbrm_{n+k}(S)}$$
and their transposes. Here $\Hilbrm_{n,n+ k}(S)$ is the nested Hilbert scheme parametrizing pairs of subschemes $Z\subset Z'$ of respective lengths $n$, $n+k$.
The scheme $\Hilbrm_{n,n+ k}(S)$ carries the tautological vector bundle $H^0(S, \mathcal{O}_{Z'}/\mathcal{O}_Z)$. Taking cup product with the characteristic classes of these bundles yields additional operators, generating a much larger algebra than $U(\mathfrak{h}_S)$.
For $S=\C^2$ equipped with the natural $(\C^*)^2$-action, this algebra was studied in \cite{SVIHES}, where it was 
identified with the affine Yangian of $\mathfrak{gl}_1$. See \cite{MO} for a different approach\footnote{the $K$-theoretic version was considered in \cite{SVDuke} where it was identified with the elliptic Hall algebra (see also \cite{FeiginTsymbaliuk})}. In loc.cit. the same algebra was shown to act on the cohomology of any of the instanton 
spaces, which are moduli spaces of framed, torsion-free sheaves on $\C^2$. The affine Yangian of $\mathfrak{gl}_1$ 
is a two-parameter deformation of the algebra $W_{1+\infty}$ of differential operators on the circle. Its 
representation theory is strongly related to that of affine $W$-algebras of $\mathfrak{gl}_r$ (for all $r$), in accordance with the Alday-Gaiotto-Tachikawa conjecture, \cite{AGT}.

\medskip

The aim of this paper is to provide a generalization of the above results to the case of an arbitrary smooth quasi-projective surface $S$ which is cohomologically pure (for instance, projective). This provides actions of explicit 
infinite-dimensional algebras that we call \textit{deformed $W_{1+\infty}$-algebras} on the Borel-Moore homology of 
many interesting moduli stacks of coherent sheaves on $S$. Our approach is based on the theory of  cohomological 
Hall algebras, which we now recall.

\medskip

\subsection{Cohomological Hall algebras.} Let $\mathcal{C}$ be a $\C$-linear Abelian category satisfying suitable finiteness conditions, such as in \cite[\S~5.1]{DHSM}, and let $\mathfrak{M}_\mathcal{C}$ denote the derived stack of objects in $\mathcal{C}$. The prime example of interest for us is the category of coherent sheaves on an algebraic surface $S$. Extensions in $\mathcal{C}$ are controled by the Hecke correspondence
\begin{equation}\label{E:Heckegeneral}
\xymatrix{\mathfrak{M}_\mathcal{C} \times \mathfrak{M}_\mathcal{C} & \widetilde{\mathfrak{M}}_\mathcal{C} \ar[r]^-{p}\ar[l]_-{q} & \mathfrak{M}_\mathcal{C}}
\end{equation}
where $\widetilde{\mathfrak{M}}_\mathcal{C}$ is the stack of short exact sequences in $\mathcal{C}$. Here $p$ and $q$ associate to a sequence its middle and extreme terms. The properties of the maps $p,$ $q$ depend heavily on 
the global dimension of $\mathcal{C}$. Crucially, the map $q$ is quasi-smooth when $\mathcal{C}$ is of global 
dimension at most $2$. When in addition $p$ is proper, the composition 
$$p_*q^! :H_*(\mathfrak{M}_{\mathcal{C}},\Q)^{\otimes 2} \to H_*(\mathfrak{M}_{\mathcal{C}},\Q)$$ 
yields a structure of an associative algebra on the Borel-Moore homology $H_*(\mathfrak{M}_{\mathcal{C}},\Q)$. This is the cohomological Hall algebra (COHA) of $\mathcal{C}$ introduced in \cite{KV}; it can also be seen as a dimensionally reduced version of 3d COHA introduced earlier in~\cite{KS}.
Furthermore, any locally closed susbtack $\mathfrak{M}^\circ_\mathcal{C} \subset \mathfrak{M}_\mathcal{C}$ for which \eqref{E:Heckegeneral} restricts to a correspondence
\begin{equation*}
\xymatrix{\mathfrak{M}_\mathcal{C} \times \mathfrak{M}^\circ_\mathcal{C} & \widetilde{\mathfrak{M}}^\circ_\mathcal{C} \ar[r]^-{p}\ar[l]_-{q} & \mathfrak{M}^\circ_\mathcal{C}}
\end{equation*}
gives rise to a $H_*(\mathfrak{M}_\mathcal{C},\Q)$-module structure on $H_*(\mathfrak{M}^\circ_\mathcal{C},\Q)$.
Such substacks are called Hecke patterns in \cite{KV}. In other words, the same algebra 
$H_*(\mathfrak{M}_\mathcal{C},\Q)$ acts simultaneously on the homology of \textit{all} Hecke patterns. 
Hecke patterns may for instance be constructed using stability conditions and/or framings. 
This construction appears in \cite{SVDuke} in the $K$-theoretic setting and in \cite{SVIHES} 
in relation to quiver varieties and instanton spaces, where it gives rise to Yangians of Kac-Moody algebras. 
COHAs were further defined and studied in more general contexts, see, e.g., \cite{KS}, \cite{Yang-Zhao}, \cite{Minets}, \cite{SalaSchiffmann}, \cite{DavisonCY2}, \cite{KV}, \cite{Zhao}, \cite{Porta-Sala},\,\ldots 

\medskip

The main object of study of this paper is the COHA of the category of zero-dimensional coherent sheaves on a smooth surface $S$. Such COHAs were previously considered in \cite{Minets}, \cite{KV} and, in the $K$-theoretical context, in \cite{Zhao} and \cite{NegutShuffle}, where quadratic relations of Ding-Iohara type between degree one generators were found and actions on smooth moduli spaces were constructed. Here, we focus on the Borel-Moore homology COHA and fully determine this COHA under the assumption that $S$ has a pure cohomology. Our method does not presuppose that $S$ is toric, or rely on localisation techniques.
As far as we are aware, the only case in which such a COHA was fully determined before was $S=\mathbb{A}^2$ with a torus action, see \cite{Davison22}, \cite{SVIHES}.

\medskip

\subsection{Deformed $W_{1+\infty}$-algebras associated to a surface.} In \S\ref{sec:Fock-ops} and \S\ref{sec:W-open-surf}, to which we refer for details, we introduce and begin the study of a family of associative algebras $W(S)$ attached to smooth, pure surfaces $S$. Let us begin by assuming that $S$ is proper. Let $c_1, c_2$ be the Chern classes of $S$ and $s_2=c_1^2-c_2$. The algebra $\DW{\c}(S)$ is generated by collections of elements 
$$\{T^\pm_n(\lambda), \psi_n(\lambda)\,:\, n \geq 0, \lambda \in H^*(S,\Q)\}$$ and a central element $\c$ modulo relations among which the most important ones are
\[ [\psi_m(\lambda), T^\pm_n(\mu)] = \pm m T^\pm_{m+n-1}(\lambda\mu), \]

\[ \begin{split}[T^\pm_{m}&(\lambda), T^\pm_{n+3}(\mu)] 
-3 [T^\pm_{m+1}(\lambda), T^\pm_{n+2}(\mu)] + 3 [T^\pm_{m+2}(\lambda), T^\pm_{n+1}(\mu)]  - [T^\pm_{m+3}(\lambda), T^\pm_{n}(\mu)]\\
 &- [T^\pm_{m}(\lambda), T^\pm_{n+1}(s_2 \mu)] + 
 [T^\pm_{m+1}(\lambda), T^\pm_{n}(s_2 \mu)] \pm \{T^\pm_m, T^\pm_n\}(c_1\Delta_S\lambda \mu)=0,\end{split}\]

\[\sum_{w\in S_3} w \cdot [T^\pm_{m_3}(\lambda_3), [T^\pm_{m_2}(\lambda_2), T^\pm_{m_1+1}(\lambda_1)]] = 0\]
as well as the double relation~\eqref{E:+-relationsdouble}, which expresses the commutators $[T^+_m(\lambda),T^-_{m'}(\mu)]$ as polynomials in $\psi_n$'s.
We denote by $W^\pm(S)$, resp. $W^0(S)$ the subalgebras generated by $\{T^\pm_n(\lambda)\}$ and $\{\psi_n(\lambda),\c\}$ respectively. The algebra $\DW{\c}(S)$ is $\Z \times \N$-graded.
Here $T^\pm_n(\lambda)$ and $\psi_n(\lambda)$ have degrees $(\pm 1, 2n-2+\deg(\lambda))$ and $(0,2n-2+\deg(\lambda))$ respectively.

\medskip

\subsection{Main results} Let us now describe our main results. We refer to the body of the text for details.  

\begin{thmintro}[Theorem~\ref{Prop:defHeis}, Propositions~\ref{prop:big-W-triang}, \ref{prop:heis-in-double}, \ref{prop:Vir-in-double}]\label{thmA} Let $S$ be a smooth and proper surface. 
\begin{enumerate}[label=$\mathrm{(\alph*)}$,leftmargin=8mm]
\item There is a triangular decomposition $\DW{\c}(S) \simeq W^-(S) \otimes W^0(S) \otimes W^+(S)$,
\item The graded character of $W^+(S)$ is given by
\begin{equation*}
    P_{W^+(S)}(z,w)=\Exp\left(\frac{P_S(z)z^{-2}w}{(1-z^2)(1-w)} \right)    
\end{equation*} 
where $P_S(z)$ is the Poincar\'e polynomial of $S$, and $\Exp$ is the plethystic exponential,
\item There are embeddings $U(\mathfrak{h}_S) \subset\DW{\c}(S)$ and 
$U(\Vir_S) \subset \DW{\c}(S)$, where $\mathfrak{h}_S$ and $\Vir_S$ are the Heisenberg and Virasoro algebras 
modeled on $H^*(S,\Q)$. The central charges of $\mathfrak{h}_S$ and $\Vir_S$ as functions of 
$\lambda,\mu\in H^*(S)$ are given by
$$ C_\mathfrak{h}= \c\int_S\lambda\mu, \qquad \eta_{\Vir}=\c\left(\int_S c_2\lambda\mu - (1-\c^2)\int_S c_1^2\lambda\mu+2\psi_0(c_1\lambda\mu)\right).$$
\end{enumerate}
\end{thmintro}

\medskip

When the surface $S$ has trivial canonical bundle, the $W$-algebra turns out to be the enveloping algebra of a Lie algebra. 

\begin{thmintrop}{\ref*{thmA}$'$}[Theorem~\ref{T:W(S)}]\label{thmAp} Let $S$ be a proper surface.
Assume that $c_1=0$ and $s_2=q^2$ for some $q \in H^2(S)$. Let $\mathfrak{w}^+(S)$ 
be the Lie algebra spanned by elements $ z^mD^n\lambda$ with $m\geq 1,n \geq 0, \lambda \in H^*(S,\Q)$, 
subject to the relations
$$ [z^mD^n\lambda\,,\,z^{m'}D^{n'}\mu] = z^{m+m'}\frac{ (D+m'q)^nD^{n'} - D^n(D+mq)^{n'} }{q}\lambda\mu.$$
We have an isomorphism
$$W^+ (S) \simeq U(\mathfrak{w}^+(S)).$$
\end{thmintrop}

\medskip

We also construct a natural representation $\mathbf{F}^{(r)}(S)$ of the algebra
$\DW{r}(S)=\DW{\c}(S)|_{\c=r}$ for any integer $r \geq 0$.
We call it the level $r$ Fock space. As a vector space, we have
$$\mathbf{F}^{(r)}(S)=\Lambda(S)[s,s^{-1}]|_{\r=r}$$
where 
$$\Lambda(S)=\Q[p_n(\lambda),\r\,:\, n \geq 0, \lambda \in H^*(S,\Q)].$$ The space $\Lambda(S)$ may be understood as the ring of universal tautological classes on the stack $\RCoh(S)$ of coherent sheaves on $S$, see \S\ref{sec:taut}. We prove in \S\ref{sec:Fock4} the following

\begin{thmintrop}{\ref*{thmA}$''$}[Corollary~\ref{cor:Fock-is-module}, Remark~\ref{E:vertexoperators}]\label{thmApp} For any $r\in \mathbb{N}$, there is an action of $\DW{r}(S)$ on $\mathbf{F}^{(r)}(S)$ via Fourier modes of the vertex operators
\begin{align*}
\Theta^+(s)&=\exp\Big( \sum_{\substack{\gamma;k\geq 1}} \frac{p_k}{k}(\gamma)\otimes \gamma^{*}s^{-k}\Big)_{[s^{<r}]}
\exp\Big(-\sum_{\substack{\gamma;n \geq 0}}\frac{\partial}{\partial \kappa_{n}(\gamma)} \otimes \gamma s^n\Big)\\
\Theta^-(s)&=\exp\Big( -\sum_{\substack{\gamma;k\geq 1}} \frac{\tau_{c_1}{p}_k}{k}(\gamma)\otimes \gamma^{*}s^{-k}\Big)_{[s^{<-r}]}
\exp\Big(\sum_{\substack{\gamma;n \geq 0}}\frac{\partial}{\partial \kappa_{n}(\gamma)} \otimes \gamma s^n\Big)
\end{align*}
where $\{\gamma\},$ $\{\gamma^*\}$ are dual bases of $H^*(S,\Q)$ and where the elements $\{\kappa_n(\lambda)\}$ are related to the $\{p_k(\lambda)\}$ through relation \eqref{E:pkappa} involving the Todd class of $S$, and where $\tau_{c_1}$ is a certain shift automorphism of $\Lambda(S)$, see \S\ref{Sec:notations}. This representation is faithful for $r>0$, but it is neither irreducible nor highest weight.
\end{thmintrop}

\medskip

In the case of open surfaces $S$, we may replace $H^*(S,\Q)$ by either $H^*(S,\Q)$ or $H^*_c(S,\Q)$, 
the cohomology with compact supports. It turns out that both are important for applications. 
This leads us to define four versions $\wnn^{(\c)}(S),$ $ \wnc^{(\c)}(S),$ $ \wcc^{(\c)}(S)$ and $\wcn^{(\c)}(S)$ of the
deformed $W_{1+\infty}$-algebra, depending on a choice of $H^*$ or $H^*_c$ for each half $W^+(S),$ $ W^-(S)$. 
Assuming that $S$ is pure, we extend Theorems~\ref{thmA}, \ref{thmAp} and~\ref{thmApp} to open 
surfaces, see \S\ref{sec:W-open-surf}. The above results still hold in the presence of a torus $T$ acting on 
$S$. Then we consider all spaces as free modules over $H^*(BT)$. Finally,  the assignment 
$S \mapsto \wnn^{(\c)}(S)$ is functorial with respect to open immersions. 
Similar result holds for the other types of $W$-algebras, see \S\ref{sec:W-open-surf}.

\medskip

Let us now return to COHAs. Let $S$ be smooth and cohomologically pure, and let $\RCoh_{n\delta}$ denote the derived stack of length $n$ coherent sheaves on $S$. The COHAs which we are interested in are
$$\mathbf{H}_0(S)=\bigoplus_{n \geq 0} H_*(\RCoh_{n\delta},\Q)$$
and its compactly supported version 
$$\mathbf{H}^c_0(S)=\bigoplus_{n \geq 0} H_*(\RCoh_{n\delta}/\Sym^n(S),\Q),$$ 
which is defined using hyperbolic Borel-Moore homology with respect to the support map $\RCoh_{n\delta} \to \Sym^n(S)$, see Appendix~\ref{sec:appA} for definitions. Both $\mathbf{H}_0(S)$ and $\mathbf{H}^c_0(S)$ are functorial with respect to open immersions.

\medskip

\begin{thmintro}[\S\ref{sec:Tmain}]\label{thmB} There are algebra isomorphisms $\Theta_S:\mathbf{H}_0(S) \simeq \Wn^+(S)$ and $\Theta^c_S:\mathbf{H}^c_0(S) \simeq \Wc^+(S)$. In particular, the algebra $\mathbf{H}_0(S)$ is spherically generated.
\end{thmintro}

\medskip

These isomorphisms are compatible with open immersions. We may include the ring of universal tautological classes $\Lambda(S)$ by forming semi-direct products
$$\widetilde{\mathbf{H}}_0(S)=\Lambda(S) \ltimes \mathbf{H}_0(S), \qquad \widetilde{\mathbf{H}}^{c}_0(S)=\Lambda(S) \ltimes \mathbf{H}^{c}_0(S).$$
The isomorphisms above extend to $\widetilde{\mathbf{H}}_0(S) \simeq \Wn^\geq(S)$ and $\widetilde{\mathbf{H}}^c_0(S) \simeq \Wc^\geq(S)$. The above results hold mutatis mutandis in the presence of a torus $T$ acting on $S$.

\begin{corollary*} If $c_1\Delta_S=0$ and there exists $q \in H^2(S,\Q)$ such that $s_2=q^2$, then
$\mathbf{H}_0(S) \simeq U(\mathfrak{w}^+(S))$.
\end{corollary*}

This corollary is in accordance with the general philosophy of Donaldson-Thomas theory for $2$ Calabi-Yau categories.
In particular, $\mathfrak{w}^+(S)$ is the Lie algebra constructed in \cite{DHSM}. As a vector space, it is isomorphic to $\mathfrak{g}^{\text{\tiny{$\mathrm{BPS}$}}}[u]$, where $\mathfrak{g}^{\text{\tiny{$\mathrm{BPS}$}}}$ is the BPS Lie algebra of $\Coh_0(S)$ determined in \textit{loc. cit}.

\medskip

Our proof of Theorem~\ref{thmB} involves the construction and comparison of suitable faithful representations of both $W^\geq(S)$ and $\mathbf{H}_0^c(S)$.
Following~\cite{KV} we introduce several notions of Hecke patterns in \S\ref{sec:Hecke-patterns}. 
We deduce from the general formalism of COHAs that a left/right $S$-strong, resp. $S$-weak, Hecke pattern $X$ gives rise to a left/right action of $\mathbf{H}_0(S)$, resp. $\mathbf{H}^c_0(S)$, on the space
$$\mathbf{V}(X)=\bigoplus_\alpha H_*(X_\alpha,\Q).$$
Denote by $\ev:\Lambda(S)\to H^*(X,\Q)$ the evaluation map.
For any class $\alpha \in \bigoplus_i H^{2i}(\oS,\Q)$, let $[X_\alpha]$ and $[X_\alpha^{cl}]$ be the virtual and classical fundamental class of $X_\alpha$. We denote by
$$\mathbf{V}^{\vtaut}(X)=\bigoplus_\alpha\ev(\Lambda(S)) \cap [X_\alpha], \qquad \mathbf{V}^{\taut}(X)=\bigoplus_\alpha\ev(\Lambda(S)) \cap [X_\alpha^{cl}],$$
the subspace of virtual and classical tautological classes in $\mathbf{V}(X)$.
Abusing notation, we denote the induced maps from the Fock space $\mathbf{F}^{(r)}$ to $\mathbf{V}^{\vtaut}(X)$, $\mathbf{V}^{\taut}(X)$ by $\ev$ as well.

\begin{thmintro}[Proposition~\ref{P:regularHP}, Corollary~\ref{cor:regularHPcompact}]\label{thmC} 
\hfill
\begin{enumerate}[label=$\mathrm{(\alph*)}$,leftmargin=8mm]
\item Let $X$ be a left $S$-strong Hecke pattern of rank $r$. The action $\Psi^+_X$ of $\mathbf{H}_0^+(S)$ on $\mathbf{V}(X)$ preserves the subspace $\mathbf{V}^{\vtaut}(X)$, and there is a commutative diagram
\begin{equation}\label{Eq:heckepatternthem}
\begin{split}
\xymatrix{\Wn^+(S) \otimes \mathbf{F}^{(r)} \ar[d]^-{\Theta_S \otimes \ev} \ar[r]^-{\Phi^+}& \mathbf{F}^{(r)}(S) \ar[d]^-{\ev}\\
\mathbf{H}_0^+(S) \otimes \mathbf{V}^{\vtaut}(X) \ar[r]^-{\Psi_X^+}& \mathbf{V}^{\vtaut}(X)
}
\end{split}
\end{equation}
Similar results hold for $S$-weak Hecke patterns, and for right Hecke patterns.
\item Let $X$ be a two-sided Hecke pattern, so that we have both an action of $\Wn^+(S)$ (or $\Wc^+(S)$) and of $\Wn^-(S)$ (or $\Wc^-(S)$) on $\mathbf{V}(X)$. Then \eqref{Eq:heckepatternthem} extends to an action of $\DW{r}(S)$ on $\mathbf{V}(X)$, fitting in a commutative diagram
\begin{equation*}
\xymatrix{\DW{r}(S) \otimes \mathbf{F}^{(r)} \ar[dr]^-{\Phi_X} \ar[r]^-\Phi& \mathbf{F}^{(r)}(S) \ar[d]^-{\ev}\\
& \mathbf{V}^{\vtaut}(X)
}
\end{equation*}
Here the appropriate version of $\DW{r}(S)$ depends on whether $X$ is (left or right) $S$-strong or $S$-weak. 
\item Assuming that $X$ satisfies the regularity condition~\eqref{E:assumption}, the results of (b-c) remain valid if we replace $\mathbf{V}^{\vtaut}(X)$ with $\mathbf{V}^{\taut}(X)$. 
\end{enumerate}
\end{thmintro}

\medskip

We conjecture that in the case of two-sided Hecke patterns of rank $r$, the action of $\DW{r}(S)$ on $\mathbf{V}^{\vtaut}(X)$ extends to an action on the whole of $\mathbf{V}(X)$. The approach which we take here is restricted to tautological classes. One may hope to apply to this problem the machinery developed by Toda in \cite{Toda}.

\medskip

We provide in \S\ref{sec:Hilbert} and \S\ref{sec:Higgs} some examples of regular two-sided Hecke patterns such as Hilbert schemes of points on $S$ (in which case our results complement those of Lehn~\cite{Lehn}) and stacks of Higgs bundles on a smooth projective curve. The action of the $W$-algebra on the homology of the stack of Higgs bundles is an essential ingredient in the proof of the $P=W$ conjecture which is given in \cite{HMMS}. Other examples include moduli of instantons and moduli stacks of one-dimensional sheaves on $K3$ surfaces, with possible applications to $\chi$-independence problems.

\medskip

The paper is organized as follows.
In \S\ref{sec:cohas} we define various forms of COHAs of sheaves on a smooth surface $S$.
The formulas for the action of length one Hecke correspondences on tautological classes are established in \S\ref{sec:derHecke}. 
We introduce and study deformed $W$-algebras in \S\S\ref{sec:Fock-ops}-\ref{sec:W-open-surf}. 
Theorem~\ref{thmC} is proven in \S\ref{sec:Hecke-patterns}. 
We zoom in on the action of $W$-algebras on Hilbert schemes of points in \S\ref{sec:Hilbert}, which results in a proof of Theorem~\ref{thmB}.
Further examples of Hecke patterns, such as moduli of Higgs bundles, are considered in \S\ref{sec:Higgs}. 
Finally, \S\ref{sec:conjectures} contains some natural conjectures, concerning in particular a possible extension of our results to threefolds.
Although we use the language of derived algebraic geometry, our approach throughout is 'low-tech' as we work with absolute (rather than relative) Borel-Moore homology. We believe that it should be possible to lift our results to the setting of local COHAs in~\cite{DHSM} (i.e. to adequate sheaves on the space $\Sym^\bullet(S)$).

\subsection{Notations}\label{Sec:notations} Throughout the paper, all geometric objects are defined over the base field $\mathbb{C}$.

\smallskip

\noindent
\textit{Stacks.} 
In this paper, a (derived) stack will mean a 1-Artin (derived) stack which is locally a 
quotient stack of finite type. Let $cl: X^{cl}\to X$ be the classical truncation of a derived stack $X$. 
Restriction to $X^{cl}$ will be often indicated by a superscript $(-)^{cl}$.
For instance, for any object $\mathcal{E} \in D(Coh(X))$ we set $\mathcal{E}^{cl}=cl^*\mathcal{E}$. If $\mathcal{E}$ is a perfect complex of finite amplitude on a derived stack $X$ then we define the total space of $\mathcal{E}$ to be $\mathbb{V}(\mathcal{E})=\Spec\Sym(\mathcal{E}^\vee)$. We will make use of a similar notion of projectivization $\mathbb{P}(\mathcal{E})$, studied by Q. Jiang~\cite{QJiang}. 
Note that if $V$ is a finite-dimensional vector space then $\mathbb{P}(V)$ parametrizes hyperplanes of $V$. Unless specifically mentioned, all fiber products and tensor products are derived.

\smallskip

\noindent
\textit{Borel-Moore homology.}
For a stack $X$, there is a well-defined notion of cohomology or Borel-Moore homology with $\Q$-coefficients which we will denote by $H^*(X,\Q)$ and $H_*(X,\Q)$ respectively, see e.g.~\cite{KV}. When the stack $X$ is pure dimensional we usually write $d_X$ for its dimension and 
$[X] \in H_{2d_X}(X, \Q)$ for its fundamental class. 
If $X$ is smooth then there is an isomorphism of vector spaces $H^i(X,\Q) = H_{2d_X-i}(X,\Q)$ such that $c \mapsto c \cap [X]$. 
For a derived stack $X$ there is also a well-behaved notion of cohomology $H^*(X,\Q)$ 
and of Borel-Moore homology $H_*(X,\Q)$, see \cite{Khan}, \cite{Porta-Yu}. 
The push-forward map $cl_*$ yields isomorphisms
$H^*(X,\Q)=H^*(X^{cl},\Q)$ and $H_*(X,\Q)=H_*(X^{cl},\Q)$; we will often identify the two spaces 
without mention. Note, however that some operators on cohomology or Borel-Moore homology do depend on the derived structure. We collect some results of 
that theory in Appendix~\ref{sec:appA}.

\smallskip

\noindent
\textit{Algebras.} The degree of an homogeneous element $a$ of a graded vector space will be 
denoted by $|a|$. When considering superalgebras, we apply the 
rule of sign for the multiplication of tensor products. In particular
 $[-,-]$ is the super-commutator  and $\{-,-\}$ the anti-super-commutator 
\begin {align}\label{supcom}
[a,b]=ab-(-1)^{|a|\cdot|b|}ba
,\quad
\{a,b\}=ab+(-1)^{|a|\cdot|b|}ba.
\end{align}
For each element $a$ we'll abbreviate $$\Ad_a=[a,-].$$

\smallskip

\noindent
\textit{Symmetric functions.} Let $\Lambda$ be the Macdonald ring of symmetric functions, which is given by
$$\Lambda=\Sym(t\Q[t])=\Q[e_1,e_2,\ldots].$$
We will use standard notations for the elements in $\Lambda$, as in~\cite{Macdonald}, and will sometimes denote the unit of $\Lambda$ by $e_0$ or $h_0$. It is convenient to add a formal element $p_0$ of degree $0$.
Let 
\begin{align}\label{Lp}\Lambda'=\Lambda \otimes \Q[p_0]\end{align}
denote the resulting algebra. The specialization maps $\pi_N:\Lambda \to \Q[x_1, x_2, \ldots, x_N]^{\mathfrak{S}_N}$ extend to $\Lambda'$ by setting $\pi_N(p_0)=N$.
We will occasionally use the following shift operation: for $c$ a formal (even) variable, there is an algebra map
\begin{align}\label{eq:shift-map}
    \tau_c : \Lambda' \to \Lambda'[c], \qquad \tau_c F(x_1,x_2,\ldots)=F(x_1+c,x_2+c,\ldots).
\end{align}
For instance, $$\tau_c(p_k)=\sum_{i=0}^k \binom{k}{i} p_i c^{k-i}, \qquad \tau_c(e_k)=\sum_{i = 0}^{k} \binom{p_0-i}{k-i}e_i c^{k-i}$$
for any $k \geq 0$.
For $\mathcal{E}$ a coherent sheaf on a stack $X$ and $f=F(e_1,e_2, \ldots) \in \Lambda$ we define 
\begin{align}\label{f(E)}
f(\mathcal{E})=F(c_1(\mathcal{E}),c_2(\mathcal{E}),\ldots) \in H^*(X,\Q).
\end{align}
We extend this to elements $f\in \Lambda'$ by setting $p_0(\mathcal{E})=\rk(\mathcal{E})$.

\medskip

\section{COHA of zero-dimensional sheaves on a surface}\label{sec:cohas} 

\subsection{The stack of coherent sheaves on a surface $S$}\label{sec:surf0}
Let $S$ be a smooth connected quasi-projective surface. Unless mentioned otherwise, we will make the following assumption:
$$ \framebox[1.1\width]{\text{The\ surface\ $S$\ has\ pure\ cohomology}.}$$
Let $t_1,t_2$ be the Chern roots of $S$. The Chern classes of $S$ are
$c_1=t_1+t_2$, $c_2=t_1t_2$. The Todd class is 
\begin{equation}\label{E:ToddclassS}
\Td_S=t_1t_2/(1-e^{-t_1})(1-e^{-t_2}).
\end{equation}
We may use its graded version
\[
    \Td_S(z) = z^2t_1t_2/(1-e^{-t_1z})(1-e^{-t_2z}) = \sum_{k\geq 0} \Td_S^{(k)} z^k.
\]
Set $s_2=t_1^2+t_1t_2+t_2^2$.
We will also consider the cohomology with compact support $H^*_c(S,\Q)$. 
Recall that there is an algebra morphism $H^*_c(S,\Q) \to H^*(S,\Q)$ and cup product maps $H^i_c(S,\Q) \otimes H^j(S,\Q) \to H^{i+j}_c(S,\Q)$.
We set $K^c_0({S})_\Q = \bigoplus_i H^{2i}_c(S,\Q)$ and denote by 
$$\langle \alpha,\beta \rangle=\int_S \alpha^\vee \cup \beta \cup \Td_S$$
the Mukai pairing on $K^c_0({S})_\Q$, where if $\alpha=\sum_k \alpha_k$ with $\alpha_k \in H^{2k}_c(S,\Q)$ then $\alpha^\vee=\sum_k (-1)^k \alpha_k$.

Taking the generic rank of a coherent sheaf yields a linear map $\rk : K^c_0({S})_\Q \to \Q$. The class of the structure sheaf of a point will be denoted by $\delta$. Note that for any $\alpha$,
\begin{equation}\label{E:eulerdeltaO}
\langle \alpha,\delta \rangle=\langle \delta, \alpha \rangle=\rk(\alpha).
\end{equation}

\smallskip

Let us pick some $\alpha \in K^c_0({S})_{\mathbb{Q}}$. Consider the derived stack $\RCoh_{\alpha}(S)$ parametrizing coherent sheaves $\mathcal{E} \in Coh(S)$ with proper support and Chern character $\alpha$, 
see e.g.~\cite{Porta-Sala}. Its underlying classical stack will be denoted $\Coh_\alpha(S)$. Note that the stack
$\Coh_\alpha(S)$ parametrizes coherent sheaves on $S$ with proper support and Chern character $\alpha$, while $Coh(S)$ parametrizes all coherent sheaves on $S$. When the surface $S$ is understood, we may abbreviate 
$\RCoh_{\alpha}=\RCoh_{\alpha}(S)$ and likewise for $\Coh_\alpha$. 
In addition, when $\alpha=r [\mathcal{O}_{{S}}] + \alpha'$, where $\alpha' \in \bigoplus_{i >0} H_c^{2i}(S,\Q)$, we may write $\RCoh_{r,\alpha'}$ instead of 
$\RCoh_{\alpha}$. Note that $\RCoh_{\alpha}$ is empty for $\rk(\alpha) \neq 0$ if $S$ is not complete. 
The stack $\Coh_{\alpha}$ is singular in general, but $\RCoh_{\alpha}$ is quasi-smooth and of virtual dimension $d_{\alpha}=-\langle \alpha, \alpha \rangle$. Unless $\alpha \in \N\delta$, the stack $\RCoh_{\alpha}$ is of infinite type. However, it may always be covered by open global quotient stacks which are of finite type. 
We say that a coherent sheaf $\mathcal{E}$ on $S$ is \textit{of dimension $\geq d$} if it contains no subsheaf 
with support of dimension strictly less than $d$. 
Let $\RCoh_{\alpha}^{\geq d}$ be the stack parametrizing dimension 
$\geq d$ sheaves in $\RCoh_{\alpha}$. It is open in $\RCoh_{\alpha}$. We denote by $\mathcal{E}_\alpha \in Coh(\RCoh_\alpha \times S)$ the tautological sheaf. Its restriction to $\RCoh^{\geq d} \times S$ will be denoted $\mathcal{E}_\alpha^{\geq d}$.
As all sheaves in $\RCoh_{\alpha}$ are properly supported, we can extend $\mathcal{E}_\alpha$ by zero and consider it as a sheaf on $\RCoh_\alpha \times \overline{S}$.
An analogous remark holds for $\mathcal{E}_\alpha^{\geq d}$.

\medskip

\subsection{The stack of zero-dimensional sheaves}\label{sec:stack0}
For any $d \in \N$, the stack $\RCoh_{d\delta}$ is the derived moduli stack of (zero-dimensional) sheaves on ${S}$ of length $d$. Its underlying classical stack $\Coh_{d\delta}$ is irreducible of dimension $d$ while $\RCoh_{d\delta}$ is of virtual dimension $-\langle d\delta,d\delta \rangle=0$. Let us set $\RCoh^0=\bigsqcup_d \RCoh_{d\delta}$
and
$\Coh^0=\RCoh^{0,cl}$.

\medskip

\noindent
\begin{example}
\leavevmode\nolisttopbreak
\begin{enumerate}[label=$\mathrm{(\alph*)}$,leftmargin=8mm]
\item
 If $S=\mathbb{A}^2$, then $\Coh_{d\delta}$ is the commuting stack $\mathcal{C}_{\mathfrak{gl}_d}= \{(x,y) \in \mathfrak{gl}_d^2\;:\; [x,y]=0\}/GL_d$. 
\item 
When $S=\Tot(\mathcal{L})$ is the total space of a line bundle $\mathcal{L}$ over a smooth curve $C$, $\Coh_{d\delta}$ is the classical stack of $\mathcal{L}$-twisted Higgs sheaves of length $d$, i.e. it parametrizes pairs $(\mathcal{F},\theta)$ with $\mathcal{F}$ a length $d$ torsion sheaf on $C$ and $\theta \in \Hom(\mathcal{F},\mathcal{F}\otimes \mathcal{L})$.
\end{enumerate}
\end{example}

\medskip

Let $\overline S$ be a smooth compactification of $S$.
Let $\Delta_S : S \to  S \times \overline{S}$ be the diagonal map and 
$$\Delta=\Delta_{S*}(1)\in H^*(S\times \oS,\Q)$$ be the class of the diagonal.
Let $\rho \in Coh(B\Gm)$ be the linear character. We define
\begin{align}\label{u}
u = c_1(\rho).
\end{align}
Recall that $ \mathcal{E}_{\delta}\in Coh(S \times \overline{S} \times B\Gm)$.
We have 
\begin{align}\label{Edelta}
\begin{split}
    \Coh_{\delta} & \simeq S \times B\Gm, \\
    H^*(\Coh_{\delta},\Q)&= H^*(S,\Q)[u],\\
    \mathcal{E}_{\delta}&=\Delta_{S*}(\mathcal{O}_{S})\boxtimes \rho,\\ 
 \ch(\mathcal{E}_{\delta})&=\Delta\cup e^u\cup \Td_S^{-1}.
 \end{split}
\end{align}
Since $\Coh_{\delta}$ is smooth, there is an isomorphism
 \begin{equation}
 \begin{split}
 H^i(\Coh_{\delta},\Q) = H_{2-i}(\Coh_{\delta},\Q) 
 ,\quad
   c\mapsto  c \cap [\Coh_{\delta}].
\end{split}
\end{equation}

\smallskip

We may assume that the surface $S$ is acted upon by a torus $T$. In this case there is an induced action of $T$ on the stacks $\RCoh_{d\delta}$, and all homology groups acquire module structure over the ring
$$\mathbf{R}_T=H^*(BT,\Q).$$

\medskip

\subsection{The COHA of zero-dimensional sheaves}\label{sec:defCOHA0}
We now introduce, following \cite[\S 4]{KV}, the cohomological Hall algebra of zero-dimensional sheaves on $S$.
See also~\cite{Zhao} for another construction of this COHA, \cite{Minets} for the case of the cotangent bundle of a curve, and~\cite{SVIHES} for the case of $S=\mathbb{A}^2$.
We consider the $\Z^2$-graded vector space
\[
    \mathbf{H}_0(S)= H_*(\Coh^0,\Q),\qquad \mathbf{H}_0(S)[l,n]=H_{n}(\Coh_{l\delta},\Q).
\]
Let us briefly recall the definition of the COHA product. 
Fix $\alpha=a\delta, \beta=b\delta$ and $\gamma=\alpha+\beta$. 
Let $\widetilde{\RCoh}_{\alpha;\beta}$ be the derived stack parametrizing short exact sequences $0 \to \mathcal{T}'\to \mathcal{T}\to \mathcal{T}'' \to 0$ with $\mathcal{T}, \mathcal{T}', \mathcal{T}''$ respectively in $\RCoh_\gamma, \RCoh_\beta$ and $\RCoh_\alpha$. 
There is a convolution diagram
\begin{equation}\label{E:convdiagram}
\xymatrix{\RCoh_\alpha \times \RCoh_\beta & \widetilde{\RCoh}_{\alpha;\beta} \ar[l]_-{q_{\alpha,\beta}} \ar[r]^-{p_{\alpha,\beta}} &\RCoh_\gamma
}
\end{equation}
in which the maps $p_{\alpha,\beta}$ and $q_{\alpha,\beta}$ assign to the 
sequence $0 \to \mathcal{T}' \to \mathcal{T} \to \mathcal{T}'' \to 0$ the object $\mathcal{T}$ 
and the pair of objects $(\mathcal{T}'', \mathcal{T}')$ respectively. 
The classical truncation of that diagram reads
\begin{equation}\label{truncinddiag}
\xymatrix{\Coh_\alpha \times \Coh_\beta & \widetilde{\Coh}_{\alpha;\beta} \ar[l]_-{q_{\alpha,\beta}^{cl}} \ar[r]^-{p_{\alpha,\beta}^{cl}} &\Coh_\gamma.
}
\end{equation}
The map $p_{\alpha,\beta}$ is proper and representable. The map $q_{\alpha,\beta}$ is neither representable nor smooth, but it is quasi-smooth. More precisely, consider the complex 
$$\mathcal{C}_{\alpha,\beta}=\RHom_{S}(\mathcal{E}_\alpha, \mathcal{E}_\beta)[1]= \mathbf{R}p_{12*}\mathbf{R}\mathcal{H}om(p_{13}^*\mathcal{E}_\alpha, p_{23}^*\mathcal{E}_\beta)[1]$$
of perfect amplitude $[-1,1]$.
Here $p_{ij}$ stands for the projection from $\RCoh_\alpha \times \RCoh_\beta \times S$ to the $i$-th and $j$-th components. There is a canonical isomorphism
of derived stacks over $\RCoh_\alpha \times \RCoh_\beta$
$$\mathbb{V}(\mathcal{C}_{\alpha,\beta}) \simeq \widetilde{\RCoh}_{\alpha,\beta}$$
where $\mathbb{V}(\mathcal{C}_\bullet)$ is the derived stack given by the total space of a complex $\mathcal{C}_\bullet$. We may thus define a virtual pullback morphism
$$q_{\alpha,\beta}^!: H_i(\Coh_\alpha \times\Coh_\beta,\Q) \to H_{i-2\langle \alpha,\beta\rangle}(\widetilde{\Coh}_{\alpha;\beta},\Q).$$
It is useful to rephrase this construction in classical terms. Let us fix an explicit representative of the complex
$\mathcal{C}_{\alpha,\beta}$
$$0\to \mathcal{V}_{-1} \to \mathcal{V}_0 \to \mathcal{V}_1\to 0.$$ 
Let $\mathcal{C}_{\alpha,\beta}^{cl}$ be the restriction of
$\mathcal{C}_{\alpha,\beta}$ to 
$\Coh_\alpha \times \Coh_\beta$.
Let $\tau_{\leq 0}$ and $(-)^{\leq 0}$ be the standard and stupid truncations.
By \cite{KV}, there is an isomorphism
$$\mathbb{V}(\tau_{\leq 0}(\mathcal{C}^{cl}_{\alpha,\beta}))= \widetilde{\Coh}_{\alpha,\beta}.$$
This yields a factorization
$$\xymatrix{\mathbb{V}(\mathcal{C}^{cl,\leq 0}_{\alpha,\beta}) \ar[d]_{\pi}& \mathbb{V}(\tau_{\leq 0}\mathcal{C}^{cl}_{\alpha,\beta}) \ar[dl]^{q_{\alpha,\beta}} \simeq  \widetilde{\Coh}_{\alpha;\beta}\ar[l]_-{\iota}  \\
\Coh_{\alpha}\times \Coh_{\beta}  & 
}$$
The map $\pi$ is a linear stack, in particular it is smooth.
Hence it yields an isomorphism 
$$\pi^*:H_{i}(\Coh_\alpha \times \Coh_\beta,\Q) \to H_{i+2d_0}(\mathbb{V}(\mathcal{C}^{cl,\leq 0}_{\alpha,\beta}),\Q).$$ 
Further, there is a refined Gysin pullback 
$$\iota^! : H_i(\mathbb{V}(\mathcal{C}^{cl,\leq 0}_{\alpha,\beta}),\Q) \to H_{i-2d_1}(\widetilde{\Coh}_{\alpha,\beta},\Q).$$ 
Here $d_0$ and $d_1$ are the ranks of $\mathcal{C}^{\leq 0}_{\alpha,\beta}$ and $\mathcal{V}_1$, so $d_0-d_1=-\langle \alpha,\beta\rangle$ is the virtual rank of $\mathcal{C}^{cl}_{\alpha,\beta}$. We have $q_{\alpha,\beta}^!= \iota^!\circ \pi^*$, see \cite{Khan}. In particular, the morphism $\iota^!\circ \pi^*$ thus defined is independent of the presentation of the complex $\mathcal{C}$. See \cite[\S~3]{KV} for a direct proof.
We set
\begin{equation*}
\begin{split}
    \star= (p_{\alpha,\beta})_*\circ q^!_{\alpha,\beta} : H_*(\Coh_\alpha,\Q) \otimes H_*(\Coh_\beta,\Q)&=
H_*(\Coh_\alpha \times\Coh_\beta,\Q)\\
&\to H_{* -2\langle \alpha,\beta\rangle}(\Coh_{\gamma},\Q).
\end{split}
\end{equation*}

\smallskip

\begin{theorem}[{\cite[thm.~4.4.2]{KV}}, \cite{Zhao}] The convolution product 
$\star$ defines on $\mathbf{H}_0(S)$ a structure of a graded associative algebra. 
\end{theorem}

\smallskip

\begin{remark} Note that $\langle \delta,\delta\rangle=0$, hence the product in $\Hb_0(S)$ is degree-preserving. 
\end{remark}

\smallskip

In the presence of a torus $T$, we may consider the $T$-equivariant COHA which is an algebra with underlying vector space given by
\[
    \Hb_0^T(S)=H^T_*(\Coh^0,\Q).
\]

An open inclusion $i:S \to S'$ of smooth surfaces gives an open inclusion of derived stacks $\underline{i}:\RCoh^0(S) \to \RCoh^0(S')$ and thus a map $\underline{i}^* : \Hb_0(S') \to \Hb_0(S)$.

\begin{lemma}\label{L:resS'toS}The map $\underline{i}^*$ is an algebra homomorphism.
\end{lemma}

\begin{proof}
The stack $\RCoh_{d\delta}(S)$ is an open substack of $\RCoh_{d\delta}(S')$ for any $d$, and the convolution diagram used to define the product is compatible with open base change. 
\end{proof}

\smallskip

We will need the following result on the Hilbert series of $\mathbf{H}_0(S)$. We define  $$h_{\Hb_0(S)}(z,w)=\sum_{l,n} \dim(\Hb_0(S)[l,n])(-z)^nw^l.$$ Let $h_S(z)=\sum_n \dim(H_n(S,\Q)) (-z)^n$ be the Borel-Moore homology Poincaré polynomial of $S$.

\begin{theorem}[{\cite[thm.~7.1.6]{KV}}]\label{T:KV} Let $q,t$ be formal variables of respective degrees $[0,-2]$ and $[1,0]$.
There is a canonical isomorphism of graded vector spaces 
$$\Hb_0(S) = \Sym\left(H_*(S \times B\Gm,\Q) \otimes qt\Q[t])\right)=\Sym\left(H_*(S,\Q) \otimes qt\Q[q,t]\right).$$
In particular, the Hilbert series of $\Hb_0(S)$ is given by 
$$h_{\Hb_0(S)}(z,w)=
\Exp\left( \frac{h_S(z)z^{-2}w}{(1-z^{-2})(1-w)}\right)$$
where $\Exp$ is the plethystic exponential.
\end{theorem}

Theorem~\ref{T:KV} is proved for an arbitrary smooth surface $S$ using factorization homology techniques, which do not extend to the $T$-equivariant setting.
However, in~\cite{DHSM2} the question of equivariant formality is treated in the much greater generality of relative COHAs, which includes our COHAs by~\cite[\S 11.1]{DHSM}.
In particular, as soon as $S$ is $T$-equivariantly formal, $\Hb_0(S)$ is a free $\mathbf{R}_T$-module of (graded) rank given by $h_{\Hb_0(S)}(z,w)$ by~\cite[Theorems 11.5, 11.6]{DHSM2}.
By~\cite[Theorem 14.1]{GKM} this assumption is satisfied for $S$ cohomologically pure.

\medskip

\subsection{The compactly supported COHA of zero dimensional sheaves} 
When $S$ is not proper, we will also consider a variant of $\Hb_0(S)$ defined
using the hyperbolic Borel-Moore homology introduced in Appendix~\ref{sec:relBMhomology}. 
More precisely, set $\Sym(S)=\bigsqcup_n \Sym^n(S)$. Let 
$$\xymatrix{\RCoh^0 \ar[r]^-\supp& \Sym(S)\ar[r]^-\pi&\mathrm{pt}}$$
be the support and the projection to a point. We set 
\begin{gather*}
    \Hb_0^{c}(S)=\bigoplus_d H^c_*(\Coh_{d\delta},\Q),
\end{gather*}
where $H^c_*(\Coh_{d\delta},\Q)=H_*(\RCoh_{d\delta}/\Sym^d(S),\Q)$
and $H_*(\RCoh_{d\delta}/\Sym^d(S),\Q)$ is defined by
\begin{gather*}
H_*(\RCoh_{d\delta}/\Sym^d(S),\Q)=H^{-*}(\pi_!\supp_*\mathbb{D}_{\RCoh_{d\delta}}).
\end{gather*}
The map $H^*_c(S,\Q)[u] \to H^c_{*}(\Coh_\delta,\Q)$
given by $x \mapsto x \cap [\Coh_\delta]$ is an isomorphism.  
 We complete the induction diagram by introducing the support maps
\begin{equation}\label{E:convdiagramcompact}
\begin{split}
\xymatrix{\RCoh_{m\delta} \times \RCoh_{n\delta}\ar[d]_-{\supp} & \widetilde{\RCoh}_{m\delta;n\delta} \ar[l]_-{q_{m,n}} \ar[r]^-{p_{m,n}} &\RCoh_{(m+n)\delta}\ar[d]_-{\supp}\\
\Sym^{m}(S) \times \Sym^{n}(S) \ar[rr]^-{\oplus}&& \Sym^{m+n}(S)
}
\end{split}
\end{equation}
where $\oplus$ is the direct sum (a finite map), which allows us to view both $\RCoh_{m\delta} \times \RCoh_{n\delta}$ and $\widetilde{\RCoh}_{m\delta;n\delta}$ as derived stacks over $\Sym^{m+n}(S)$. 
Note that 
\begin{equation*}
 \begin{split}
     H_*(\Coh_{m\delta}& \times \Coh_{n\delta} / \Sym^{m+n}(S),\Q) \\
     & = H_*(\Coh_{m\delta} \times \Coh_{n\delta} / \Sym^{m}(S)\times \Sym^{n}(S),\Q)\\
     & = H_*(\Coh_{m\delta}/\Sym^m(S),\Q)\otimes H_*(\Coh_{n\delta}/\Sym^n(S),\Q);
 \end{split}   
\end{equation*}
here the first equality follows from~\eqref{eq:hyp-change-of-S} and properness of $\oplus$.
We may now define the convolution product
\begin{equation*}
 \begin{split}
\star= (p_{m,n})_!\circ q_{m,n}^*: 
H^c_*(\Coh_{m\delta},\Q) \otimes H^c_*(\Coh_{n\delta},\Q) &=
H^c_*(\Coh_{m\delta}\times\Coh_{n\delta},\Q)\\
&\to H^c_*(\Coh_{(m+n)\delta},\Q).
  \end{split}   
\end{equation*}

\smallskip

\begin{proposition} The convolution product $\star$ endows $\Hb_0^{c}(S)$ with the structure of a graded associative algebra. 
\end{proposition}

\begin{proof} The proof is in all points analogous to the case of the $\Hb_0(S)$. Alternatively, one can observe that $\Hb_0^c(S)$ is obtained by taking compactly supported cohomology of sheaf-theoretical COHA $\supp_{*}\mathbb{D}_{\Coh^0}$ on $\Sym(S)$, see~\cite{DHSM}. 
\end{proof}

If the surface $S$ is projective, then we have $\Hb_0^c(S)=\Hb_0(S)$. 
For each open embedding $i: S \to S'$ we have an open embedding
of derived stacks $\underline{i}:\RCoh_{n\delta}\to\RCoh_{n\delta}(S')$. 
Hence, there are pushforward maps $\underline{i}_!:H_*^c(\RCoh_{n\delta},\Q) \to H_*^c(\RCoh_{n\delta}(S'),\Q)$, which combine to $\underline{i}_!:\Hb_0^c(S) \to \Hb^c_0(S')$.

\smallskip

\begin{lemma}\label{L:embStoS'} 
The map $\underline{i}_!:\Hb_0^c(S) \to \Hb^c_0(S')$ is an algebra morphism.
\end{lemma}
\begin{proof}
Quasi-smooth pullback and proper pushforward in hyperbolic 
homology are compatible with open base change, see Proposition~\ref{P:basechange}. 
\end{proof}

If $\iota: S \to \oS$ is a smooth compactification 
of $S$, Lemmas \ref{L:resS'toS}, \ref{L:embStoS'} yield algebra homomorphisms
\begin{equation}\label{E:relCOHAvsCOHAc}
\xymatrix{\Hb_0^c(S) \ar[r]^-{\underline{\iota}_!} & \Hb^c_0(\oS) = \Hb_0(\oS) \ar[r]^-{\underline{\iota}^*} &\Hb_0(S).}
\end{equation}
The composed map 
$$\phi_S=\underline{\iota}^*\underline{\iota}_! : \Hb_0^c(S) \to \Hb_0(S)$$ is independent of the choice of $\oS$. 
Let 
$$h_S^c(z)=\sum_n \dim(H_n^c(S,\Q)) (-z)^n$$ be the homology Poincaré polynomial of $S$.
We will need the following variant of Theorem~\ref{T:KV}, which can be found in~\cite[Corollary 7.11]{DHSM2}.

\smallskip

\begin{theorem}\label{T:KVcompact}
There is a canonical isomorphism of graded vector spaces 
$$\Hb^c_0(S) = \Sym\left(H^c_*(S \times B\Gm,\Q) \otimes qt\Q[t])\right)=\Sym\left(H^c_*(S,\Q) \otimes qt\Q[q,t]\right).$$
In particular, the Hilbert series of $\Hb^c_0(S)$ is given by 
$$h_{\Hb^c_0(S)}(z,w)=\Exp\left(\frac{h^c_S(z)z^{-2}w}{(1-z^{-2})(1-w)}\right).$$
\end{theorem}

As explained after Theorem~\ref{T:KV}, a $T$-equivariant version of Theorem~\ref{T:KVcompact} follows from~\cite[\S 11]{DHSM2} provided that $S$ is pure.

\medskip

\subsection{The COHA of properly supported sheaves}\label{sec:cohaprop}

\smallskip

 Following \cite[\S 4]{KV}, one can extend the construction of COHA product to the stack of properly supported sheaves on $S$. We set
 $$\Hb(S)= \bigoplus_{\alpha} H_*(\Coh_{\alpha},\Q).$$
 As in the case of zero-dimensional sheaves, 
 for any $\alpha,\beta \in K^c_0(S)_\Q$ there is an induction diagram 
 \begin{equation}\label{E:convdiagram2}
    \RCoh_\alpha \times \RCoh_\beta \xleftarrow{q_{\alpha,\beta}} \widetilde{\RCoh}_{\alpha;\beta} \xrightarrow{p_{\alpha,\beta}}\RCoh_\gamma,\quad \gamma=\alpha+\beta
 \end{equation}
with $q_{\alpha,\beta}$ quasi-smooth and $p_{\alpha,\beta}$ proper. 

\smallskip

\begin{theorem}[{\cite[thm.~4.4.2]{KV}}]\label{thm:fullCOHA-KV} Convolution with respect to the correspondences \eqref{E:convdiagram2} for all $\alpha,\beta \in K^c_0({S})_\Q$ endows the space $\Hb(S)$ 
with the structure of a graded associative algebra. The multiplication $H_*(\Coh_{\alpha},\Q) \otimes  H_*(\Coh_{\beta},\Q) \to  H_*(\Coh_{\gamma},\Q)$ is of homological degree $-2\langle \alpha, \beta \rangle$.
\end{theorem}

\smallskip

Assume that $i: S \to S'$ is an open immersion into another smooth surface $S'$. This gives rise to an open immersion $\underline{i}:\RCoh(S) \to \RCoh(S')$ and hence to a restriction morphism $\underline{i}^* : \Hb(S') \to \Hb(S)$. 
The proof of Lemma~\ref{L:resS'toS} implies that the map 
$\underline{i}^*$ is an algebra homomorphism.

\smallskip

\begin{remark} There is no higher rank analog of $\Hb^c_0(S)$ because in general there is no useful map from $\RCoh$ to a coarse moduli space.
\end{remark}

\medskip

\subsection{Tautological classes}\label{sec:taut} Let us now study a family of 
cohomology classes $ch_i(\lambda)$ on 
each $\Coh_{\alpha}$ for $i \geq 1$ and $\lambda \in H^*(S,\Q)$.
Fix a smooth compactification $\iota: S \to \overline{S}$ of $S$.  
Since $S$ is pure, the restriction $H^*(\overline{S},\Q)\to H^*(S,\Q)$ is surjective.
Let $I(S)\subset H^*(\overline{S},\Q)$ be the kernel of this map. 
It is identified with the relative cohomology $H^*(\oS,S,\Q)$.
The map pushforward $H^*_c(S,\Q)\to H^*(\overline{S},\Q)$ is injective.
The intersection pairing on $H^*(\overline{S},\Q)$ identifies $H^*_c(S,\Q)$ with the orthogonal complement $I(S)^\perp$ in $H^*(\overline S,\Q)$.
This is a (non-unital) subalgebra under the cup product.  Dually, $H^*(S,\Q)$ and $H^*(\overline{S},\Q)$ are equipped with natural coproducts and the restriction $\iota^*$ is a coalgebra homomorphism. 

\smallskip

When $S$ is not proper, it will be convenient to formally add a class $\pt$ of degree $4$ to $H^*(S,\Q)$, satisfying $\pt \cup H^{>0}(S,\Q)=\{0\}$. Likewise, it will be convenient to formally add a unit $1$ to $H^*_c(S,\Q)$. We will modify the coproduct accordingly. Let 
$$\overline{H}^*(S,\Q),\quad\overline{H}^*_c(S,\Q)$$ be the resulting rings.  
If $\Delta'$ is the coproduct on $H^*(S,\Q)$ we define the coproduct of $\overline{H}^*(S,\Q)$ by
\begin{align*}
    \Delta : \overline{H}^*(S,\Q) \to \overline{H}^*(S,\Q) \otimes \overline{H}^*(S,\Q),
    \end{align*}
such that
\begin{align*}
    \Delta(\pt)=\pt \otimes  \pt,\qquad \Delta(\lambda)=\pt \otimes \lambda + \lambda \otimes \pt + \Delta'(\lambda),\quad \lambda \in H^*(S,\Q).
\end{align*}

\smallskip

Let us define a variant of the ring of symmetric function which is colored by $H^*(S,\Q)$. Consider 
\[
    U(S)=\Sym(H^*(S,\Q)\otimes t\Q[t]).  
\]
For each $i \geq 1$ and $\lambda \in H^*(S,\Q)$, we denote $\uch_i(\lambda) = \lambda\otimes t^i$.
Set $\text{deg}(\uch_i(\lambda))=2i + \text{deg}(\lambda)-4$. 
To keep track of the rank of coherent sheaves, we add an extra element $\r$ of degree $0$ and set 
$$U'(S) = U(S) \otimes \Q[\r].$$ 
It is the free graded-commutative algebra generated by 
$\uch_i(\lambda)$ and $\r$ subject to the relations
$$\uch_i(\lambda+\mu)=\uch_i(\lambda) + \uch_i(\mu), \quad \uch_i(a\lambda)=a\,\uch_i(\lambda),$$
for any $a \in \mathbb{Q}$ and $\lambda,\mu \in H^*(S,\Q)$.

\begin{definition}
    Let $\Lambda(S)$ be the quotient of $U'(S)$ by the ideal generated by the negative degree elements $\uch_1(\lambda)$ for $\deg(\lambda)=0,1$.
    The \textit{universal Chern character} $\uch(z)$ is defined as follows:
    \[
        \uch(z)=\r \otimes 1 +\sum_{i\geq 1} \uch_{i} z^i \in \Lambda(S)\otimes \overline{H}^*_c(S,\Q)[[z]], \qquad \uch_i=\sum_\lambda \uch_{i}(\lambda)\otimes \lambda^*,
    \]
    where $\sum \lambda \otimes\lambda^*\in H^*(S,\Q) \otimes H^*_c(S,\Q)$ is the intersection pairing tensor.
\end{definition}

\begin{remark}\label{rmk:taut-class-restr}
To avoid confusions we may write $\uch_{S}$ instead of $\uch$.
    The definitions above are compatible with restriction along $\iota:S\to \oS$.
    Namely, we have a natural quotient $\Lambda(\oS)\twoheadrightarrow \Lambda(S)$, such that the image of $\uch_{\oS}(z)$ in $\Lambda(S)\otimes H^*(\oS,\Q)[[z]]$ is precisely $\uch_S(z)$.
\end{remark}

\begin{remark}
    Note that $\uch(z)$ belongs to $\Q[\r]\otimes 1 + z\Lambda(S)\otimes H^*_c(S,\Q)[[z]]$.
\end{remark}

We define a coalgebra structure on $\Lambda(S)$ by requiring the elements $\uch_i(\lambda)$ and $\r$ to be primitive.
In other words,
$$(\Delta\otimes \Id)(\uch(z))=\uch(z)_{13} + \uch(z)_{23} \in \Lambda(S) \otimes \Lambda(S) \otimes \Hbar^*_c(S,\Q)[[z]].$$

\medskip

\begin{example}
The primitive elements of degree $0$ in $\Lambda(S)$ are linearly spanned by $\r$ and $\uch_2(1)$, $\uch_1(\lambda)$ for $\lambda \in H^2(S,\Q)$.
\end{example}

We also consider an involution
\begin{align}\label{eq:involution}
    \upsilon : \Lambda(S) \to \Lambda(S), \quad \upsilon(\r)=-\r, \quad \upsilon(\uch_i(\lambda))=-\uch_i(\lambda),
    \quad i \geq 1,\quad\lambda \in H^*(S,\Q).
\end{align}

\medskip

The elements $\uch_i$, being even, commute with each other. Recall the algebra $\Lambda'$ from \eqref{Lp}.
We may define an algebra morphism 
$$p: \Lambda' \to \Lambda(S) \otimes \overline{H}^*_c(S,\Q), \quad p_0 \mapsto \r \otimes 1, \qquad p_i/i! \mapsto \uch_i,\quad i \geq 1.$$
We will use the following notation  
\begin{align}\label{flambda}
f(\lambda)=\int_S p(f) \cup \lambda\in\Lambda(S),\quad
f \in \Lambda',\quad
\lambda \in \Hbar^*(S,\Q),
\end{align}
where the map
$$\int_S:\overline H^*_c(S,\Q)\to\Q$$ is defined in the obvious way.
We may simply write 
$f(\lambda)=\int_S f \lambda$ when there is no risk of confusion.
For instance, we have 
$p_i(\lambda)=i! \uch_i(\lambda)$. Observe that 
\begin{equation}\label{E:fgdelta}
(f \cdot g)(\lambda)=(f \otimes g)(\Delta(\lambda))=\sum f(\lambda^{(1)})g(\lambda^{(2)})
\end{equation}
for $f,g \in \Lambda'$ and $\lambda \in \Hbar^*(S,\Q)$, where $\Delta(\lambda)=\sum \lambda^{(1)} \otimes \lambda^{(2)}$ in Sweedler's notation. In particular,
$$(p_{i_1} \cdots p_{i_l})(\lambda)=i_1! \cdots i_l!\sum \uch_{i_1}(\lambda^{(1)}) \cdots \uch_{i_l}(\lambda^{(l)}).$$
We have $\text{deg}(f(\lambda))=2\text{deg}(f) + \deg(\lambda)-4$.

\medskip

\begin{remark}\label{rmk:unit-and-point}
    Note that with our conventions we have 
    $$h_0(\pt)=1(\pt) = 1 \in \Lambda(S), \qquad  p_0(\pt) =  \r \otimes 1(\pt) =\r \in \Lambda(S)$$
    while $1(\lambda)=p_0(\lambda)=0$ if $\deg \lambda <4$, regardless of whether $S$ is proper or not.
    In a similar vein, we have $f(\pt)=0$ for any $f$ in the augmentation ideal of $\Q[p_1, p_2, \ldots]$ when $S$ is not proper.
\end{remark}

\medskip

Now fix $\alpha \in K_0^c(\oS)_\Q$ a class of rank $r$ and consider a locally closed substack $\mathcal{U}_\alpha\subset\Coh_\alpha(\oS)$.
Let $\mathcal{E}_{\alpha} \in Coh(\mathcal{U}_{\alpha}\times \overline{S})$ 
denote the restriction of the tautological sheaf to 
$\Coh_{\alpha}(\overline{S}) \times \overline{S}$. 
Consider its Chern character
\[
    \ch(\mathcal{E}_\alpha,z) = r + \sum_{i\geq 1}p_i(\mathcal{E}_\alpha)z^i/i!
\]
in 
$$H^*(\mathcal{U}_{\alpha}\times \overline{S},\Q)=H^*(\mathcal{U}_{\alpha},\Q)\otimes H^*(\overline{S},\Q)$$
where $p_i(\mathcal{E}_\alpha)$ is as in \eqref{f(E)}.
We have a unique graded ring homomorphism 
\begin{align}\label{eval}
\ev_{\alpha} : \Lambda({\oS}) \to H^*(\mathcal{U}_{\alpha},\Q),
\end{align}
defined by
$$(\ev_{\alpha} \otimes \Id)(\uch(z))=\ch(\mathcal{E}_{\alpha},z) \in H^*(\mathcal{U}_{\alpha},\Q) \otimes H^*(\oS,\Q)[[z]].$$
Observe that $r = \ev_\alpha(\r)$.
The following lemma is a straightforward corollary of Remark~\ref{rmk:taut-class-restr}.

\begin{lemma}\label{lem:eval-open-restriction}
 Assume that the Chern character of $\mathcal{E}_\alpha$ belongs to 
$\Q\cdot 1+H^*(\mathcal{U}_\alpha,\Q) \otimes H^*_c(S,\Q)$.
The evaluation map $\ev_{\alpha}$ factors through $\Lambda(S)$.
The classes $\ev_\alpha(f(\lambda))$ for $f \in \Lambda'$ and $\lambda \in \overline H^*(S,\Q)$ are independent of 
the choice of compactification $\overline{S}$.\qed
\end{lemma}

The condition of Lemma~\ref{lem:eval-open-restriction} is verified if $\mathcal{U}_\alpha\subset \Coh_{\alpha}$.
It also holds if the restriction of $\mathcal{E}_\alpha$ to $\mathcal{U}_\alpha \times (\oS\setminus S)$ is a trivial 
vector bundle. Such situations occur when considering moduli stacks of sheaves on $\oS$ which are trivialized along $\oS\setminus S$.
From now on, unless specified otherwise, we will assume that $\mathcal{U}_\alpha=\Coh_\alpha$. 
Then, the map $\ev_{\alpha}$ is a ring homomorphism
$\Lambda({\oS}) \to H^*(\Coh_{\alpha},\Q)$ which factors through $\Lambda(S) \to H^*(\Coh_{\alpha},\Q).$

\medskip

\subsection{Extended COHAs}
Assume now that $\alpha \in \N \delta$.
Composing the map $\ev_{\alpha}$ above with the cap product yields an action $\bullet$ of $\Lambda({S})$ on $\mathbf{H}_0(S)$ such that
$$
x \bullet c=\ev_{\alpha}(x) \cap c
,\quad
c \in H_*(\Coh_{\alpha},\Q)
,\quad
x \in \Lambda({S}).
$$
This action preserves each $H_*(\Coh_{\alpha},\Q)$ and is compatible with the (co)homological gradings, i.e. $\text{deg}(x\bullet c)=\text{deg}(c)-\text{deg}(x)$.
The same holds for $\Hb^c_0(S)$, where the action of $H^*(\Coh_\alpha,\Q)$ on $H^c_*(\Coh_\alpha,\Q)$ is given by Lemma~\ref{lem:actioncohrelBM}.

\medskip

\begin{example}\label{ex:p1pt} Assuming that $S$ is proper, let us compute $\ev_{n\delta}(p_1(\pt))$. 
By definition, we have
$$\sum_\lambda\ev_{n\delta}(p_1(\lambda))\otimes\lambda^*=
c_1(\mathcal{E}_{n\delta}).$$
Hence $\ev_{n\delta}(p_1(\pt))=c_1(i_x^*(\mathcal{E}_{n\delta}))$ in $H^2(\Coh_{n\delta},\Q)$ for any closed point $i_x: \Spec(\C) \to S$. The support of $i_x^*(\mathcal{E}_{n\delta})$ being of codimension $2$, its first Chern class vanishes, hence $\ev_{n\delta}(p_1(\pt))=0$ for any $n>0$.
    
\end{example}
\medskip

\begin{proposition}\label{prop:modulealgebra}
The ring $\mathbf{H}_0(S)$ is a $\Lambda({S})$-module algebra, 
i.e. we have
\begin{equation}\label{E:modulealgebra}
x \bullet (c_1 \star c_2)=\sum (-1)^{|c_1|\cdot |x^{(2)}_i|}(x^{(1)}_i \bullet c_1) \star (x^{(2)}_i \bullet c_2)
,\quad
x \in \Lambda(S)
,\quad
c_1,c_2 \in \Hb_0(S),
\end{equation}
where $\Delta(x) = \sum x^{(1)}_i\otimes x^{(2)}_i$. The same holds for $\Hb^c_0(S)$. The map $\phi_S: \Hb^c_0(S) \to \Hb_0(S)$ is a 
morphism of $\Lambda(S)$-modules.
\end{proposition}

\begin{proof} We will deal with the case of $\Hb_0(S)$, the other one is similar. We can assume that 
$c_i \in H_*(\Coh_{\alpha_i},\Q)$ for $i=1,2$. Set $\gamma=\alpha_1+\alpha_2$. Recall the induction diagram \eqref{E:convdiagram}. We abbreviate $p=p_{\alpha_1,\alpha_2}$ and $q=q_{\alpha_1,\alpha_2}$. By the projection formula
$$x \bullet (c_1 \star c_2)=\ev_{\gamma}(x) \cap p_*q^!(c_1 \otimes c_2)=p_*(p^*(\ev_{\gamma}(x) ) \cap q^!(c_1 \otimes c_2)).$$
There is a short exact sequence of tautological sheaves
$$0 \to q^*(\mathcal{E}_{\alpha_2}) \to p^*(\mathcal{E}_{\gamma}) \to q^*(\mathcal{E}_{\alpha_1})\to 0$$
Hence $p^*(\ch(\mathcal{E}_{\gamma}))=q^*(\ch(\mathcal{E}_{\alpha_1}) + \ch(\mathcal{E}_{\alpha_2}))$.
We deduce that
\begin{equation*}
\begin{split}
p_*(p^*(\ev_{\gamma}(x) ) \cap q^!(c_1 \otimes c_2))&=p_*q^!\left( (\ev_{\alpha_1} \otimes \ev_{\alpha_2})(\Delta(x)) \cap (c_1 \otimes c_2)\right)\\
&=\sum(-1)^{|c_1|\cdot |x^{(2)}_i|} (x^{(1)}_i \bullet c_1) \star (x^{(2)}_i \bullet c_2).
\end{split}
\end{equation*}
Note that we used Proposition~\ref{Prop:appA}(c) in the first line.
The compatibility between $\phi_S$ and the action of $\Lambda(S)$ results from the projection formula in hyperbolic or Borel-Moore homology.
\end{proof}

\smallskip

We define the semi-direct product 
\begin{align}\label{tildeH}\widetilde{\Hb}_0(S)=\Hb_0(S) \rtimes \Lambda(S)
\end{align} 
to be the algebra generated by $\Hb_0(S)$ and $\Lambda(S)$ modulo the relations
\begin{equation}\label{E:defsemidirectprod}
x \cdot c=\sum (-1)^{|c|\cdot |x^{(2)}_i|} (x^{(1)}_i \bullet c) \cdot x^{(2)}_i ,\quad x \in \Lambda(S),\; c \in \Hb_0(S).
\end{equation}
The multiplication map $\Lambda(S) \otimes \Hb_0(S) \to \widetilde{\Hb}_0(S)$ is an isomorphism of graded vector spaces. 
We define similarly the semi-direct product 
\begin{align}\label{tildeHc}\widetilde{\Hb}^c_0(S)=\Hb^c_0(S) \rtimes \Lambda(S).
\end{align}

\smallskip

We finish with the following observation. The degree one piece of $\Hb_0(S)$ is 
$$\Hb_0(S)[1,-]=H_*(\Coh_\delta,\Q)=H_*(S \times B\Gm,\Q)=H_*(S,\Q)[u].$$
We consider the linear map
\begin{align}\label{omegadelta}\omega_\delta : \Lambda({S}) \to 
H_*(\Coh_\delta,\Q)
,\quad 
x \mapsto x \bullet [\Coh_\delta]
\end{align}

\begin{lemma}\label{lm:omega-surj}
     The map $\omega_\delta$ is surjective.
\end{lemma}

\begin{proof}
The restriction $H_*(\overline{S},\Q) \to H_*(S,\Q)$ is surjective.
So the restriction $H_*(\Coh_\delta(\overline{S}),\Q) \to H_*(\Coh_\delta,\Q)$ is also surjective. 
As the latter is a morphism of $\Lambda(\overline{S})$-module, it is enough to prove the statement for $S$ projective. 
By \eqref{Edelta}, we have
$$\ch(\mathcal{E}_\delta)=\Td_S^{-1}e^u\Delta_S=\Td_S^{-1}e^u \sum_{\lambda} \lambda \otimes \lambda^*$$
where $\{\lambda\},$ $\{\lambda^*\}$ are dual bases
of $H^*(S,\Q)$ and $H^*_c(S,\Q)$. Since $\Td_S$ is invertible, the result follows.
\end{proof}

When $S$ is not pure, the map $\ev_\delta$ may still be defined but 
it cannot be surjective since $\mathcal{E}_\delta$ extends to the 
compactification $\Coh_\delta(\overline{S}) \times \overline{S}$.

\begin{remark}
   The definition of the action of $\Lambda(S)$ on $\Hb_0(S)$ as well as
Proposition~\ref{prop:modulealgebra} and its proof extend \textit{mutatis mutandis} from $\Hb_0(S)$ to $\Hb(S)$. 
 \end{remark}

\medskip

\section{Derived Hecke correspondences}\label{sec:derHecke}

In this section we consider and describe the simplest type of Hecke correspondence.

\subsection{Hecke correspondences} From \eqref{E:convdiagram2}, 
we can derive the following induction diagrams
\begin{equation}\label{E:convdiagram3}
    \begin{gathered}
        \RCoh_{n\delta} \times \RCoh_\alpha\xleftarrow{q_{n\delta,\alpha}} \widetilde{\RCoh}_{n\delta;\alpha} \xrightarrow{p_{n\delta,\alpha}}\RCoh_{\alpha+n\delta},\\
        \RCoh_{\alpha} \times \RCoh_{n\delta}\xleftarrow{\overline{q}_{\alpha,n\delta}} \widetilde{\RCoh}_{n\delta;\alpha-n\delta}\xrightarrow{\overline{p}_{\alpha,n\delta}}  \RCoh_{\alpha-n\delta}.
    \end{gathered}
\end{equation}
For the compactly supported COHA it is useful to factor the maps $p_{n\delta,\alpha},$ $ \overline{p}_{\alpha,n\delta}$ as follows
\begin{equation}\label{E:convdiagram3'}
\begin{gathered}
    \widetilde{\RCoh}_{n\delta;\alpha}\xrightarrow{p'_{n\delta,\alpha}} \RCoh_{\alpha+n\delta} \times \Sym^n(S)\xrightarrow{p''_{n\delta,\alpha}} \RCoh_{\alpha+n\delta},\\
    \widetilde{\RCoh}_{n\delta;\alpha-n\delta} \xrightarrow{\overline{p}'_{\alpha,n\delta}} \RCoh_{\alpha-n\delta} \times \Sym^n(S) \xrightarrow{\overline{p}''_{\alpha,n\delta}} \RCoh_{\alpha-n\delta}.
\end{gathered}
\end{equation}
We call the first/second diagram in~\eqref{E:convdiagram3}
the positive/negative length $n$ Hecke correspondences. 
Since $\RCoh^{\geq d}$ is stable under taking subobjects, the Hecke correspondences restrict to $\RCoh^{\geq d}$,
yielding the following restricted induction diagrams
 \begin{equation}\label{E:truncconvdiagram3}
 \begin{gathered}
\xymatrix{\RCoh_{n\delta} \times \RCoh_\alpha^{\geq d} & \widetilde{\RCoh}_{n\delta;\alpha}^{\geq d} \ar[l]_-{q_{n\delta,\alpha}} \ar[r]^-{p_{n\delta,\alpha}} &\RCoh_{\alpha+n\delta}^{\geq d},
}\\
\xymatrix{\RCoh_{\alpha}^{\geq d} \times \RCoh_{n\delta} & \widetilde{\RCoh}_{n\delta;\alpha-n\delta}^{\geq d} \ar[l]_-{\overline{q}_{\alpha,n\delta}} \ar[r]^-{\overline{p}_{\alpha,n\delta}} &\RCoh_{\alpha-n\delta}^{\geq d},
}
\end{gathered}
\end{equation}
where 
$$\widetilde{\RCoh}^{\geq d}_{\alpha;\beta}=\widetilde{\RCoh}_{\alpha;\beta} \underset{\RCoh_{\alpha+\beta}}{\times} 
\RCoh^{\geq d}_{\alpha+\beta}$$
is the derived fiber product.

\subsection{Locally free resolutions} The Hecke correspondences enjoy much better properties when the tautological sheaf $\mathcal{E}$ has perfect amplitude in $[-1,0]$ and admits locally a two-step locally free resolution. This is true in our situation after we restrict to the open substack $\RCoh^{\geq 1}$.

\medskip

\begin{lemma} Let $\gamma \in K^c_0({S})_{\Q}$ and let $\mathfrak{U} \subset \RCoh^{\geq 1}_{\gamma}$ be any finite type open substack. The tautological sheaf $\mathcal{E}_{\gamma} |_{ \mathfrak{U}\times\overline S} $ admits a $2$-step resolution by locally free sheaves $0\to \mathcal{E}_{-1} \to \mathcal{E}_0 \to \mathcal{E}_{\gamma} \to 0$.
\end{lemma}

\begin{proof} Having a $2$-step resolution by locally free sheaves is a local condition on any stack of finite type. Indeed, let $\mathcal{F} \in Coh(X)$, where $X$ is a stack of finite type, and assume that $\mathcal{F}|_{U_i}$ admits $2$-step resolutions by locally free sheaves for some open cover $\bigcup_i U_i = X$. Since $X$ is of finite type, there exists a short exact sequence $$0 \to \mathcal{E}_{-1} \to \mathcal{E}_0 \to \mathcal{F}\to 0$$ with $\mathcal{E}_0$ being a locally free sheaf (see \cite[Prop. 2.1.10]{Huybrechts-Lehn}). We claim that $\mathcal{E}_{-1}$ is locally free. This may be done locally and on any atlas $[U/R] \to X$; it is enough to check that the stalk $(\mathcal{E}_{-1})_{|U,u}$ is a flat $\mathcal{O}_{U,u}$-module for any $u \in U$; this follows from the fact that $\text{Tor}^{>1}(\mathcal{F}_{|U},-)=\{0\}$ and the standard long exact sequence. Returning to the setting of the lemma, as 
$\mathcal{E}_{\gamma} |_{ \mathfrak{U}\times\overline S}$ is $\mathfrak{U}$-flat, it is enough to check that for any $\mathbb{C}$-point $x\in\mathfrak{U}(\mathbb{C})$ the sheaf $\mathcal{E}_{\gamma}|_{\{x\}\times\overline S}$ admits a $2$-step resolution by locally free sheaves over $\overline S$, see \cite[Lemma 2.1.7]{Huybrechts-Lehn}. This in turn follows from the fact that $\mathcal{E}_{\gamma}|_{\{x\}\times\overline S}$ is of dimension $\geq 1$ and that $\overline{S}$ is smooth, so that $\mathcal{E}_{\gamma}|_{\{x\}\times\overline S}$ has perfect (equivalently, Tor) amplitude in $[-1,0]$ (see e.g. \cite[Part II., 5.]{Huybrechts-Lehn}).
\end{proof}

\medskip

In the remainder of this section, we describe the length one Hecke correspondences and compute their action on tautological classes. Until \S\ref{sec:atlashecke}, we let $S$ be an arbitrary smooth connected surface. A general framework for derived Hecke correspondences has recently been worked out by Q. Jiang in \cite[\S 8]{QJiang}, in the language of derived algebraic geometry, following the work of Negut \cite{Negut}.

\medskip

\subsection{Length one Hecke correspondences and operators}\label{sec:derived Hecke}
 Fix $\alpha$ and set $\gamma=\alpha+\delta$.  
 In this section we consider length one Hecke correspondences given by the diagrams in \eqref{E:convdiagram3} with $n=1$ restricted to $\RCoh^{\geq 1}$. 
 To unburden the notation, we will drop the indices 
 of the maps $p$, $q$, etc. 
 Let $K_{S}$ be the canonical bundle of $S$ and set $\mathcal{F}_\alpha=\mathcal{E}_{\alpha}^\vee \otimes K_{S}[1]$, a complex over $\RCoh_\alpha \times S$. 
 By~\eqref{Edelta}, the Serre duality gives an isomorphism of complexes over $\RCoh_\alpha \times \RCoh_{\delta}$
\begin{equation*}
\begin{split}
(\RHom_S(\mathcal{E}_\delta, \mathcal{E}_{\alpha})[1])^\vee &= \RHom_S(\mathcal{E}_\alpha, K_{S}\otimes\mathcal{E}_\delta)[1]\\
&=Rp_{12*}(\mathcal{E}_{\alpha}^\vee \otimes K_{S}\otimes \mathcal{O}_{\Delta_{S}} \otimes \rho)[1]\\
&=\mathcal{E}_{\alpha}^\vee\otimes K_S\otimes\rho[1]
\end{split}
\end{equation*}
where $p_{12}:\RCoh_\alpha \times \RCoh_{\delta} \times S \to \RCoh_\alpha \times \RCoh_{\delta}$ is the projection. In particular, the complex $\mathcal{F}_\alpha$ is the restriction to $\RCoh_\alpha \times S$ of $(\RHom_S(\mathcal{E}_\delta, \mathcal{E}_{\alpha})[1])^\vee$.
Let 
$$\tau : \mathbb{P}(\mathcal{E}_{\gamma}) \to \mathbb{P}(\mathcal{E}_{\gamma})  \times {S}
,\quad
\overline{\tau} : \mathbb{P}(\mathcal{F}_{\alpha}) \to \mathbb{P}(\mathcal{F}_{\alpha}) \times {S}$$ 
be the diagonal morphisms, i.e., the morphisms making the following diagrams commutative:
\begin{equation}\label{E:deftau}
    \begin{split}
        \xymatrix{\mathbb{P}(\mathcal{E}_{\gamma}) \ar[r]^-{\tau} \ar[d]_-{\pi} & \mathbb{P}(\mathcal{E}_{\gamma}) \times {S} \ar[d]^-{pr_{{S}}} &  \mathbb{P}(\mathcal{F}_{\alpha}) \ar[r]^-{\overline{\tau}} \ar[d]_-{\overline{\pi}} & \mathbb{P}(\mathcal{F}_{\alpha}) \times {S} \ar[d]^-{pr_{{S}}}\\
\RCoh_{\gamma} \times {S} \ar[r]^-{pr_{{S}}} & {S} & \RCoh_{\alpha} \times {S} \ar[r]^-{pr_{{S}}} & {S}}
    \end{split}
\end{equation}
Here, $pr_{{S}}$ is the projection to ${S}$.
We consider the following partial truncations
$$\widetilde{\Coh}_{\delta;\alpha}= \widetilde{\RCoh}_{\delta;\alpha} \underset{\RCoh_\delta}{\times} \Coh_\delta, \qquad \widetilde{\Coh}^{\geq 1}_{\delta;\alpha}= \widetilde{\RCoh}^{\geq 1}_{\delta;\alpha} \underset{\RCoh_\delta}{\times} \Coh_\delta.$$
The motivation to consider these derived stacks will become clear in~\S\ref{sec:atlashecke}. 
The notation is in conflict with~\eqref{truncinddiag}, but its usage will be clear from the context.
We have the following important result.

\smallskip

\begin{proposition}[{\cite[\S 8\,,\,prop.~4.33]{QJiang},\cite[\S 2]{Negut}}]\label{L:Negut} 
\leavevmode\nolisttopbreak
\begin{enumerate}[label=$\mathrm{(\alph*)}$,leftmargin=8mm]
\item
There is an isomorphism of derived stacks 
$\widetilde{\Coh}^{\geq 1}_{\delta;\alpha} = \mathbb{P}(\mathcal{E}_{\gamma})$ 
which identifies the tautological sheaf $(q \times \Id)^*(\mathcal{E}_{\delta})$ 
with $\tau_*(\mathcal{O}_{\mathbb{P}(\mathcal{E}_{\gamma})}(1))$.
\item
There is an isomorphism of derived stacks 
$\widetilde{\Coh}^{\geq 1}_{\delta;\alpha} = \mathbb{P}(\mathcal{F}_{\alpha})$ 
which identifies the tautological sheaf $(\overline{q} \times \Id)^*(\mathcal{E}_{\delta})$ with 
$\overline{\tau}_*(\mathcal{O}_{\mathbb{P}(\mathcal{F}_{\alpha})}(-1))$.\qed
\end{enumerate}
\end{proposition}

\smallskip

Recall that the classical truncation map $\Coh_\delta \to \RCoh_\delta$ induces an isomorphism in Borel-Moore homology (and likewise in relative Borel-Moore homology). To carry out the computation of length one Hecke operators, we may and will therefore use the partially truncated induction diagrams
\begin{gather*}
    \Coh_{\delta} \times \RCoh_\alpha^{\geq 1} \xleftarrow{q_{\delta,\alpha}} \widetilde{\Coh}_{\delta;\alpha}^{\geq 1} \xrightarrow{p_{\delta,\alpha}} \RCoh_{\gamma}^{\geq 1},\\
    \RCoh_{\gamma}^{\geq 1} \times \Coh_{\delta} \xleftarrow{\overline{q}_{\gamma,\delta}}\widetilde{\Coh}_{\delta;\alpha}^{\geq 1}  \xrightarrow{\overline{p}_{\gamma,\delta}} \RCoh_{\alpha}^{\geq 1},
\end{gather*}
obtained by base change from \eqref{E:truncconvdiagram3}.

\smallskip

\begin{corollary}\label{Cor:proppq} The restrictions of the maps $p, \overline{p}$ and  $p',$ $\overline{p}'$ to 
$\widetilde{\Coh}^{\geq 1}$ are proper and representable.
The restrictions of the maps $q,$ $\overline{q}$ to
$\widetilde{\Coh}^{\geq 1}$ are quasi-smooth.   
\end{corollary}

\begin{proof}
We already know that $p$ is proper. Since both $\mathcal{E}_\gamma$ and $\mathcal{F}_\alpha$ have perfect amplitude in $[-1,0]$ over $\RCoh^{\geq 1} \times \oS$, by \cite[Lem.~5.4]{QJiang} the maps $p'$ and $\overline{p}'$ are proper and quasi-smooth. Hence $\overline{p}$ is proper as well when $S$ is proper. Since $\RCoh_\alpha$ is an open substack of $\RCoh_{\alpha}(\oS)$, we may deduce the case of an arbitrary $S$ by base change. Indeed, note that there is a cartesian diagram
$$\xymatrix{\widetilde{\Coh}^{\geq 1}_{\delta;\alpha}(S) \ar[r] \ar[d]_-{\overline{p}} & \widetilde{\Coh}^{\geq 1}_{\delta;\alpha}(\overline{S}) \ar[d]_-{\overline{p}} \\ \RCoh^{\geq 1}_\alpha(S) \ar[r] & \RCoh^{\geq 1}_\alpha(\overline{S})
}
$$
since for any extension $0 \to \mathcal{F} \to \mathcal{E} \to \mathcal{O}_x\to 0$ with $\mathcal{F}, \mathcal{E}$ of dimension $\geq 1$, we have $\supp(\mathcal{E})=\supp(\mathcal{F})$.
Next, we claim that the maps $q$ and $\overline{q}$ are the restrictions to suitable open substacks of the projections
$$\mathbb{V}(\RHom_S(\mathcal{E}_\delta,\mathcal{E}_\alpha)[1]) \to \Coh_\delta \times \RCoh_\alpha
,\quad
\mathbb{V}(\RHom_S(\mathcal{E}_\gamma,\mathcal{E}_\delta)) \to \RCoh_\gamma \times \Coh_\delta.$$ 
Indeed, the condition for a sheaf to be supported on $S \subset \overline{S}$ is open, as is the condition for a morphism $\mathcal{E}_\gamma \to \mathcal{E}_\delta$ to be surjective.
As a consequence, $q,\overline{q}$ are both quasi-smooth.
\end{proof}

We will prove a more general version of Corollary~\ref{Cor:proppq} in \S\ref{sec:Heckepatterns}. Thanks to Corollary~\ref{Cor:proppq}, the truncated induction diagrams yield two types of Hecke operators
\begin{gather*}
    \begin{split}
        T_+ =p_!\circ q^*: H_*(\Coh_\delta,\Q) \otimes H_*(\Coh_{\alpha}^{\geq 1},\Q) \to H_*(\Coh_\gamma^{\geq 1},\Q),\\
        T_- =\overline{p}_!\circ\overline{q}^*: H_*(\Coh_\delta,\Q) \otimes H_*(\Coh_{\gamma}^{\geq 1},\Q) \to H_*(\Coh_\alpha^{\geq 1},\Q).
    \end{split}
\end{gather*}
Considering $S$-hyperbolic homology, we also define 
\begin{gather}\label{E:cpt-supp-Hecke}
\begin{split}
    T^c_+ =r\circ p'_!\circ q^*: 
    H^c_*(\Coh_\delta,\Q) \otimes H_*
    (\Coh_{\alpha}^{\geq 1},\Q) \to 
    H_*(\Coh_\gamma^{\geq 1},\Q),\\
    T^c_- =r\circ\overline{p}'_!\circ\overline{q}^*: 
    H^c_*(\Coh_\delta,\Q) \otimes H_*(\Coh_{\gamma}^{\geq 1},\Q) 
    \to H_*(\Coh_\alpha^{\geq 1},\Q),
\end{split}
\end{gather}
where $r: H_*^c(S,\Q) \to \Q$ is the canonical degree 0 map.
We will next use Proposition~\ref{L:Negut} to compute their action on tautological classes over suitable open substacks of $\RCoh^{\geq 1}$.

\medskip

\subsection{Atlases for length one correspondences}\label{sec:atlashecke}
We assume until~\S\ref{sec:open-Hecke} that $S=\oS$ is a projective surface. 
Let us fix a finite type open derived substack
$\mathfrak{U} \subset \RCoh^{\geq 1}_{\gamma}$ and a locally free resolution 
\begin{align}\label{C1}
0\to \mathcal{E}_{-1} \to \mathcal{E}_0 \to \mathcal{E}_{\gamma}|_{\mathfrak{U}\times\oS} \to 0.
\end{align}
We keep the notations from the previous section. Let us explicitly describe the
projectivization $\mathbb{P}(\mathcal{E}_{\gamma})$ in terms of the complex
$\mathcal{E}_{-1} \to \mathcal{E}_0$. Let $\mathbb{P}(\mathcal{E}_0)$ be the
total space of the projective bundle associated to $\mathcal{E}_0$ and 
$\pi: \mathbb{P}(\Ec_0) \to \mathfrak{U} \times \overline{S}$ be the projection.
The $\C$-points of $\mathbb{P}(\Ec_0)$ parametrize triples $(\mathcal{W},y,\lambda)$ where $\mathcal{W} \in \mathfrak{U}$ is a coherent sheaf on $\overline{S}$ of dimension 
$\geq 1$, $y \in \overline{S}$ and $\lambda : \mathcal{E}_0|_{(\mathcal{W},y)} \to \C$
is a nontrivial linear form, defined up to multiplication by a scalar. 
The morphism $\pi^*(\mathcal{E}_0^\vee) \to \pi^*(\Ec_{-1}^\vee)$
yields a map 
$\mathcal{O}_{\mathbb{P}(\Ec_0)}(-1) \to \pi^*(\mathcal{E}_{-1}^\vee)$.
Recall that, if $V$ is a finite-dimensional vector space, then $\mathbb{P}(V)$ parametrizes hyperplanes of $V$, see \S\ref{Sec:notations}.
Unless specifically mentioned, all fiber products and tensor products are derived.
This map can be viewed as a section 
\begin{align}\label{section-s}
s \in H^0(\mathbb{P}(\mathcal{E}_0),\pi^*(\mathcal{E}_{-1}^\vee)(1)).
\end{align}
The zero locus $Z(s)$ of $s$ parametrizes the triples
$(\mathcal{W},y,\lambda) \in \mathbb{P}(\mathcal{E}_0)$ 
for which the map $\lambda$ descends to $\mathcal{W}|_y$. 
By Proposition \ref{L:Negut}, we have
$$\widetilde{\Coh}_{\delta;\alpha} \underset{\RCoh_{\gamma}}{\times} \mathfrak{U}=\mathbb{P}(\mathcal{E}_{\gamma}|_{\mathfrak{U}\times\overline S})\simeq \mathbb{P}(\mathcal{E}_0) \underset{\mathbb{V}(\pi^*(\mathcal{E}_{-1}^\vee)(1))}{\times}\mathbb{P}(\mathcal{E}_0).$$
The derived fiber product is taken with respect to the sections $s$ and 
0  of $\pi^*(\mathcal{E}_{-1}^\vee)(1)$. 
Let $\mathbb{P}(\mathcal{E}_0^{cl})$ be the classical projective bundle of $\mathcal{E}^{cl}_0$ over $\mathfrak{U}^{cl}\times \oS$.
By \cite[prop.~4.21]{QJiang},
the classical truncation of the derived stack
$\mathbb{P}(\Ec_{\gamma}|_{\mathfrak{U}\times\overline S})$ 
is isomorphic to the zero locus in
$\mathbb{P}(\mathcal{E}_0^{cl})$ of the section $s^{cl}$, i.e., we have
$$\mathbb{P}(\Ec_{\gamma}|_{\mathfrak{U}\times\overline S})^{cl} = Z(s^{cl}).$$

On the other hand, over the partial truncation
$\Coh_{\delta} \times \RCoh_{\alpha}^{\geq 1}$, the complex $\RHom(\mathcal{E}_\delta, \mathcal{E}_{\alpha})[1]$ 
has perfect amplitude in $[0,1]$. 
We fix a finite type open derived substack 
$\mathfrak{U}'$ of $\RCoh_{\alpha}^{\geq 1}$ with
a presentation 
of the complex $\RHom(\mathcal{E}_\delta, \mathcal{E}_{\alpha})[1]\big|_{\Coh_{\delta} \times \mathfrak{U}'}$
\begin{align}\label{C2}
0\to\mathcal{V}_0 \to \mathcal{V}_1\to 0
\end{align}
 Let $t:\mathbb{V}(\mathcal{V}_0) \to \Coh_{\delta} \times\mathfrak{U}'$ 
 be the projection.
 The map $\mathcal{V}_0 \to \mathcal{V}_1$ yields a section
\begin{align}\label{section-s'}s' \in H^0(\mathbb{V}(\mathcal{V}_0), t^*(\mathcal{V}_1)).\end{align}
By \S \ref{sec:defCOHA0} there is an isomorphism
$$\widetilde{\Coh}_{\delta;\alpha} \underset{\RCoh_{\alpha}}{\times}\mathfrak{U}'= \mathbb{V}(\mathcal{V}_0) \underset{\mathbb{V}(t^*(\mathcal{V}_1))}{\times}\mathbb{V}(\mathcal{V}_0)$$
where the derived fiber product is taken with respect to the sections $s'$
and 0. By \cite[prop.~4.10]{QJiang}, the classical truncation is
isomorphic to the zero locus $Z((s')^{cl}) \subset \mathbb{V}(\mathcal{V}^{cl}_0)$.
We thus get the following isomorphisms of derived and classical stacks over any open set on which both presentations \eqref{C1} and \eqref{C2} exist:
\begin{align}\label{E:isomclassicalstacks}
\begin{split}
\mathbb{V}(\mathcal{V}_0) \underset{\mathbb{V}(t^*(\mathcal{V}_1))}{\times}\mathbb{V}(\mathcal{V}_0) &\simeq \widetilde{\RCoh}_{\delta;\alpha} \simeq \mathbb{P}(\Ec_0) \underset{\mathbb{V}(\pi^*(\Ec_{-1}^\vee)(1))}{\times}\mathbb{P}(\Ec_0),\\
\mathbb{V}(\mathcal{V}^{cl}_0) &\supset Z((s')^{cl}) = Z(s^{cl}) \subset \mathbb{P}(\Ec^{cl}_0).
\end{split}
\end{align}

\medskip

The case of $\mathbb{P}(\mathcal{F}_\alpha)$ is similar, let us briefly sketch it. Observe that $$\RHom_S(\mathcal{E}_\gamma,\mathcal{E}_\delta)=Rp_{12*}(\mathcal{E}_\gamma^\vee \otimes \mathcal{E}_\delta)=Rp_{12*}(\mathcal{E}_\gamma^\vee \otimes \mathcal{O}_{\Delta_{\oS}} \otimes \rho)=\mathcal{E}_\gamma^\vee\otimes\rho.$$
Hence there is a factorization
$$\xymatrix{\mathbb{V}(\mathcal{E}_\gamma^\vee\otimes \rho) \ar[d] & \widetilde{\Coh}^{\geq 1}_{\delta;\alpha} \ar[l]_-{j} \ar[ld]^-{\overline{q}}\\
\RCoh^{\geq 1}_\gamma \times\Coh_\delta
}$$
where $j$ is an open embedding. 
Let us restrict everything to an open substack of finite type $\mathfrak{U}\subset \RCoh_\gamma^{\geq 1}$. Consider locally free resolution~\eqref{C1}, and let $\overline{\pi}: \mathbb{V}(\mathcal{E}^\vee_\gamma\otimes \rho) \to \mathfrak{U} \times \Coh_\delta$ be the projection. 
We have an obvious section 
\begin{align}\label{section-bar-s}
\overline{s} \in H^0(\mathbb{V}(\mathcal{E}_0^\vee\otimes \rho),\overline{\pi}^*(\mathcal{E}^\vee_{-1}\otimes \rho)).
\end{align} 
The zero locus of $\overline{s}^{cl}$
in $\mathbb{V}(\mathcal{E}_0^\vee\otimes \rho^{cl})$ is
$$\mathbb{V}(\mathcal{E}^{\vee}_\gamma\otimes \rho|_{\mathfrak{U} \times\Coh_\delta})^{cl} =Z(\overline{s}^{cl}) $$
Likewise, fix an open substack of finite type 
$\mathfrak{U}' \subset \RCoh_\alpha^{\geq 1}$ 
as in~\eqref{C2}, so that
$\mathcal{F}_{\alpha}|_{ \mathfrak{U}'\times \oS}$
has a presentation
\[
0\to\mathcal{V}_1^\vee \to \mathcal{V}_0^\vee\to 0.
\]
We have an obvious section 
\begin{align}\label{section-bar-s'}
\overline{s}' \in H^0(\mathbb{P}(\mathcal{V}_0^\vee),\overline{p}^*(\mathcal{V}_1(1))),
\end{align} 
and the zero locus of $(\overline{s}')^{cl}$ in $\mathbb{P}(\mathcal{V}_0^{\vee,cl})$ is $\mathbb{P}(\mathcal{F}_{\alpha}|_{\mathfrak{U}\times\overline S})^{cl} =Z((\overline{s}')^{cl})$.
Therefore, over any open set over which both presentations \eqref{C1}
and \eqref{C2} exist, 
we have the following isomorphisms of derived and classical stacks
\begin{align}
\begin{split}
 \mathbb{P}(\mathcal{V}_0^\vee) \underset{\mathbb{V}(\overline{p}^*(\mathcal{V}_1\otimes\rho))}{\times}\mathbb{P}(\mathcal{V}_0^\vee) &\simeq \widetilde{\Coh}_{\delta;\alpha} \simeq \mathbb{V}(\mathcal{E}_0^\vee\otimes\rho) \underset{\mathbb{V}(\overline{\pi}^*(\mathcal{E}_{-1}^\vee\otimes\rho))}{\times}\mathbb{V}(\mathcal{E}_0^\vee\otimes\rho),\\
\mathbb{P}(\mathcal{V}_0^{\vee,cl}) &\supset Z((\overline{s}')^{cl}) \simeq Z(\overline{s}^{cl}) \subset \mathbb{V}(\mathcal{E}_0^{\vee,cl}\otimes\rho).
\end{split}
\end{align}

\medskip

For the future use, note that the image of the map
$j^{cl}$ is the complement of the zero section in $\mathbb{V}(\mathcal{E}_0^{\vee, cl}\otimes\rho)$. 
Therefore, the section $\overline{s}^{cl}$ is regular over this complement if and only if the section
$s^{cl}$ is regular. Similarly, the section 
$(\overline{s}')^{cl}$ is regular if and only if the section $(s')^{cl}$ is regular.

\medskip

\subsection{Computation of Hecke operators on fundamental classes}\label{S:Hecketaut} 
We are now in position to 
compute the action of Hecke operators on the fundamental classes 
$[\Coh^{\geq 1}_\alpha], [\Coh^{\geq 1}_\gamma]$ and on the virtual fundamental classes 
$[\RCoh^{\geq 1}_\alpha], [\RCoh_\gamma^{\geq 1}]$. 
We keep the notation of the previous sections. 
Fix finite type open substacks 
$\mathfrak{U} \subset \RCoh^{\geq 1}_{\gamma}$
and  
$\mathfrak{U}' \subset  \RCoh_{\alpha}^{\geq 1}$ such that
\begin{align}\label{eq:T-plus-condition}
    q^{-1}(\RCoh_{\delta}(\oS) \times\mathfrak{U}') \supseteq p^{-1}(\mathfrak{U}).   
\end{align}
Recall that $\gamma=\alpha+\delta$.
We will carry out the computation of the action on the non-virtual fundamental 
classes first, under the following assumption
\begin{equation}\label{E:assumption}
    \framebox[1.1\width]{\text{the sections $s^{cl}$ and $(s')^{cl}$ are regular.}}
\end{equation}
This condition implies that the sections $\overline{s}^{cl}$ and $(\overline{s}')^{cl}$ 
in \eqref{section-bar-s}, \eqref{section-bar-s'}
are regular as well. 
This means that the maps $q_{\delta,\alpha}^{cl}$ and $p_{\delta,\alpha}^{cl}$ 
in \eqref{E:convdiagram3}
are of the expected dimension over each irreducible component of 
$\mathfrak{U}^{cl}$ and $(\mathfrak{U}')^{cl}$ respectively.
Recall from \S\ref{sec:stack0} the isomorphisms
$$\Coh_\delta = \oS \times B\mathbb{G}_m, \qquad H^*(\Coh_\delta,\Q)=H^*(\oS,\Q)\otimes \Q[u],$$
where $u=c_1(\rho)$ is as in \eqref{u}. Let 
$$[\Delta_{\oS}] \in H^*(\Coh_\delta,\Q)\otimes H^*(\overline{S},\Q)$$ 
be the fundamental class of the diagonal.
Let $h_n$ be the complete symmetric function of degree $n$.
Recall the notation $h_n(\mathcal{E}_\gamma)$ in \eqref{f(E)}.

\medskip

\begin{proposition}\label{P:negut} 
Assume that \eqref{E:assumption} holds. 
Let $r=\rk(\alpha)$. 
\begin{enumerate}[label=$\mathrm{(\alph*)}$,leftmargin=8mm]
\item
For any $l \geq 0$ the following equality holds in $H_*(\mathfrak{U},\Q) \otimes H^*(\overline{S},\Q)$:
 \begin{equation}\label{E:SegreeFormula1}
 T_+\Big(\big(u^l[\Delta_{\overline S}] \cap [\Coh_\delta]\big) \otimes [\Coh_{\alpha}^{\geq 1}]\Big)\Big|_{\mathfrak{U}}= h_{l+1-r}(\mathcal{E}_{\gamma}) \cap [\mathfrak{U}]
\end{equation} 
\item For any $\mu \in H^*({\oS},\Q)$ we have
 \begin{equation}\label{E:SegreeFormula15}
T_+ \Big(\big(\mu u^l\cap [\Coh_\delta]\big) \otimes [\Coh_{\alpha}^{\geq 1}]\Big)\Big|_{\mathfrak{U}}= h_{l+1-r}(\mu) \bullet [\mathfrak{U}].
\end{equation} 
\end{enumerate}
\end{proposition}

Before giving the proof, it is worth spelling out the meaning of formula \eqref{E:SegreeFormula1}. On the l.h.s., 
we view $\big(u^l[\Delta_{\overline S}] \cap [\Coh_\delta]\big)\otimes [\Coh_{\alpha}^{\geq 1}]$ as an element of the triple tensor product $H_*(\Coh_\delta,\Q) \otimes H^*(\overline{S},\Q) \otimes H_*(\Coh_\alpha^{\geq 1})$, and we apply the operator $T_+$ to the first and third factors. On the r.h.s., we view $h_{l+1-r}(\mathcal{E}_\gamma)$ as an element of $H^*(\Coh_\gamma,\Q) \otimes H^*(\overline{S},\Q)$ and act by cap product on the left term in the tensor product.

\begin{proof} Let $j$ be the open immersion 
$p^{-1}(\mathfrak{U}) \subset q^{-1}(\Coh_{\delta} \times\mathfrak{U}')$.
We will work with classical stacks, 
but we will omit the superscript $cl$ in the notation. 
Under the assumption~(\ref{E:assumption}), we have
$$j^*q^!([\Coh_\delta] \otimes [\Coh_\alpha])=j^*([Z(s')])=[Z(s)].$$
We abbreviate $\mathcal{O}(1)=\mathcal{O}_{\mathbb{P}(\mathcal{E}_0)}(1)$.
By Proposition~\ref{L:Negut}, we have the following isomorphism of coherent sheaves over $p^{-1}(\mathfrak{U}) \times \overline{S}$:
$$(j \times \Id)^*(q \times \Id)^*\mathcal{E}_\delta =\tau_*\mathcal{O}(1)$$
We deduce that
\begin{equation}\label{form3}
    \begin{split}
    (j \times \Id)^*(q \times \Id)^!((\ch(\mathcal{E}_\delta) \cap [\Coh_\delta \times  \overline{S}] )\otimes [\Coh_\alpha])
    &=\ch(\tau_*\mathcal{O}(1)) \cap [Z(s) \times \overline{S}]
    \end{split}
\end{equation}
Now, we consider the commutative diagram
$$\xymatrix{Z(s)\ar[d]_-{\tau} \ar[r]^-{\iota} & \mathbb{P}(\mathcal{E}_0) \ar[dr]^-{\pi} \ar[d]_-{\tau'}& &\\
 Z(s) \times \overline{S}\ar[r]^-{\iota\times \Id} & \mathbb{P}(\mathcal{E}_0) \times \overline{S}\ar[r]^-{p'\times \Id}
& \mathfrak{U} \times \overline{S}}$$
The maps $\iota,$ $ \iota'$ are the obvious closed immersions, 
the map $\pi$ is the projection to $\mathfrak{U} \times \overline{S}$, 
the map $p'$ is the projection to $\mathfrak{U}$, 
and the maps $\tau,$ $\tau'$ are defined as in \eqref{E:deftau}. 
The bottom row of the diagram clearly composes to $p\times \Id$.

Let $i: \mathfrak{U} \to \RCoh^{\geq 1}_\gamma$ be the open immersion.
Applying $(p\times \Id)_*$ to \eqref{form3}
and restricting to the open subset
$\mathfrak{U} \times \overline{S}$, we get
\begin{equation*}
    \begin{split}
    (i \times \Id)^*&T_+\big((\ch(\mathcal{E}_\delta) \cap [\Coh_\delta \times  \overline{S}]) \otimes [\Coh_\alpha^{\geq 1}]\big)\\
    &= (p' \times \Id)_*(\iota \times \Id)_*\big(\ch(\tau_*\mathcal{O}(1)) \cap ([Z(s) \times \overline{S}])\big)
    \end{split}
\end{equation*}
We claim that 
\begin{equation}\label{E:basechange}
\tau_*\iota^*\mathcal{O}(1)=(\iota \times \Id)^*\tau'_*\mathcal{O}(1).
\end{equation}
Indeed, since $s$ is regular, the maps $\iota\times \Id$ and $\tau'$ are Tor-independent, and so we may use the proper base change theorem. Applying Chern character to \eqref{E:basechange}, we get
\begin{equation*}
    \begin{split}
    (i \times \Id)^*&T_+\big((\ch(\mathcal{E}_\delta) \cap [\Coh_\delta \times  \overline{S}]) \otimes [\Coh_\alpha^{\geq 1}])\big)\\
    &= (p' \times \Id)_*\big(\ch(\tau'_*\mathcal{O}(1)) \cap (\iota \times \Id)_*([Z(s) \times \overline{S}])\big)\\
    &=(p' \times \Id)_*\big(\ch(\tau'_*\mathcal{O}(1)) \cap \eu\left(\pi^*\mathcal{E}_{-1}^\vee(1)\right) \cap [\mathbb{P}(\mathcal{E}_0)]\big)\\
    &=(p' \times \Id)_*\big(\eu\left(\pi^*\mathcal{E}_{-1}^\vee(1)\right)\cap \tau'_*\ch(\mathcal{O}(1)) \cap \Td_{\oS}^{-1} \cap [\mathbb{P}(\mathcal{E}_0)]\big)\\
    &=\Td_{\oS}^{-1} \cap (p' \times \Id)_*\tau'_*\big(\eu\left(\pi^*\mathcal{E}_{-1}^\vee(1)\right) \cap \ch(\mathcal{O}(1)) \cap [\mathbb{P}(\mathcal{E}_0)]\big)\\
    &=\Td_{\oS}^{-1} \cap \pi_*\big(\eu\left(\pi^*\mathcal{E}_{-1}^\vee(1)\right) \cap \ch(\mathcal{O}(1)) \cap [\mathbb{P}(\mathcal{E}_0)]\big)
    \end{split}
\end{equation*}
where we have successively used the fact that $s$ is regular and the Grothendieck-Riemann-Roch formula for the proper morphism $\tau$. 
The following formula is well-known.

\medskip

\begin{lemma} Let $X$ be a stack. 
Let $\mathcal{E}_{-1},$ $ \mathcal{E}_0$ be vector bundles over $X$. Set $r=\rk(\mathcal{E}_0)-\rk(\mathcal{E}_{-1})$. 
Let $\pi: \mathbb{P}(\mathcal{E}_0) \to X$ be the projection. 
The following formula holds in $H^*(X,\Q)$:
\[
    \pushQED{\qed}
    \pi_*\big(\ch(\mathcal{O}(1)) \cup \eu(\pi^*\mathcal{E}_{-1}^\vee(1))\big)=
    \sum_{n\in\mathbb{N}}\frac{1}{n!}\,h_{n-r+1}(\mathcal{E}_0 - \mathcal{E}_{-1}).\qedhere
    \popQED
\]
\end{lemma}


Using the above lemma, we deduce that
\begin{equation}\label{E:SegreeFormula}
i^* T_+\Big(\big(\ch(\mathcal{E}_\delta) \cap [\Coh_\delta]\big) \otimes [\Coh_{\alpha}^{\geq 1}]\Big)=\Td_{\oS}^{-1} \cap \sum_n \frac{1}{n!}h_{n+1-r}(\mathcal{E}_{\gamma}) \cap [\mathfrak{U}].
\end{equation} 
The proof of \eqref{E:SegreeFormula1} now follows by multiplying throughout by the Todd class $\Td_{\oS}$ and equating the terms of fixed homological degrees.
Finally, the formula \eqref{E:SegreeFormula15} is obtained by taking the intersection pairing with $\mu$.
\end{proof}

\smallskip

A similar analysis can be made for negative Hecke correspondences. 
Once again, we fix finite type open substacks $\mathfrak{U} \subset \RCoh^{\geq 1}_{\gamma}$ and $\mathfrak{U}' \subset  \RCoh_{\alpha}^{\geq 1}$ satisfying
\begin{align}\label{eq:T-minus-condition}
    \overline{q}^{-1}(\RCoh_{\delta}(\oS) \times\mathfrak{U}) \supseteq \overline{p}^{-1}(\mathfrak{U}')\qquad \Leftrightarrow \qquad q^{-1}(\RCoh_{\delta}(\oS) \times\mathfrak{U}') \subseteq p^{-1}(\mathfrak{U}). 
\end{align}
Recall the shift operation $\tau_c : \Lambda' \to \Lambda'[c]$ 
defined by~\eqref{eq:shift-map}, and the notations in~\eqref{flambda}.
We consider the algebra automorphism $x \mapsto \widetilde{x}$ of $\Lambda(\oS)$ such that
$$\widetilde{f}(\lambda) = \int_S p(\tau_{c_1}f) \cup \lambda \in \Lambda(\oS)
,\quad
f\in \Lambda
,\quad
\lambda\in H^*(\oS,\Q).$$ 
For any $\gamma \in K_0^c(\oS)_\Q$, we have
$$\ev_{\gamma}(\widetilde{f}(\lambda))=\int_S f(\mathcal{E}_\gamma \otimes K^\vee_\oS) \cup \lambda$$
where $f(\mathcal{E}_\gamma \otimes K^\vee_\oS)$ is as in \eqref{f(E)}.

\begin{remark}
    For each $n\geqslant 0$ we have
    \[
        \widetilde{e}_n(\mu) = \sum_{i=0}^n \binom{\r-n+i}{i} e_{n-i}(\mu\cup c_1^i).
    \]
    In the absence of an equivariant parameter we have $c_1^3=0$, and so only the first three terms of the sum above do not vanish. In particular, when $S$ has trivial canonical bundle we have 
    $$\widetilde{e}_n(\mu) = e_n(\mu).$$ 
    In the situation of Proposition~\ref{P:negutdual} below we have $\r=r$.
\end{remark}

\smallskip

\begin{proposition}\label{P:negutdual}
Assume that \eqref{E:assumption} holds. 
Let $r=\rk(\alpha)$. 
\begin{enumerate}[label=$\mathrm{(\alph*)}$,leftmargin=8mm]
\item
For any $l \geq 0$ we have the following equality in $H_*(\mathfrak{U}',\Q) \otimes H^*(\overline{S},\Q)$
\begin{equation}\label{E:SegreeFormula1dual}
 T_-\Big(\big(u^l[\Delta_\oS] \cap [\Coh_\delta]\big) \otimes [\Coh_{\gamma}]\Big)\Big|_{\mathfrak{U}'}= (-1)^{l}e_{l+1+r}(\mathcal{E}_{\alpha} \otimes K_{\oS}^\vee) \cap [\mathfrak{U}']
\end{equation} 
\item
For any $\mu \in H^*({S},\Q)$ we have
 \begin{equation}\label{E:SegreeFormula15dual}
    \pushQED{\qed}
T_- \Big(\big(\mu u^l\cap [\Coh_\delta]\big) \otimes [\Coh_{\gamma}]\Big)\Big|_{\mathfrak{U}'} = (-1)^{l}\widetilde{e}_{l+1+r}(\mu)\bullet [\mathfrak{U}'].\qedhere \popQED
\end{equation} 
\end{enumerate}
\end{proposition}

\medskip

In this paper, we will check the assumption \eqref{E:assumption} in two situations of interest: the Hilbert schemes of points $\Hilb_n(\overline{S})$ in \S\ref{sec:HilbregularHP}, and the stacks of Higgs bundles $\Higgs_{r,d}$ over a smooth projective curve in \S\ref{sec:conditionHiggs}.

\medskip

Let us now turn our attention to the action of Hecke operators on \textit{virtual} fundamental classes. 

\smallskip

\begin{proposition}\label{prop:Negutlemmavirtual} For any $\alpha$, we have
\begin{align*}
    T_+\Big(\big(\mu u^l \cap [\Coh_\delta]\big) \otimes [\RCoh_{\alpha}]\Big)\Big|_{\mathfrak{U}} &= h_{l+1-r}(\mu) \bullet [\mathfrak{U}],\\
    T_-\Big(\big(\mu u^l \cap [\Coh_\delta]\big) \otimes [\RCoh_{\gamma}]\Big)\Big|_{\mathfrak{U}'} &= (-1)^{l}\widetilde{e}_{l+1+r}(\mu)\bullet [\mathfrak{U}'].
\end{align*}
\end{proposition}

\begin{proof} The proof follows along exactly the same lines as the proof of Propositions~\ref{P:negut} and \ref{P:negutdual}. We use the fact that the Gysin pullback by a quasi-smooth morphism preserves virtual fundamental classes, see \S\ref{sec:relBMhomology} for more details, and we use the projection formula of Proposition \ref{Prop:appA}(d).
\end{proof}

\smallskip

\begin{remark}
   Given that Proposition~\ref{prop:Negutlemmavirtual} holds without any regularity assumption, one might wonder why one should bother considering non-virtual fundamental classes at all. The answer we give is that the virtual fundamental classes $[\mathfrak{U}]$ and $[\mathfrak{U}']$ typically lie in a homological degree lower than that of their non-virtual cousins, and hence generate a different (and in many cases strictly smaller) space of tautological classes.  
 
\end{remark}

\medskip

\subsection{Hecke operators on tautological classes}\label{sec:explicitcomputHecke} 
Recall that $S=\oS$ is proper and that $u=c_1(\rho)$ is as in \eqref{u}.
In particular, the map $\omega_\delta : \Lambda(S) \to H_*(\Coh_\delta,\Q)$ 
in \eqref{omegadelta} is surjective by Lemma~\ref{lm:omega-surj}.
Hence Proposition~\ref{P:negut} allows us to describe the action of the space $\mathbf{H}_0(S)[1,-]$ on 
the fundamental class $[\Coh_\alpha]$ after restriction to suitable open substacks 
$i:\mathfrak{U}_\alpha\to\mathfrak{Coh}_\alpha$. 
Using the $\Lambda({S})$-module algebra structure of $\mathbf{H}(S)$, we will now deduce a formula for the action of $\mathbf{H}_0(S)[1,-]$ on the subspace 
$$\Lambda({S}) \bullet [\mathfrak{U}_\alpha]\subset H^*(\mathfrak{U}_\alpha,\Q)$$
spanned by the tautological classes by length one Hecke operators.
To do so, observe that there is a unique linear map
$$\mathcal{L}^+_r : H^*(\Coh_\delta,\Q) \otimes \Lambda({S}) \to \Lambda({S})
,\quad
r\geqslant 0$$
such that
\begin{enumerate}[label=$\mathrm{(\alph*)}$,leftmargin=6mm]
\setlength\itemsep{0.5em}

\item
for any $\lambda \in H^*({S},\Q)$ and $l \in \N$, we have
$$\mathcal{L}^+_r(\lambda u^l \otimes 1)=h_{l+1-r}(\lambda),$$

\item 
for any $x,y \in \Lambda({S})$ and $\xi \in H^*(\Coh_\delta,\Q)$, we have

$$x \cdot \mathcal{L}^+_r(\xi \otimes y)=
\sum \mathcal{L}^+_r( (\ev_{\delta}(x^{(1)}) \cup \xi )\otimes( x^{(2)} \cdot y))$$

\end{enumerate}
where
$\Delta(x)=\sum x^{(1)} \otimes x^{(2)}$.
By \eqref{E:SegreeFormula15} and Proposition~\ref{prop:modulealgebra}, we deduce that for $r=\rk(\alpha)$ we have
\begin{equation}\label{E:defT}
i^*\big((\xi \cap[\Coh_\delta]) \star (
x \bullet [\Coh_\alpha])\big)=i^*\big(\mathcal{L}^+_r(\xi \otimes x) \bullet [\Coh_\gamma]\big)
\end{equation}
We will abbreviate
$$\mathcal{L}^+_r(\xi) : \Lambda({S}) \to \Lambda({S})
,\quad
x\mapsto\mathcal{L}^+_r(\xi\otimes x).$$
Then (b) translates into the following relation
\begin{align}\label{eq:pT-comm}
    [p_n(\lambda), \mathcal{L}^+_r(\xi)]=\mathcal{L}^+_r(\ev_\delta(p_n(\lambda)) \cup \xi).  
\end{align}
A direct computation using \eqref{E:ToddclassS} and \eqref{Edelta} gives
\begin{align}\label{E:evpn}
    \ev_\delta(p_n(\lambda))=f_n(u)\cup\lambda
\end{align}
where $u=c_1(\rho)$ is as in \eqref{u} and
$$f_n(u)=\frac{u^n - (u-t_1)^n - (u-t_2)^n + (u-t_1-t_2)^n}{t_1t_2}.$$
We define an algebra homomorphism $R^+$ and a $\Lambda({S})$-linear map $Q^+$ such that
\begin{align}\label{RQ}
\begin{split}
R^+&: \Lambda({S}) \to  \Lambda({S}) \otimes H^*(\Coh_\delta,\Q)
,\quad
p_n(\lambda)\mapsto p_n(\lambda)\otimes 1  - 1 \otimes \ev_\delta(p_n(\lambda)),\\ 
Q^+&: \Lambda({S}) \otimes H^*(\Coh_\delta, \Q)  \to  \Lambda({S})
,\quad
x \otimes \lambda u^l\mapsto x\cdot h_{l}(\lambda).
\end{split}
\end{align}
We will extend $Q^+$ to a linear map
$$Q^+: \Lambda(S) \otimes H^*(S,\Q)[u,u^{-1}] \to \Lambda(S)$$ 
such that $Q^+(x \otimes  u^l)=0$ for $l <0$. Then, the following formula holds
\begin{equation}\label{E:finalformT}
\mathcal{L}^+_r(\lambda u^l)=L_{l+1-r}^+(\lambda)
\end{equation}
where the right hand side is given by
\begin{equation}\label{eq:T-fla+}
    L^+_n(\lambda): \Lambda(S)\to \Lambda(S), \quad f \mapsto Q^+(\lambda u^n R^+(f))
    ,\quad n\in\Z.
\end{equation}
Again, it will be convenient to formally extend this definition to any $r \in \Z$.

\smallskip

Using Proposition~\ref{P:negutdual}, we write similar formulas for negative correspondences.
The analog of the relation (b) is 
$$x \cdot \mathcal{L}^-_r(\xi \otimes y)=
\sum \mathcal{L}^-_r\big( (\ev_{\delta}(\upsilon(x^{(1)})) \cup \xi )\otimes (x^{(2)} \cdot y)\big)
,\quad
x,y \in \Lambda({S}),$$
where $\upsilon$ is the involution~\eqref{eq:involution}.
This translates to the relation 
\[
    [\widetilde{p}_n(\lambda),\mathcal{L}_r^-(\xi)] = - \mathcal{L}_r^-(\ev_\delta(\widetilde{p}_n(\lambda))\cup \xi).
\]
Note that
\begin{align*}
    \ev_\delta(\widetilde{p}_n(\lambda)) 
    &= \int_S p_n(\mathcal{E}_\delta\otimes K^\vee_\oS)\cup \lambda \\
    &= f_n(u+c_1)\cup\lambda \\
    &= (-1)^n f_n(-u)\cup \lambda.
\end{align*}
Thus we obtain the following formula
$$ \mathcal{L}^-_r(\lambda u^l)=(-1)^{r+1}\,L^-_{l+1+r}(\lambda)$$
where
\begin{equation}\label{eq:T-fla-}
    L^-_n(\lambda): \Lambda(S)\to \Lambda(S), \quad f \mapsto Q^-(\lambda u^n R^-(f))
    ,\quad n\in\Z.
\end{equation}
and $R^-$ and $Q^-$ are the algebra homomorphism and the 
$\Lambda(S)$-linear map such that
\begin{align*}
R^- &: \Lambda(S) \to \Lambda(S) \otimes H^*(\Coh_\delta,\Q)
,\quad
\widetilde{p}_n(\lambda)\mapsto \widetilde{p}_n(\lambda) \otimes 1 +  (-1)^n \otimes f_n(-u)\cup \lambda,\\
Q^- &: \Lambda({S}) \otimes H^*(\Coh_\delta)  \to  \Lambda({S}) 
,\quad
x \otimes \lambda u^l\mapsto (-1)^lx\cdot \widetilde{e}_{l}(\lambda).
\end{align*}

\smallskip

Propositions~\ref{P:negut}, \ref{P:negutdual} yield the following.

\begin{lemma} Let $r=\rk(\alpha)$.
We have
\[
    \pushQED{\qed}
    T_\pm((\xi\cap[\Coh_\delta])\otimes(f\bullet[\mathfrak{U}_\alpha]))=
    \mathcal{L}^\pm_r(\xi)(f)\bullet[\mathfrak{U}_\alpha]
    ,\quad
    \xi \in H^*(\Coh_\delta)
    ,\quad 
    f\in\Lambda(S).\qedhere
    \popQED
\]
\end{lemma}

\medskip

For future use, we record here the following easily deduced formulas, valid for any $n$ and $\lambda$:
\begin{align}
    \begin{split}
R^+(\widetilde{p}_n(\lambda))&=\widetilde{p}_n(\lambda)\otimes 1 - (-1)^n \otimes f_n(-u)\cup \lambda\label{E:R+R-onpn}\\
R^-(p_n(\lambda))&=p_n(\lambda) \otimes 1 +1 \otimes f_n(u)\cup \lambda.
    \end{split}
\end{align}

\begin{remark}\label{E:vertexoperators}
The operators $\mathcal{L}_r^\pm$ above are example of vertex operators. We have
\begin{equation*}
\begin{split}
\sum_{\substack{n \geq 0\\ \gamma}} &\mathcal{L}_r^+(u^n\gamma)\otimes \gamma^{*}s^{-n+r-1}\\
&=\Bigg\{\exp\Big( \sum_{\substack{n>0\\\gamma}} \frac{p_n}{n}(\gamma)\otimes \gamma^{*}s^{-n}\Big)_{[s^{<r}]}
\exp\Big(-\sum_{\substack{n \geq 0\\\gamma}}\frac{\partial}{\partial \kappa_{n}(\gamma)} \otimes \gamma s^n\Big)\Bigg\}_{[s^{<r}]}
\end{split}
\end{equation*}
\begin{equation*}
\begin{split}
\sum_{\substack{n \geq 0\\ \gamma}}& \mathcal{L}_r^-(u^n\gamma)\otimes \gamma^{*}s^{-n-r-1}\\
&=(-1)^{r+1}
\Bigg\{\exp\Big(- \sum_{\substack{n>0\\ \gamma}} \frac{\tau_{c_1}p_n}{n}(\gamma)\otimes \gamma^{*}s^{-n}\Big)_{[s^{<-r}]}
\exp\Big(\sum_{\substack{n \geq 0\\ \gamma}}\frac{\partial}{\partial \kappa_{n}(\gamma)} \otimes \gamma s^n\Big)\Bigg\}_{[s^{<-r}]}
\end{split}
\end{equation*}
where $\{\gamma\},$ $\{\gamma^*\}$ are dual bases of $H^*(S,\Q)$ and the elements $\{\kappa_n(\lambda)\}$ are related to the $\{p_n(\lambda)\}$ through the relation
\begin{equation}\label{E:pkappa}
    \sum_{\substack{n \geq 0\\ \gamma}} \frac{x^{n+2}}{n!}\kappa_n(\gamma)\otimes \gamma^*=\Big(\Td_S(x) \cdot 
    \sum_{\substack{n>0\\ \gamma}} \frac{x^n}{n!}p_n(\gamma) \otimes \gamma^*\Big)_{[x^{>1}]}.
\end{equation}
\end{remark}

\smallskip

\subsection{Hecke operators on open surfaces} \label{sec:open-Hecke}
Let us return to the situation when the surface $S$ is cohomologically pure but not necessarily proper.
Pick a smooth compactification $\iota: S \to \oS$, and fix open substacks $\mathfrak{U} \subset \RCoh^{\geq 1}_{\gamma}(\oS)$ and $\mathfrak{U}' \subset  \RCoh_{\alpha}^{\geq 1}(\oS)$ such that
\footnote{note the difference between these conditions and (\ref{eq:T-plus-condition},~\ref{eq:T-minus-condition})}
\begin{align}\label{eq:T-open-conditions}
\begin{split}
    q^{-1}(\RCoh_{\delta}(S) \times\mathfrak{U}') \supseteq p^{-1}(\mathfrak{U}), \\ 
    \overline{q}^{-1}(\RCoh_{\delta}(S) \times\mathfrak{U}) \supseteq \overline{p}^{-1}(\mathfrak{U}').
\end{split}
\end{align}
These conditions imply that the correspondences~\eqref{E:convdiagram3} restrict to Hecke operators
\begin{align*}
    T_+: H_*(\Coh_\delta(S),\Q)\otimes H_*(\mathfrak{U}',\Q)\to H_*(\mathfrak{U},\Q),\\
    T_-: H_*(\Coh_\delta(S),\Q)\otimes H_*(\mathfrak{U},\Q)\to H_*(\mathfrak{U}',\Q)
\end{align*}
The following simple lemma, which follows from open base change, relates these operators to the analogous operators for $\oS$. Let $j$ denote both inclusions $\mathfrak{U} \to \RCoh^{\geq 1}_{\gamma}$ and $\mathfrak{U}' \to  \RCoh_{\alpha}^{\geq 1}$.
\begin{lemma}\label{lm:Hecke-base-change}
    For $x \in H_*(\Coh_\delta(\oS),\Q)$ and $c \in H_*(\mathfrak{U},\Q)$, $c' \in H_*(\mathfrak{U}',\Q)$ we have
    \begin{align*}
        j^*T_+(x \otimes c')=T_+(\iota^*x \otimes j^*c') \in H_*(\mathfrak{U},\Q),\\
        j^*T_-(x \otimes c)=T_-(\iota^*x \otimes j^*c) \in H_*(\mathfrak{U}',\Q).
    \end{align*}\qedhere
\end{lemma}

Next, let us consider the compactly supported Hecke correspondences. 
Assume that the open substacks $\mathfrak{U} \subset \RCoh^{\geq 1}_{\gamma}$ and $\mathfrak{U}' \subset  \RCoh_{\alpha}^{\geq 1}$ satisfy
\begin{align}\label{eq:T-cpt-conditions}
\begin{split}
    q^{-1}(\RCoh_{\delta}(S) \times\mathfrak{U}') \supseteq (p')^{-1}(\mathfrak{U}\times S), \\ 
    \overline{q}^{-1}(\RCoh_{\delta}(S) \times\mathfrak{U}) \supseteq (\overline{p}')^{-1}(\mathfrak{U}'\times S).
\end{split}
\end{align}
In this case we can define restrictions of Hecke operators~\eqref{E:cpt-supp-Hecke}:
\begin{align*}
    T_+^c: H^c_*(\Coh_\delta(S),\Q)\otimes H_*(\mathfrak{U}',\Q)\to H_*(\mathfrak{U},\Q),\\
    T_-^c: H^c_*(\Coh_\delta(S),\Q)\otimes H_*(\mathfrak{U},\Q)\to H_*(\mathfrak{U}',\Q).
\end{align*}

\smallskip

\begin{proposition}\label{P:heckecompact} 
For $x \in H_*^c(\Coh_\delta(S),\Q)$ and $c \in H_*(\mathfrak{U},\Q)$, $c' \in H_*(\mathfrak{U}',\Q)$ we have
\[
    j^*T_+(\iota_!(x)\otimes c') = T^c_+(x\otimes j^*c'),\qquad j^*T_-(\iota_!(x)\otimes c) = T^c_-(x\otimes j^*c).
\]
\end{proposition}

\begin{proof} The two cases being identical, we will prove the statement for $T^c_+$. Consider the following diagram in which we omit the obvious indices
$$\xymatrix{\Coh_\delta(\oS) \times \RCoh_\alpha(\oS) & \widetilde{\Coh}_{\delta,\alpha}(\oS) \ar[r]^-{p'} \ar[l]_-{q} & \RCoh_\gamma(\oS)\times \oS \ar[r]^-{p''}& \RCoh_\gamma(\oS)\\
\Coh_\delta \times \RCoh_\alpha(\oS) \ar[u]^-{\iota \times \Id}& \widetilde{\Coh}_{\delta,\alpha}(\oS)\times_{\oS}S \ar[u]^-{\tilde{\iota}}\ar[r]^-{s'} \ar[l]_-{t} & \RCoh_\gamma(\oS)\times S \ar[u]^-{\Id \times \iota}\ar[r]^-{s''}& \RCoh_\gamma(\oS)\ar@{=}[u]\\
\Coh_\delta \times \mathfrak{U}' \ar[u]^-{\Id \times j}& \widetilde{\Coh}_{\delta,\alpha}(\oS)\times_{\RCoh_\gamma(\oS)\times \oS}(\mathfrak{U}\times S)\ar[u]^-{\widetilde{j}} \ar[r]^-{(p')^\circ} \ar[l]_-{q^\circ} & \mathfrak{U} \times S \ar[u]^-{j \times \Id} \ar[r]^-{(p'')^\circ} & \RCoh_\gamma \ar[u]^-{j}
}$$
Apart from the rightmost column, the second row is obtained from the top row by base change $-\times_{\oS}S$. 
In addition, all of the vertical arrows are open embeddings and the middle square in the bottom is cartesian. 
By Proposition~\ref{P:basechange} (a), we have $q^!(\iota_!  \otimes \Id)=\tilde{\iota}_!t^!$. 
Hence, using Proposition~\ref{P:basechange} (b), we get
$$p''_*p'_*q^!(\iota_!(x) \otimes c')=p''_*p'_*\tilde{\iota}_!t^!(x \otimes c')=p''_*(\Id \otimes \iota_!)s'_*t^!(x \otimes c').$$ 
Note further that 
$$p''_*(\Id \otimes \iota_!)=r : H_*^c(S,\Q) \to \Q,$$ see Example~\ref{ex:relbmhomology}. 
It follows that 
$$T_+(\iota_!(x) \otimes c')=rs'_*t^!(x \otimes c').$$
On the other hand, by base change and functoriality of Gysin pullbacks, we have $$j^*s'_*t^!=(p')^\circ_*\widetilde{j}^*t^!=(p')^\circ_*(q^\circ)^!(\Id \otimes j^*).$$
Thus, we have
\[
    rs'_*t^!(x \otimes c')=r(p')^\circ_*(q^\circ)^!(x \otimes j^*(c'))=T^c_+(x\otimes j^*(c')).\qedhere
\]
\end{proof}

\smallskip

One particular instance when both restrictions of Hecke operators are well defined is when 
$$\mathfrak{U} = \Coh_\gamma^{\geq 1}
,\quad
\mathfrak{U}' = \Coh_\alpha^{\geq 1}.$$
If we further restrict to an appropriate open substack of $\Coh^{\geq 1}$ where the regularity conditions~\eqref{E:assumption} hold, Proposition~\ref{P:heckecompact} and Lemma~\ref{lm:Hecke-base-change} imply that the formulas of \S\ref{sec:explicitcomputHecke} apply verbatim to the restricted operators $T_\pm$, $T_\pm^c$ over $\mathfrak{U}'$, $\mathfrak{U}$.
Note that proper support implies $r=0$.
Since the evaluation map $\ev$ factors through $\Lambda(S)$ by Lemma~\ref{lem:eval-open-restriction}, in this case we can interpret the operators $\mathcal{L}^\pm_0$ as acting on $\Lambda(S)$.

\begin{remark}
Proposition~\ref{P:heckecompact} and Lemma~\ref{lm:Hecke-base-change} continue to hold if we replace the embedding $S \subset \oS$ by any open immersion.
\end{remark}

\medskip

\section{Deformed $W$-algebras of projective surfaces}\label{sec:Fock-ops}
In this section we introduce and study a class of associative algebras which are associated to the cohomology ring
of $S$. As these bear a resemblance to the deformed $W_{1+\infty}$-algebra studied
in \cite{SVIHES}, which corresponds to the case $S=\mathbb{A}^2$ 
with an action of the torus $(\C^\times)^2$, we will refer to these
as deformed $W$-algebras. In this section, we assume that $S$ is proper, in which case there is only one such type of $W$-algebra. The case of open surfaces will be addressed in \S\ref{sec:W-open-surf}.

\smallskip

We fix a smooth projective surface $S$. Recall that $c_1$, $c_2$ are the Chern classes of $S$ and that $s_2=c_1^2-c_2$.  The Poincaré polynomial of $S$ is
$$P_S(z)=\sum_n \dim(H^n(S,\Q))(-z)^n=h_S(z^{-1})z^{4}.$$

\medskip

\subsection{Positive halves of deformed $W$-algebras}\label{sec:W-alg-defs}
In what follows, we consistently write $\lambda\mu$ for the cup product $\lambda\cup\mu$, in order to simplify the notation.
\begin{definition}\label{def:W} 
Let $W^\geq(S)$ be the $\N\times\Z$-graded associative algebra generated by 
$$\psi_n(\lambda)
,\quad
T_n(\lambda)
,\quad
 n \geq 0
,\quad\lambda \in H^*(S,\Q)$$
and a central element $\c$
modulo the following relations for $a,b \in \Q$, $n,m \in\N$, $\lambda,\mu \in H^*(S,\Q)$
\begin{align*}
    \psi_m(a\lambda+b\mu) &= a\psi_m(\lambda)+b\psi_m(\mu), \tag{\rellab{W:a}{a}}\\
    T_n(a\lambda+b\mu) &= aT_n(\lambda)+bT_n(\mu), \tag{\rellab{W:b}{b}}\\
    [\psi_m(\lambda), \psi_n(\mu)] &= 0, \tag{\rellab{W:c}{c}}\\
    [\psi_m(\lambda), T_n(\mu)] &= m T_{m+n-1}(\lambda\mu), \tag{\rellab{W:d}{d}}\\
    [T_m(\lambda\mu), T_n(\nu)] &= [T_m(\lambda), T_n(\mu\nu)],\tag{\rellab{W:e}{e}}
\end{align*}
\[ \begin{split}[T_{m}(\lambda), T_{n+3}(\mu)] 
    &-3 [T_{m+1}(\lambda), T_{n+2}(\mu)] + 3 [T_{m+2}(\lambda), T_{n+1}(\mu)]  - [T_{m+3}(\lambda), T_{n}(\mu)]\\
     &- [T_{m}(\lambda), T_{n+1}(s_2 \mu)] + 
     [T_{m+1}(\lambda), T_{n}(s_2 \mu)] + \{T_m, T_n\}(c_1\Delta_S\lambda \mu)=0,\end{split} \tag{\rellab{W:f}{f}}\]
\begin{align*}
    \sum_{w\in S_3} w \cdot [T_{m_3}(\lambda_3),& [T_{m_2}(\lambda_2), T_{m_1+1}(\lambda_1)]] = 0, \tag{\rellab{W:g}{g}}\\
    \psi_n(\lambda)=0\text{ if }&2n-2+\deg(\lambda)<0.\tag{\rellab{W:h}{h}}
\end{align*} 
Recall that $[-,-]$ and $\{-,-\}$ are as in  \eqref{supcom}.
The expression 
$\{T_m, T_n\}(c_1\Delta_S\lambda \mu)$ above
is the anti-super-commutator of $T_m$ and $T_n$, whose arguments are taken from the symmetric 2-tensor 
$$c_1\Delta_S\lambda \mu\in H^*(S,\Q)\otimes H^*(S,\Q).$$
The generators have the following degrees
\begin{equation}\label{E:W-bigrading}
  \deg(T_n(\lambda))=(1,2n-2+\deg(\lambda)), \quad \deg(\psi_n(\lambda))=(0, 2n-2+\deg(\lambda))
  \end{equation}
and $\deg(\c)=(0,0)$.  

Let $W^\geq(S)[m,n]$ be the subspace spanned by all bidegree $(m,n)$ elements.
We define $W^+(S)$ and $W^0(S)$ to be the graded subalgebras generated by 
$$\{T_n(\lambda)\,:\,n \geq 0\,,\,\lambda \in H^*(S,\Q)\}
,\quad
\{\c, \psi_n(\lambda)\,:\,n \geq 0\,,\,\lambda \in H^*(S,\Q)\}.$$
\end{definition}
Note that $W^0(S)$ is a supercommutative algebra generated by elements indexed by pairs $(n,\lambda)$. 
It will be convenient to identify\footnote{This identification is why we included the central element $\c$ in $W^\geq(S)$ rather than just in $W(S)$.} $\Lambda(S)$ and $W^0(S)$ as follows. Set 
\begin{align*}
    \psi_\lambda(z)&=\sum_{n \geq 0} \psi_n(\lambda) \frac{z^n}{n!},\\
    \underline{\psi}(z)&=z^{-1}\c \otimes 1 +\sum_\lambda \psi_\lambda(z) \otimes \lambda^* \in W^0(S) \otimes H^*(S,\Q)[[z]].
\end{align*}
There is a unique graded algebra isomorphism 
$$i: W^0(S) \simeq \Lambda(S)
,\quad\underline{\psi}(z) \mapsto z^{-1}\mathbf{r} \otimes 1 +z^{-1}(\uch(z)-\mathbf{r} \otimes 1)(1\otimes \Td_S(z)).$$
Through this identification, we may consider the elements 
$$p(\lambda) \in W^0(S)
,\quad
p\in\Lambda'
,\quad
\lambda \in H^*(S,\Q).$$

\smallskip

\begin{theorem}\label{Prop:defHeis}
\hfill
\begin{enumerate}[label=$\mathrm{(\alph*)}$,leftmargin=8mm]
\item The graded character of $W^+(S)$ is given by
\begin{equation}\label{eq:Wplus-char}
    P_{W^+(S)}(z,w)=\sum_{l,n} \dim(W^+(S)[l,n])(-z)^nw^l=\Exp\left(\frac{P_S(z)z^{-2}w}{(1-z^2)(1-w)} \right).    
\end{equation}    
\item 
The elements $D_{m,0}(\lambda)$ for $m \geq 1$ and  $\lambda \in H^*(S,\Q)$ generate a free graded commutative polynomial algebra $\mathfrak{h}^+_S$.
\end{enumerate}
\end{theorem}

\smallskip

\begin{definition}\label{def:W-semidef}
Assume that $c_1=0$ and $s_2=q^2$ for some $q \in H^2(S,\Q)$.
We define the Lie algebra $\mathfrak{w}^\geq(S)$ to be spanned by $z^mD^n\lambda$ with $m,n\in\N$ and $\lambda\in H^*(S,\Q)$, and a central element $\c$, with the Lie bracket given by
\begin{equation}\label{E:Lie-rels-semidef}
    [z^mD^n\lambda\,,\,z^aD^b\mu] = \frac{ (D+aq)^nD^b - D^n(D+mq)^b }{q}\,z^{m+a}\lambda\mu.  
\end{equation}
Let $\mathfrak{w}^+(S)$ be the subalgebra spanned by $z^mD^n\lambda$ for all $m \geq 1,$ $ n \geq 0$ and $\lambda \in H^*(S,\Q)$.
\end{definition}

\begin{remark}\label{rmk:semidef-lin-param}
    The right hand side of~\eqref{E:Lie-rels-semidef} is 
    \[
        \sum_{i=1}^n \binom{n}{i}n^iz^{m+a}D^{n+b-i}q^{i-1}\lambda\mu - \sum_{j=1}^b \binom{b}{j}m^jz^{m+a}D^{n+b-j}q^{j-1}\lambda\mu.
    \]
The $q$-term of this sum only depends on $(m+a)$ and $(mb-na)$ in accordance with the last claim of Proposition~\ref{prop:W-PBW} below.
\end{remark}

When $c_1=s_2=0$, we write $D_{m,n}(\lambda)=z^mD^n\lambda$. Then, the Lie bracket degenerates as follows:
\begin{align}\label{E:Lie-rels}
    [D_{m,n}(\lambda), D_{m',n'}(\mu)] = (n m' - m n') D_{m+m',n+n'-1}(\lambda\mu).
\end{align}
We will denote this Lie algebra by $\mathfrak{w}^\geq_0(S)$. It is well-defined for any surface $S$.

\smallskip

\begin{theorem}\label{T:W(S)} If $c_1=0$ and $s_2=q^2$, then the assignment $T_n(\lambda) \mapsto zD^n\lambda$ extends to isomorphisms 
$$\Phi:W^\geq(S) \to U(\mathfrak{w}^\geq(S))
,\quad
\Phi:W^+(S) \to U(\mathfrak{w}^+(S)).$$
\end{theorem}

\smallskip

We will prove Theorem~\ref{Prop:defHeis} in \S\S\ref{sec:str-def},~\ref{sec:Heis-alg}, and Theorem~\ref{T:W(S)} in \S\ref{sec:semidef}.

\medskip

\subsection{Structure of $W^\geq (S)$ in the non-deformed case}
In this section we assume that $c_1=s_2=0$.
So, the relation~\eqref{W:f} simplifies to
\begin{align}\label{W-rel6nondef}
    [T_{m}(\lambda), T_{n+3}(\mu)] &- 3 [T_{m+1}(\lambda), T_{n+2}(\mu)] + 3 [T_{m+2}(\lambda), T_{n+1}(\mu)] - [T_{m+3}(\lambda), T_{n}(\mu)] = 0.
\end{align}

\begin{theorem}\label{thm:W-undef}
If $c_1=s_2=0$, then the assignment 
\begin{align}\label{E:iso-undef}
\psi_n(\lambda)\mapsto D_{0,n}(\lambda)
,\quad
T_n(\lambda)\mapsto D_{1,n}(\lambda)
,\quad
\c \mapsto \c
\end{align}
extends to an algebra isomorphism 
$$\Phi: W^\geq(S)\to U(\mathfrak{w}_0^\geq(S)).$$
\end{theorem}

\begin{proof} We abbreviate $T_n=T_n(1)$ and $\psi_n=\psi_n(1)$.
The elements $D_{0,n}(\lambda)$, $D_{1,n}(\lambda)$ 
satisfy the relations \eqref{W:a}-\eqref{W:g}.
Hence the assignment above yields an algebra homomorphism.
Set
$$\tilD_{m,n}(\lambda) = \frac{n!}{(m+n)!} (-\Ad_{T_0})^m \psi_{m+n}(\lambda).$$
We claim that the following defines an inverse homomorphism $\Psi$ to the map $\Phi$
\begin{align}\label{E:inv-undef}
        \Psi: D_{m,n}(\lambda) \mapsto \tilD_{m,n}(\lambda) .
\end{align}
To prove the theorem, we need to check that the elements $\tilD_{m,n}(\lambda)$ satisfy the defining relations~\eqref{E:Lie-rels}. They are satisfied for $m=n=0$. We begin by computing a few commutators of low rank.
Since $T_n$ is even, from \eqref{W:e} we deduce that
    $$[T_n,T_n(\lambda)]=[T_n(\lambda),T_n]=0$$
Unraveling the definition of $\tilD_{m,n}(\mu)$ we obtain
    \begin{align}\label{E:T0-comm}
        [T_0(\lambda),\tilD_{m,n}(\mu)] = -n\tilD_{m+1,n-1}(\lambda\mu), \quad n\geq 1.
    \end{align}
In the same way, the relations \eqref{W:d}, \eqref{W:c} imply that 
$$[\psi_1(\lambda), T_0(\mu)] = T_0(\lambda\mu)
,\quad
[\psi_1(\lambda),\psi_n(\mu)] = 0.$$ So we get
    \begin{align}\label{E:psi1}
        [\psi_1(\lambda),\tilD_{m,n}(\mu)] = m\tilD_{m,n}(\lambda\mu).
    \end{align}
Let us now consider the commutators $[T_m(\lambda),T_n(\mu)]$. By \eqref{W:e}, it's enough to assume that $\lambda=1$. 
Using \eqref{W-rel6nondef} first for $m \in \{n,n \pm 1\}$ and then successively for $m>n$ and $m<n$ we deduce that for any $n$, the commutators $[T_m,T_{n-m}(\mu)]$ for $m=0, 1, \ldots, n$ are all proportional to one another, hence, say to $[T_n,T_0(\mu)]=n\widetilde{D}_{2,n-1}(\mu)$. But since $\Phi(\widetilde{D}_{2,m}(\mu))=D_{2,m}(\mu)$, we deduce that
$$[T_m(\lambda),T_{n-m}(\mu)]=[T_m,T_{n-m}(\lambda\mu)]=(2m-n)\widetilde{D}_{2,n-1}(\lambda\mu)$$
as expected. Further taking commutators with $\psi_m(\mu) = \tilD_{0,m}(\mu)$, we get
    \begin{align}\label{E:psi-D2}
        [\tilD_{2,n}(\lambda),\tilD_{0,m}(\mu)] = -2m\tilD_{2,n+m-1}(\lambda\mu).
    \end{align}    
Next, we consider commutators with $\tilD_{2,0}$. Relation \eqref{W:g} yields
$$[\tilD_{2,0},T_0] = [[T_1,T_0],T_0] = 0,$$ and more generally 
$$[\tilD_{2,0}(\lambda),T_0(\mu)] = 0.$$ 
Using~\eqref{E:T0-comm}, we deduce from~\eqref{E:psi-D2} by induction on $m$ that
    \begin{align}\label{E:D20}
        [\tilD_{2,0}(\lambda),\tilD_{m,n}(\mu)] &= -\frac{1}{n+1}[T_0,[\tilD_{2,0}(\lambda),\tilD_{m-1,n+1}(\mu)]] \\\nonumber
        &= 2[T_0,\tilD_{m+1,n}(\lambda\mu)] \\\nonumber
        & = -2n \tilD_{m+2,n-1}(\lambda\mu)
    \end{align}
for any $n \geq 1$. 
Finally, from relations \eqref{E:D20}, \eqref{E:T0-comm}, \eqref{E:psi1} and the definition of $\tilD_{m,n}$ we deduce in turn the following equalities 
    \begin{align}\label{E:Lie-rels-red-T1}
        [T_1(\lambda),\tilD_{m,n}(\mu)] &= (m-n)\tilD_{m+1,n}(\lambda\mu),\\\label{E:Lie-rels-red-psi2} [\psi_2(\lambda),\tilD_{m,n}(\mu)] &= 2m \tilD_{m,n+1}(\lambda\mu)
    \end{align}
    for all $m,n\geq 0$. 
Assume now that the relations \eqref{E:Lie-rels} hold (with $\widetilde{D}$ in place of $D$) for a given pair $(m,n)$ and all pair $(m',n')$ with $n'\geq 1$. Then, by applying $\Ad_{T_1}$ or $\Ad_{T_0}$, resp. $\Ad_{\psi_2}$ and using \eqref{E:Lie-rels-red-T1} or \eqref{E:T0-comm}, resp.~\eqref{E:Lie-rels-red-psi2}, we deduce that the same is true for the pairs $(m+1,n)$, resp. $(m,n+1)$. Starting from $m=1,n=0$, we deduce that \eqref{E:Lie-rels} holds for any pairs $(m,n)$, $(m',n')$ with $m \geq 1$ and $n'\geq 1$, and thus, by symmetry, whenever $m'\geq 1$ and $n\geq 1$. The only remaining cases occur when $m=m'=0$ (for which \eqref{E:Lie-rels} trivially holds) and when $n=n'=0$, which we now deal with. We will prove by induction on $m+m'$ that 
    \begin{equation}\label{E:dm0-comm}
        [\tilD_{m,0}(\lambda),\tilD_{m',0}(\mu)]=0.
    \end{equation}
    Fix $s$ and assume that \eqref{E:dm0-comm} holds whenever $m+m'\leq s$. Note that the above calculations show that this is indeed the case for $s=3$. If $m,m'>0$ satisfy $m+m'=s+1$ then by the induction hypothesis and 
what we've already established we have
\begin{align*}
[T_0,\tilD_{m-1,1}(\lambda)]&=-\tilD_{m,0}(\lambda), \\
[T_0,\tilD_{m'-1,1}(\mu)]&=-\tilD_{m',0}(\mu),\\
[T_0,\tilD_{m,0}(\lambda)]&=0,\\ 
[T_0,\tilD_{m',0}(\mu)]&=0,\\
\frac{1}{m'}[\tilD_{m-1,1}(\lambda),\tilD_{m',0}(\mu)]&=\tilD_{m+m',0}(\lambda\mu)\\
&=-\frac{1}{m}[\tilD_{m,0}(\lambda),\tilD_{m'-1,1}(\mu)].
\end{align*}
    Applying $\Ad_{T_0}$ to this last equation we deduce $[\tilD_{m,0}(\lambda),\tilD_{m',0}(\mu)]=0$. 
    We are done.
    \end{proof}

\medskip

\subsection{Order filtration}\label{sec:order-filt}
Let us return to the deformed case, with $S$ a projective surface.
Let $F_n$, $n\in\Z$, be the smallest increasing filtration of $W^\geq(S)$
such that $\c \in F_0$ and
\begin{itemize}[leftmargin=6mm]
\item[-]
$\psi_n(\lambda), T_n(\lambda)\in F_n$ for all $n\in\N$, $\lambda\in H^*(S,\Q)$,
\item[-] 
$F_nF_{n'} \subset F_{n+n'}$,
\item[-]
$[F_n,F_{n'}]\subset F_{n+n'-1}$.
\end{itemize}
We call $F_\bullet$ the order filtration. 
The algebra $W^\geq(S)$ is bigraded by~\eqref{E:W-bigrading}. Set 
$$F[m]_n = F_n\cap W^\geq(S)[m,-].$$
The filtration $F_\bullet$ can be given more explicitly as follows. 
By a Lie word we mean a combination of Lie brackets of the
generators of $W^\geq(S)$. By a monomial we mean a product
of Lie words. The order of a Lie word is the sum
of the indices of the generators minus the number of
brackets. The order of a monomial is the sum of the orders
of the Lie words. 
Then $F_n$ is the span of expressions of order $\leq n$.
The relations of $W^\geq(S)$ are not homogeneous for the order.
We define
\begin{align}\label{orderfilt}\Gr_\bullet W^\geq(S) = \sum_{n\in\Z} F_n/F_{n-1}.
\end{align}
The graded vector space $\Gr_\bullet W^\geq(S)$
has two operations: a graded commutative multiplication and a Lie bracket of degree $-1$.

\begin{lemma}\label{lem:filtr-bound}
We have $F[0]_{-1} = 0$ and $F[m]_{-m} = 0$ for any $m>0$.
\end{lemma}

\begin{proof}
We must show that the order $n$ of any non-zero monomial of degree $m$ is 
$n>-m$ if $m>0$ and $n>-1$ if $m=0$.
Since degrees and weights are additive for products, it is enough to prove the statement for generators and Lie words.
The generators have order $\geqslant 0$, so they satisfy the claim.
Next, take a non zero Lie word of the form $[f,g]$ for two Lie words $f$, $g$ of degrees $m$, $m'$ 
and order $n$, $n'$ respectively.
We prove the claim by induction.
The order of $[f,g]$ is $n+n'-1$, its degree is $m+m'$.
If $m,m'>0$, then $n>-m$ and $n'>-m'$,
hence $n+n'-1 > -m-m'$.
Now suppose $m=0$.
We have $W^\geq(S)[0,-]= \Lambda(S)$. 
Hence, the algebra $W^\geq(S)[0,-]$ is super-commutative.
Thus $f$ cannot be a Lie bracket. So we have $f=\psi_n(\lambda)$ for some $\lambda\in H^*(S,\Q)$.
Since $\psi_0(\lambda)$ is central, we have $n>0$.
Hence $n+n'-1\,\geq\, n'> -m' = -m-m'$.
\end{proof}

\begin{proposition}\label{prop:W-PBW}
There are $D_{m,n}(\lambda)\in W^\geq(S)$ for each $m,n\in\N$ and
$\lambda\in H^*(S,\Q)$, such that
\begin{enumerate}[label=$\mathrm{(\alph*)}$,leftmargin=8mm]
\item
$D_{0,n}(\lambda) = \psi_n(\lambda)$, $D_{1,n}(\lambda) = T_n(\lambda)$,
\item 
$F_{-1} = 0$ and $F_n$ is spanned by all products $D_{m_1,n_1}(\lambda_1)\cdots D_{m_k,n_k}(\lambda_k)$ with $\sum_i n_i\leq n$,
\item 
the relation~\eqref{E:Lie-rels} holds modulo $F_{n+n'-3}$. \end{enumerate}
\end{proposition}

\begin{proof}
Consider the Lie algebra $\mathfrak{g}$ generated by the elements $\psi_n(\lambda), $ $T_n(\lambda)$ for $n \geq 0, \lambda \in H^*(S,\Q)$.
The relations \eqref{W:a}-\eqref{W:e}, \eqref{W:g} hold in $\mathfrak{g}$, as well as the simplified version~\eqref{W-rel6nondef} of the relation \eqref{W:f}.
By Theorem~\ref{thm:W-undef} the Lie algebra $\mathfrak{g}$ is spanned by the elements $D_{m,n}(\lambda)$, and therefore $\Gr_\bullet W^\geq(S)$ is generated by these elements as an algebra.
In particular, we have $F_n/F_{n-1} = 0$ for $n<0$, hence Lemma~\ref{lem:filtr-bound} implies that $F_{-1} = 0$.
Further $F_n/F_{n-1}$ is spanned by all products $D_{m_1,n_1}(\lambda_1)\cdots D_{m_k,n_k}(\lambda_k)$ with $\sum_i n_i = n$.
Therefore the first two claims of the theorem are satisfied with any lift of the elements $D_{m,n}(\lambda)$ in $\mathfrak{g}$ to $W^\geq(S)$.
We also get a weak version of last claim, i.e., 
the identity~\eqref{E:Lie-rels} holds modulo $F_{n+n'-2}$.
To prove this identity modulo $F_{n+n'-3}$, we repeat the argument above with the algebra $\bigoplus_{n\in\Z} F_n/F_{n-2}$.
\end{proof}

\begin{definition}\label{def:W-graded}
Let $W^\geq_0(S)$ be the algebra generated by $\psi_n(\lambda)$, $T_n(\lambda)$ for all $n, \lambda$ subject to the relations \eqref{W:a}-\eqref{W:e}, ~\eqref{W-rel6nondef} and \eqref{W:g}.
\end{definition}

Recall the Lie algebra $\mathfrak{w}_0^{\geq}(S)$ in \eqref{E:Lie-rels}. 
The proof of Theorem~\ref{thm:W-undef} yields 
$$U(\mathfrak{w}^\geq_0(S))=W^\geq_0(S)$$ 
Thus Proposition~\ref{prop:W-PBW} yields a surjective algebra homomorphism
\begin{equation}\label{E:w0toGrW}
\rho:\Sym(\mathfrak{w}^\geq_0(S)) \twoheadrightarrow \text{Gr}_\bullet W^\geq(S)
\end{equation}
and a Lie algebra homomorphism 
\begin{equation}\label{E:w0toGrWLie}
\rho': \mathfrak{w}^\geq_0(S) \rightarrow \text{Gr}_\bullet W^\geq(S).
\end{equation}
The Lie bracket on $\text{Gr}_\bullet W^\geq(S)$ has order $-1$.

\medskip

\subsection{Deformed $W$-algebras}\label{sec:defoWalg}
We now define a doubled version $W^{(\c)}(S)$ of $W^\geq(S)$. 
Let 
\begin{align}\label{theta}
\begin{split}
\delta(z)=\sum_{n \in \Z} z^n,\quad
\theta(z)=\sum_{n \geq 0}h_nz^n,
\quad
\widetilde{\theta}(x)=\tau_{c_1}(\theta(x)).
\end{split}
\end{align}

\medskip

\begin{definition}\label{def:Wdoublenoncompact} The algebra $W^{(\c)}(S)$ is called 
the deformed $W$-algebra of $S$. It is generated by elements $\psi_n(\lambda), T^\pm_n(\lambda)$ for $n\geq 0$, $\lambda \in H^*(S,\Q)$ and central elements $\mathbf{c}, e^{\pm i\pi \mathbf{c}}$ with the following relations:
\begin{itemize}[leftmargin=8mm]
\setlength\itemsep{.3em}
\item[-] relations \eqref{W:a}, \eqref{W:b}, \eqref{W:c} and \eqref{W:e} of \S 3.1 with $T^\pm_n(\lambda)$ in place of $T_n(\lambda)$,
\item[-] $[\psi_m(\lambda),T^\pm_n(\mu)]=\pm m T^\pm_{n+m-1}(\lambda \mu)$,
\item[-] the assignments $T_n(\lambda) \mapsto T^+_n(\lambda)$ and $T_n(\lambda) \mapsto T^-_n(\lambda)$
extend to a morphism and an anti-homomorphism $W^{+}(S) \to W^{(\c)}(S)$,
\item[-] the double relation, which is best expressed in terms of generating series
\begin{equation}\label{E:+-relationsdouble}
        \big[T^+_{\lambda}(z), T^-_{\mu}(w)\big]=-\exp(i\pi\c)\left[\frac{1}{c_1z} \left( 1-\frac{\theta(z)}{\widetilde{\theta}(z)}\right) \delta\Big(\frac{w}{z}\Big)(\lambda\mu)\right]_{++}
    \end{equation}
    where $T^{\pm}_\lambda(z)=\sum_{n\geq 0} T_n^\pm(\lambda)z^n$ and  $A(x,y)_{++}$ is the truncation of a power series to its terms $x^ay^b$ with $a,b \geq 0$.
We may rewrite this relation as follows:
\begin{equation}\label{E:doublerelation}
    [T_m^+(\lambda),T_n^-(\mu)] = -\exp(i\pi\c)\sum_{0\leq j\leq l\leq m+n} (-1)^l\binom{\c-l+j}{j+1}h_{m+n-l}e_{l-j}(c_1^{j}\lambda\mu).
\end{equation}
\end{itemize}
\end{definition}

\smallskip

Let $W^+(S),$  $W^-(S)$, $W^0(S)$ be the subalgebras generated by $\c$ and $T^+_n(\lambda)$,  $T^-_n(\lambda)$, $\psi_n(\lambda)$ respectively, for $n \geq 0$ and $\lambda \in H^*(S,\Q)$. 
We likewise define $W^\geq(S)$ and $W^\leq(S)$. The algebra
 $W^{(\c)}(S)$ is $\Z^2$-graded with
$$\deg(T^\pm_n(\lambda))=(\pm 1, 2n-2+\deg(\lambda)), \qquad \deg(\psi_n(\lambda))=(0,2n-2+\deg(\lambda))$$ and
$\deg(\c)=(0,0)$. Let $W^{(\kappa)}(S)$ be the specialization of $W^{(\c)}(S)$ to $\c=\kappa$ and $e^{i\pi\mathbf{c}}=e^{i\pi\kappa}$.

\begin{remark}\leavevmode\nolisttopbreak
   \begin{enumerate}[label=(\roman*),leftmargin=8mm]
        \item[(a)] The elements $T^-_{n}(\lambda)$ satisfy the following sign-corrected version of \eqref{W:f}
\[ \begin{split}[T_{m}(\lambda), T_{n+3}(\mu)] 
&-3 [T_{m+1}(\lambda), T_{n+2}(\nu)] + 3 [T_{m+2}(\lambda), T_{n+1}(\mu)]  - [T_{m+3}(\lambda), T_{n}(\mu)]\\
 &- [T_{m}(\lambda), T_{n+1}(s_2 \mu)] + 
 [T_{m+1}(\lambda), T_{n}(s_2 \mu)] -\{T_m, T_n\}(c_1\Delta_S\lambda \mu)=0.\end{split} \tag{f'}\]
        \item[(b)] If $c_1=0$ then, up to the factor $-\exp(i\pi\c)$, the right hand side of \eqref{E:doublerelation} reduces to
\begin{equation*}
\begin{split}
    \sum_{0\leq k\leq n+m}& (\c-k)(-1)^kh_{n+m-k}e_{k}(\lambda\mu) \\&=\c\delta_{n+m,0}(\lambda\mu) -\sum_{0\leq k\leq n+m} k(-1)^k h_{n+m-k}e_{k}(\lambda\mu) =p_{n+m}(\lambda\mu).
\end{split}
\end{equation*}
If in addition we have $s_2=0$ then we also have $\Td_S(z)=1$ and $p_{l}(\lambda)=l\psi_{l-1}(\lambda)$.\smallskip

    \item[(c)]
    In all of our cases of interest $\c$ will be specialized to a non-negative integer $r$, so that $\exp(i\pi\c) = (-1)^{r}$ and $W^{(r)}(S)$ is defined over $\Q$.
    \end{enumerate}
\end{remark}

\begin{proposition}\label{prop:big-W-triang}
    The obvious maps $W^\geq(S)^\op\to W^{(\c)}(S)$ and $W^\geq(S)\to W^{(\c)}(S)$ are embeddings of algebras.
    The multiplication $W^-(S)\otimes W^0(S)\otimes W^+(S)\to W^{(\c)}(S)$ is a linear isomorphism.
\end{proposition}
\begin{proof}
    Analogous to~\cite[Appendix A]{Tsymbaliuk}. 
\end{proof}

\begin{remark} The same proof as in~\cite[Appendix A]{Tsymbaliuk} shows that $W^+(S)$ is isomorphic to the algebra generated by the elements $T_n(\lambda)$ subject to the relations \eqref{W:b}, \eqref{W:e}, \eqref{W:f} and \eqref{W:g}. 
\end{remark}

\medskip

Set $D_{i,-1}= \delta_{i,0}\c$. Let 
$\mathfrak{w}(S)$ be the Lie algebra  generated by the elements $D_{m,n}(\lambda)$ with $m\in \Z$, $n\in\N$, $\lambda\in H^*(S,\Q)$, whose Lie bracket is given by~\eqref{E:Lie-rels}.

\smallskip

\begin{theorem}\label{thm:big-W-undef}
If $c_1=s_2=0$, then
there is an algebra isomorphism $$\Phi: \DW{\c}(S)\to U(\mathfrak{w}(S))$$ which sends $\psi_n(\lambda)$ to $D_{0,n}(\lambda)$, $T^+_n(\lambda)$ to $D_{1,n}(\lambda)$, and $T^-_n(\lambda)$ to $\exp(i\pi\c)D_{-1,n}(\lambda)$.
\end{theorem}
\begin{proof}
    Let $\mathfrak{w}^+(S)$ and  $\mathfrak{w}^-(S)$ be the Lie subalgebras of $\mathfrak{w}(S)$ spanned by $D_{m,n}(\lambda)$ with $m>0$ and $m<0$ respectively.
    The restriction of $\Phi$ to $W^\pm$ yields algebra isomorphisms $W^\pm(S)\simeq U(\mathfrak{w}^\pm(S))$ by Theorem~\ref{thm:W-undef}. 
    An inductive argument analogous to the proof of Proposition~\ref{prop:heis-in-double} below shows that the commutation relations between $\mathfrak{w}^+(S)$ and $\mathfrak{w}^-(S)$ hold in $\DW{\c}(S)$ as well. 
    We conclude by Proposition~\ref{prop:big-W-triang} and the PBW theorem.
\end{proof}

\medskip

\subsection{Heisenberg subalgebra}\label{sec:Heis-alg}
The Heisenberg algebra  of $S$ is the Lie algebra $\mathfrak{h}_S$ generated by
\[
    \left\{ \mathfrak{q}_{n}(\lambda), C \,:\, n\neq 0, \lambda\in H^*(S,\Q) \right\}
\]
with the relation $\mathfrak{q}_n(\lambda+\mu)=\mathfrak{q}_n(\lambda) + \mathfrak{q}_n(\mu)$ for any $n$ and $\lambda,\mu$ and the Lie bracket 
\begin{align}\label{E:Heis-bracket}
    [\mathfrak{q}_{m}(\lambda),\mathfrak{q}_{l}(\mu)] = m\delta_{-m,l}C\int_S\lambda\mu
    ,\quad 
    \text{$C$ is central}.
\end{align}
Recall the elements $D_{m,n}(\lambda)\in W^\geq(S)$ from Proposition~\ref{prop:W-PBW}. 
Let us consider the homomorphism 
$$\Theta:W^+(S)\to W^-(S),\quad T_n^+(\lambda)\mapsto (-1)^nT_n^-(\lambda).$$ 
We set $D_{-m,n}(\lambda)=\Theta(D_{m,n}(\lambda))$.
Proposition \ref{prop:W-PBW} and Relation~\eqref{E:Lie-rels} imply that elements $D_{m,0}(\lambda)$ with $m\geq 1$ and $\lambda \in H^*(S,\Q)$ super-commute with each other, and are uniquely determined by the following formula
\begin{equation}\label{E:relinWforHeis}
    D_{m+1,0}(\lambda) = \frac{1}{m}[D_{1,1}(1),D_{m,0}(\lambda)].    
\end{equation}

\begin{proposition}\label{prop:heis-in-double}
    The assignment 
    $$\mathfrak{q}_{m}(\lambda)\mapsto D_{m,0}(\lambda), \quad\mathfrak{q}_{-m}(\lambda)\mapsto -\exp(-i\pi\c)D_{-m,0}(\lambda),\quad C\mapsto \c,\quad m >0$$ 
    defines an algebra  homomorphism $U(\mathfrak{h}_S)\to \DW{\c}(S)$.
    In particular, Theorem~$\ref{Prop:defHeis}\operatorname{(b)}$ holds.
\end{proposition}

\begin{proof}
Let $L_{\pm} = D_{\pm 1,1}(1)$ and $L_0 = \psi_1(1)$.
    By definition, we have 
$$[L_\pm,\mathfrak{q}_{\pm m}(\lambda)] = m\mathfrak{q}_{\pm(m+1)}(\lambda).$$
The following equalities follow from the relation \eqref{E:doublerelation}
    \begin{align}\label{E:heisrel}
\begin{split}
        [\mathfrak{q}_1(\lambda),\mathfrak{q}_{-1}(\mu)] &= \c\int_S \lambda\mu, \\
        [L_{\pm},\mathfrak{q}_{\mp 1}(\lambda)] &= \pm \psi_0(\lambda) \mp \binom{\c}{2}\int_S\lambda c_1,\\
        [L_0,L_\pm] &= \pm L_\pm,\\ 
[L_{+},L_-] &= 2L_0 - \c\psi_0(c_1) + \binom{\c}{3}\int_S c_1^2.
\end{split}
    \end{align}
In view of Proposition~\ref{prop:W-PBW}, it suffices to check the relation
$$[\mathfrak{q}_{-m}(\lambda),\mathfrak{q}_{n}(\mu)]=n\c\delta_{m,n}(\lambda,\mu)
,\quad
m,n>0.$$
    We proceed by induction.
    First, we prove that $[L_0, \mathfrak{q}_{\pm n}] = \pm n\mathfrak{q}_{\pm n}$. Indeed,
    \begin{align*}
        [L_0, \mathfrak{q}_{n+1}] 
        &= \frac{1}{n}[L_0,[L_+,\mathfrak{q}_n]]\\
        &= - \frac{1}{n}\left( [\mathfrak{q}_n,[L_0,L_+]] + [L_+,[\mathfrak{q}_n,L_0]] \right)\\
        &= \mathfrak{q}_{n+1} + [L_+,\mathfrak{q}_n] \\
        & = (n+1)\mathfrak{q}_{n+1}.
    \end{align*}
    Next, we prove that $[L_{\pm},\mathfrak{q}_{\mp n}(\lambda)] = \mp n\mathfrak{q}_{\mp (n-1)}$ for $n>1$.
    Indeed,
    \begin{align*}
        [L_+,\mathfrak{q}_{-(n+1)}(\lambda)]
        &= \frac{1}{n}[L_+,[L_-,\mathfrak{q}_{-n}(\lambda)]]\\
        &= - \frac{1}{n}\left( [\mathfrak{q}_{-n},[L_+,L_-]] + [L_-,[\mathfrak{q}_{-n},L_+]] \right)\\
        &= -2\mathfrak{q}_{-n} - [L_-,\mathfrak{q}_{-(n-1)}]\\
        &= -(n+1)\mathfrak{q}_{-n}.
    \end{align*}
    It is easy to see that $[\mathfrak{q}_{\pm 1}(\lambda), \mathfrak{q}_{\mp n}(\mu)] = 0$ for $n>1$. Indeed,
    \begin{align*}
        [\mathfrak{q}_{-1}(\lambda), \mathfrak{q}_{n+1}(\mu)]
        &= \frac{1}{n}[\mathfrak{q}_{-1}(\lambda),[L_+,\mathfrak{q}_{n}(\mu)]]\\
        &= - \frac{1}{n}\left( [L_+,[\mathfrak{q}_{n}(\mu),\mathfrak{q}_{-1}(\lambda)]] + [\mathfrak{q}_{n}(\mu),[\mathfrak{q}_{-1}(\lambda),L_+]] \right) \\
        &= 0.
    \end{align*}
    Finally, for any positive $m$, $n$ we have by induction
    \begin{align*}
        [\mathfrak{q}_{-(m+1)}(\lambda), \mathfrak{q}_{n+1}(\mu)]
        &= \frac{1}{n}[\mathfrak{q}_{-(m+1)}(\lambda),[L_+,\mathfrak{q}_{n}(\mu)]]\\
        &= \frac{1}{n}[L_+,[\mathfrak{q}_{-(m+1)}(\lambda),\mathfrak{q}_{n}(\mu)]] + \frac{m+1}{n}[\mathfrak{q}_{-m}(\lambda),\mathfrak{q}_{n}(\mu)]\\
        &= \c(m+1)\delta_{m,n}\int_S \lambda\mu,
    \end{align*}
    which proves the desired statement.
\end{proof}

\begin{remark}
    Note that \eqref{E:heisrel} contains central terms which do not appear in~\cite[Theorem 3.3]{Lehn}. 
    Indeed, these terms vanish for the Hilbert schemes, since in this case $\c=1$ and $\psi_0(\lambda)=0$ for all $\lambda\in H^{\geq 2}(S)$, see \S\ref{sec:Hilbert}.
\end{remark}

\medskip

The relations (c), (d) imply that $D_{0,0}(\lambda) = \psi_0(\lambda)$ is central in $\DW{\c}(S)$ for any $\lambda\in H^{\geq 2}(S)$.
Let $Z(S)$ be the super-commutative subalgebra generated by $D_{0,0}(\lambda)$'s.
Let $\DW{\c}_\red(S)$ be the quotient of $\DW{\c}(S)$ by the two-sided ideal generated by $Z(S)$.

\begin{lemma}\label{lem:Heis-faithful}\hfill
\begin{enumerate}[label=$\mathrm{(\alph*)}$,leftmargin=8mm,itemsep=1.2mm]
\item 
If $I\subset \DW{\c}_\red(S)$ is a two-sided ideal such that $I\cap U(\mathfrak{h}_S)= \{0\}$, then $I= \{0\}$. 
\item 
If $I^+\subset W^+(S)$ is a two-sided ideal with $I^+\cap U(\mathfrak{h}^+_S )= \{0\}$, then $I^+=\{0\}$.
\end{enumerate}
\end{lemma}

\begin{proof}
The proofs of the claims being analogous, we will concentrate on the first one. We follow the proof of~\cite[Lemma F.7]{SVIHES}.
The algebra $\DW{\c}_\red(S)$ admits the order filtration as in \S\ref{sec:order-filt}. 
Recall that $\Gr_\bullet \DW{\c}_\red(S)$ is a graded super-commutative algebra, equipped with a Lie bracket of degree $-1$.
Let $\Gr_\bullet I\subset \Gr_\bullet \DW{\c}_\red(S)$ be the associated graded 
of $I$ with respect to the induced filtration.
Using Theorem~\ref{thm:big-W-undef} instead of Theorem~\ref{thm:W-undef}, we can repeat the proof of Proposition~\ref{prop:W-PBW} to get an algebra surjection 
$$\nu:\mathbb{Q}[D_{m,n}(\lambda)]\to \Gr_\bullet \DW{\c}_\red(S).$$
Set $J =\nu^{-1}(\Gr_\bullet I)$. 
It is enough to show that $J=\{ 0\}$.     
The ideal $J$ is graded by the weight. Let $x\in J$ be a non-zero element of minimal weight $n$. 
Since $U(\mathfrak{h}_S)\subset \Gr_0 \DW{r}_\red(S)$, we have $J\cap U(\mathfrak{h}_S)=\{0\}$.
Hence $n>0$. We write
\begin{equation*}
x = \sum_{i} c_i\prod_{j}D_{m_{ij},n_{ij}}(\lambda_{ij})
\quad, \quad 
\sum_j n_{ij} = n.
\end{equation*}
Assume that for each $i$, $j$ we have either $n_{ij}>n_{i,j+1}$, or $n_{ij}=n_{i,j+1}$ and $m_{ij}\geq m_{i,j+1}$.
Let $n_i = \{n_{i1}\geq n_{i2}\geq\ldots\}$. 
Let $\overline{\iota}$ be the index of the maximal tuple among all $n_i$'s,
with respect to the lexicographic order. 
We write 
$$n_{\overline{\iota}} = \{\overline{n}_1\geq \ldots \geq \overline{n}_s\}
,\quad
m_{\overline{\iota}} = \{\overline{m}_1\geq \ldots \geq \overline{m}_s\}.$$
The space $\mathbb{Q}[D_{m,n}(\lambda)]$ is equipped with the Lie bracket 
given by~\eqref{W-rel6nondef} and Leibniz rule. Consider the operators $\sigma_l = \Ad(D_{l,0}(1))$ of degree $-1$.
We have
\begin{equation}\label{eq:sigma-explicit}
\begin{split}
\sigma_l(D_{m_1,n_1}&(\lambda_1)\cdots D_{m_k,n_k}(\lambda_k))\\
&= 
-l\sum_{i=1}^k n_iD_{m_1,n_1}(\lambda_1)\cdots D_{m_i+l,n_i-1}(\lambda_i)\cdots D_{m_k,n_k}(\lambda_k).
\end{split}
\end{equation}
The ideal $J$ is preserved by the action of the operators $\sigma_l$. 
Fix $l>\max \{m_{ij}\}$. Let us compute the coefficient in $\sigma_l(x)$
of the monomial $$D_{\overline{m}_1+l,\overline{n}_1-1}(\lambda_1)D_{\overline{m}_2,\overline{n}_2}(\lambda_2)\cdots D_{\overline{m}_s,\overline{n}_s}(\lambda_s)$$ 
The condition on $l$ implies that the only monomial in $x$ which can contribute to this coefficient is the monomial corresponding to $\overline{\iota}$. 
Using the formula \eqref{eq:sigma-explicit}, we obtain that this coefficient is $-l\,c_{\overline{\iota}}\,\overline{n}_1\,t$, where $t$ is the maximal number with $\overline{n}_t=\overline{n}_1$ and $\overline{m}_t=\overline{m}_1$. 
This coefficient is non-zero. Hence, we have $\sigma_l(x)\neq 0$. However, we have $\sigma_l(x)\in J$ and the weight of $\sigma_l(x)$ is strictly smaller than the weight of $x$.
This contradicts the minimality of $n$.
\end{proof}

\smallskip

\subsection{Virasoro subalgebra}
Let us introduce another Lie subalgebra of $\DW{\c}(S)$. The results of this section will not be used anywhere, but seem to be of independent interest.
\begin{definition}
    Let $\eta: H^*(S,\Q)^{\otimes 2}\to Z(S)[\c]$ be a bilinear map. 
    The Virasoro algebra $\Vir_S(\eta)$ of $S$ of central charge $\eta$ is the Lie algebra generated by 
    \[
        \{ \mathfrak{L}_n(\lambda), \gamma, \c \,:\, n\in\Z, \lambda\in H^*(S,\Q),\gamma\in Z(S)\},
    \]
    where $\gamma\in Z(S)$ and $\c$ are central, and the Lie bracket is given by
    \begin{equation}\label{eq:Vir-comm}
        [\mathfrak{L}_m(\lambda),\mathfrak{L}_n(\mu)] = (n-m)\mathfrak{L}_{m+n}(\lambda\mu) - \frac{n^3-n}{12}\delta_{-m,n}\eta(\lambda,\mu).
    \end{equation}
\end{definition}

\begin{remark}
    Our definition differs from the standard conventions by a sign; in other words, we are considering the opposite Lie algebra of the usual definition.
\end{remark}

Let us fix the following elements in $\DW{\c}(S)$
\[
    D_{\pm 2,1} = \pm \frac{1}{2} [D_{\pm 1,2},D_{\pm 1,0}], \quad D_{\pm (n+1),1} = \frac{\pm 1}{n-1}[D_{\pm 1,1}, D_{\pm n,1}],\quad n\geq 2.
\]
We also define the following elements
\begin{align*}
    \mathfrak{L}_n(\lambda) &= D_{n,1}(\lambda) + \frac{(n-1)\c}{2}D_{n,0}(c_1\lambda) + \frac{1}{2}\delta_{n,0}\binom{\c}{3}\int_S c_1^2\lambda,\qquad n\geq 0,\\
    \mathfrak{L}_n(\lambda) &= \exp(i\pi\c)\left( D_{n,1}(\lambda) - \frac{(n+1)\c}{2}D_{n,0}(c_1\lambda) \right),\qquad n<0.
\end{align*}
\begin{proposition}\label{prop:Vir-in-double}
    The assignment $\mathfrak{L}_n(\lambda)\mapsto \mathfrak{L}_n(\lambda)$, $\gamma\mapsto\gamma$, $\c\mapsto\c$ defines a morphism of algebras $U(\Vir_S(\eta))\to \DW{\c}(S)$, where the central charge $\eta$ is given by
    \[
        \eta(\lambda,\mu) = \c\left(\int_S c_2\lambda\mu - (1-\c^2)\int_S c_1^2\lambda\mu+2\psi_0(c_1\lambda\mu)\right).
    \]
\end{proposition}\label{prop:Virasoro}
\begin{proof}
    Note that the relation~\eqref{eq:Vir-comm} for $m$, $n$ of the same sign follows from Proposition~\ref{prop:W-PBW}.
    For other commutators, note that $\Vir_S(\eta)$ is generated over $Z(S)[\c]$ by the elements $\mathfrak{L}_n(\lambda)$, $|n|\leq 2$.
    Once we check the commutation relations between these elements, the rest of the relations can be deduced by an inductive argument as in Proposition~\ref{prop:heis-in-double}.
    The computation for the five elements above is straightforward, albeit laborious.
    We leave it to the reader.
    It is performed using the definitions of elements $\mathfrak{L}_n(\lambda)$ and the defining relations of $\DW{\c}(S)$.
    Let us briefly comment on the appearance of $c_2$ in the formula $\eta$.
    While~\eqref{E:doublerelation} does not manifestly depend on $c_2$, after writing out its r.h.s. in terms of $\psi_i$'s for $m+n=4$ we obtain 
    \begin{equation*}
    \begin{split}
        -\exp(i\pi\c)\big( 4\psi_3(\lambda\mu) - 3\c\psi_2&(c_1\lambda\mu) - 2(\psi_0\psi_1)(c_1\lambda\mu)\\ &+ (\c^2-\c+2)\psi_1(c_1^2\lambda\mu)
         - \underline{2\psi_1(c_2\lambda\mu)} + \cdots \big)
    \end{split}
    \end{equation*}
    where the omitted summands belong to the center of $\DW{\c}(S)$.
    The underlined term is precisely the one which gives rise to $\int_S c_2\lambda\mu$.
\end{proof}

\smallskip 

\subsection{Proof of Theorem~\ref{Prop:defHeis}(a)}\label{sec:str-def}
We have obtained an upper bound on the graded dimension of $W^\geq(S)$ in \S\ref{sec:order-filt}.
To get a lower bound, we consider the descending algebra filtration $G^N$, $N\in\N$, of $\DW{\c}(S)$ such that 
$T^\pm_n(\lambda),$  $\psi_l(\lambda)$ are in degree $\deg(\lambda)$. From the defining relations and Proposition~\ref{prop:big-W-triang} it follows that $G^N$ is spanned by the monomials 
\[
T^-_{i_1}(\mu^-_1)\cdots T^-_{i_{l_-}}(\mu^-_{l_-})\psi_{j_1}(\lambda_1)\cdots \psi_{j_{l_0}}(\lambda_{l_0})T^+_{k_1}(\mu^+_1)\cdots T^+_{k_{l_+}}(\mu^+_{l_+})
\]
for which
$\sum \deg(\lambda_i) + \sum\deg(\mu^\pm_j) \geq N$.
The restriction of this filtration to $W^\geq(S)$ is given by the same definition without $T^-$'s.
Each $G^N$ is a two-sided ideal in $\DW{\c}(S)$. We define
\begin{align}\label{descfilt}\Gr^\bullet W^{\c}(S) = \sum_{N\in\N} G^N/G^{N-1}.
\end{align}
We define a double $$\DW{\c}_0(S)=U(\mathfrak{w}^{(\c)}_0(S))$$ of $W^\geq_0(S)$ in an obvious way, 
using relations as in Theorem~\ref{thm:big-W-undef}.
The images of $\psi_n(\lambda)$ and $T^\pm_n(\lambda)$ in $\Gr^\bullet W^\geq(S)$ satisfy the relations \eqref{W:a}-\eqref{W:e}, \eqref{W:g}. Since
the last three summands in \eqref{W:f} have higher $G$-degree than the first four, the relation~\eqref{E:Lie-rels} holds in $\Gr^\bullet W^\geq(S)$ as well.
For the same reason, relation~\eqref{E:doublerelation} with $c_1=s_2=0$ holds. 
Thus there is an obvious algebra homomorphism 
$$\zeta:\DW{\c}_0(S)\to \Gr^\bullet \DW{\c}(S)$$ 

\begin{lemma}\label{lm:W-bound-below}
If $e\neq 0$ then the morphism $\zeta:\DW{e}_0(S)\to \Gr^\bullet \DW{e}(S)$ is an isomorphism.
\end{lemma}

\begin{proof}
The morphism is surjective. Its restriction to $U(\Heis(S))$ is non-zero. 
Its kernel is a two-sided ideal of $U(\Heis(S))$, which is 0 because $e\neq 0$. 
We conclude by Lemma~\ref{lem:Heis-faithful}.
\end{proof}

\begin{proof}[Proof of Theorem~\ref{Prop:defHeis}(a)]
The map $\zeta$ restricts to an isomorphism $W_0^\geq(S) \to \Gr^\bullet W^\geq(S)$ of graded algebras.
It suffices to observe that the Hilbert series of $W_0^\geq(S)= U(\mathfrak{w}^\geq_{0}(S))$ is given by~\eqref{eq:Wplus-char}.
\end{proof}

\subsection{Structure of $W^\geq(S)$ in the semi-deformed case}\label{sec:semidef}
Assume that $c_1 = 0$ and $s_2=q^2$.
Recall the Lie algebra $\mathfrak{w}^\geq(S)$ from Definition~\ref{def:W-semidef}.
The following proposition is a refinement of Theorem~\ref{T:W(S)}.

\begin{proposition}
    If $c_1=0$ and $s_2=q^2$, then
    there is an algebra isomorphism 
    $$\Phi: W^\geq(S)\to U(\mathfrak{w}^\geq(S))$$ such that
    \[
        T^+_n(\lambda)\mapsto zD^n\lambda, \quad \psi_n(\lambda)\mapsto q^nB_n(Dq^{-1})\lambda,
    \]
    where $B_n(x) = \sum_{i=0}^n \binom{n}{i}B_{n-i}x^i $ is the Bernoulli polynomial.
\end{proposition}

\begin{proof}
    Relations \eqref{W:b}, \eqref{W:e}-\eqref{W:g} hold between $\Phi(T^+_n(\lambda))$.
    Relations \eqref{W:a} and \eqref{W:c} between $\Phi(\psi_n(\lambda))$ also hold for obvious reasons. Moreover,
    \[
        [B_m(Dq^{-1})\lambda,zD^n\mu] = zq^{-1}\left( B_m(Dq^{-1}+1) - B_m(Dq^{-1}) \right)D^n\lambda\mu = mzq^{-m}D^{m+n-1}\lambda\mu,
    \]
    where we used the ``umbral'' relation $B_n(x+1)-B_n(x) = nx^{n-1}$.
    Therefore the relation \eqref{W:d} holds as well.
    We have thus obtained a well-defined homomorphism 
    $$\Phi: W^\geq(S)\to U(\mathfrak{w}^\geq(S)).$$
   Observe that the Lie algebra $\mathfrak{w}^\geq(S)$ is generated by the elements $D^n\lambda$ and $zD^n\lambda$.
    In particular, $\Phi$ is surjective.
    Finally, the graded dimension of $W^\geq(S)$ is equal to the graded dimension of $U(\mathfrak{w}^\geq(S))$ by Theorem~\ref{T:W(S)}(a), so we may conclude.
\end{proof}

\begin{remark}
In~\cite{LQW}, an action of a Lie algebra $\mathcal{W}_S$ on the cohomology of the Hilbert schemes of points on $S$ was constructed via vertex algebra methods.
A basis is given by elements 
$$\mathfrak{J}_m^p(\lambda)
,\quad
m\in\Z
,\quad
p\in \N
,\quad
\lambda\in H^*(S,\Q)$$ and the Lie bracket is, up to central charge,
\[
    [\mathfrak{J}_m^p(\lambda),\mathfrak{J}_n^q(\mu)] = (qm-pn)\mathfrak{J}_{m+n}^{p+q-1}(\lambda\mu) - \frac{\Omega_{m,n}^{p,q}}{12}\mathfrak{J}_{m+n}^{p+q-3}(c_2\lambda\mu),
\]  
where $\Omega_{m,n}^{p,q}$ is given by formula (5.2) in~\emph{loc. cit.}
This Lie algebra lives in a certain completion of $\DW{1}(S)$.
If $c_1=0$, it is a subalgebra of $\DW{1}(S)$ by Lemma~5.2 in \emph{loc. cit.}
Let us set 
\[
    T^\pm_n(\lambda) = \mathfrak{J}_{\pm 1}^n(\lambda), \qquad \psi_n(\lambda) = \mathfrak{J}_0^n(\lambda).
\]
A direct computation implies that the relations of $\DW{1}(S)$ with $c_1=0$ hold.
Therefore we obtain a homomorphism of algebras $\DW{1}(S)\to U(\mathcal{W}_S)$, which is an isomorphism by a dimension check.
The existence of this homomorphism relies on the fact that $c_2^2=0$, which is true in $H^*(S,\Q)$ for degree reasons. In the presence of a torus action this vanishing fails, hence the results of~\cite{LQW} do not apply, and one obtains instead the semi-deformed algebra $U(\mathfrak{w}^\geq(S))$. Our presentation does not involve the factor $\Omega_{m,n}^{p,q}$.
\end{remark}

\medskip

\section{Fock space representations of $\DW{r}(S)$}\label{sec:Fock4}

In this section, we construct a Fock space representation of $\DW{r}(S)$ for any $r \geq 0$ by considering the action of Hecke correspondences on tautological cohomology rings. We still assume that $S$ is projective.

\medskip

\subsection{The algebra of universal Hecke operators}\label{sec:fockdef} 
Recall that $R^\pm$ is the algebra homomorphism and $Q^\pm$ the $\Lambda(S)$-linear map in 
\eqref{E:finalformT}, and that $u=c_1(\rho)$ is as in \eqref{u}.
We define the elements $\phi_n(\lambda),c(\lambda)\in\Lambda(S)$ by the following generating series
in $\Lambda(S)((z))$
\begin{align*}
\phi_\lambda(z)&=z^{-1}c(\lambda) + \sum_{n\geq 0} \frac{z^n}{n!}\phi_n(\lambda) \\
&= \sum_{n\geq 0} \frac{z^{n-1}}{n!} p_n\big(\Td_S(z)\cup\lambda\big) .
\end{align*}
Note that $c(\lambda)=0$ if $\lambda \not\in \C\pt$ while $c(\pt)=p_0(\pt)=\mathbf{r}$.
Let $\phi_n(\lambda),$ $c(\lambda)$ denote also the operators of left multiplication in $\Lambda(S)$. 
Recall the operators $L^\pm_n(\lambda)$
introduced in \eqref{eq:T-fla+} and \eqref{eq:T-fla-}.
We set
\begin{align}\label{fsL}L^{\pm}_\lambda(z)=\sum_{n \in \Z} L^{\pm}_n(\lambda)z^n.\end{align}
We recall $\theta(z)=\sum_{m \geq 0} h_mz^m$ from \S\ref{sec:defoWalg}.

\smallskip

\begin{proposition}\label{prop:psi-T-rels}
The assignment 
$$\psi_n(\lambda) \mapsto \phi_n(\lambda)
,\quad
\c \mapsto \mathbf{r}
,\quad
T^\pm_n(\lambda) \mapsto L^\pm_n(\lambda)
,\quad
n \geq 0
,\quad
\lambda \in H^*(S,\Q)$$ extends to actions of $W^{\geq}(S)$ and $W^{\leq}(S)$ on $\Lambda(S)$. 
\end{proposition}

\begin{proof}
We will deal with the positive operators only, the second case being identical. To unburden the notation, we suppress $+$ from the notation.
We have to show that the operators $\phi_n(\lambda), L_n(\lambda)$ satisfy the defining relations \eqref{W:a}-\eqref{W:g} of $W^\geq(S)$. 
The relations \eqref{W:a}, \eqref{W:b}, \eqref{W:c} being immediate, we concentrate on the remaining ones.
We deduce from~\eqref{E:evpn} that
$$\sum_{m\geq 0}\frac{z^m}{m!}\ev_\delta(p_m(\lambda)) = z^2e^{uz}\Td^{-1}_S(z)\lambda.$$ 
In particular, from \eqref{RQ}, \eqref{eq:T-fla+}, we deduce that
\begin{align*}
[\phi_\lambda(z),L_n(\mu)] 
& = \sum_{m\geq 0} \frac{z^{m-1}}{m!}[p_m(\Td_S(z)\lambda),L_n(\mu)]\\
&= \sum_{m\geq 0} \frac{z^{m-1}}{m!}
Q^+\Big(\mu u^n(p_m-R^+(p_m))(\Td_S(z)\lambda)R^+(-)\Big)\\
& = z Q^+\Big( \mu u^n \sum_{m\geq 0}\frac{z^{m-2}}{m!}\ev_\delta(p_m)(\Td_S(z)\lambda)R^+(-) \Big)\\
&= zQ^+\Big(u^ne^{uz}\lambda\mu R^+(-)\Big),\\
&= \sum_{m\geq 0} \frac{z^{m+1}}{m!}Q^+\Big(\lambda\mu u^{m+n}R^+(-)\Big)\\
&= \sum_{m\geq 1} \frac{z^m}{(m-1)!}L_{m+n-1}(\lambda\mu)
\end{align*}
which proves the relation \eqref{W:d}. 
Thanks to this, it suffices to check relations \eqref{W:e}, \eqref{W:f} and \eqref{W:g} when evaluated at $1$.
Let us write
\[
        \Omega(x,y) = \frac{y^2}{(1-y(x^{-1}-t_1))(1-y(x^{-1}-t_2))} \in H^*(S,\Q)[x,x^{-1}][[y]].
\]

\begin{lemma}\label{lem:T-product}
We have the following identity
\[L_\lambda(x)L_\mu(y)(1) = m(\theta(x)\otimes\theta(y))(1-\Delta\Omega(x,y))(\lambda\otimes \mu)\]
where $m : \Lambda(S) \otimes \Lambda(S) \to \Lambda(S)$ is the multiplication.
\end{lemma}

\begin{proof}
Let us first note that
\[L_\mu(y)(1) = \sum_{n\geq 0} y^nh_n(\mu) = 
\exp\Big( \sum_{k\geq 1} \frac{y^k}{k}p_k \Big)(\mu).\]
Using formula~\eqref{E:evpn}, we get after applying $R$
\begin{align*}
R^+(L_\mu(y)(1)) &= R^+ \left( \int_S \exp\Big( \sum_k \frac{y^k}{k} \uch_k\Big) \mu_2\right),\\
&=\int_S \exp\Big( \sum_{k\geq 1} \frac{y^k}{k}((\uch_k)_{12} - (\Delta f_k(u))_{23}) \Big)\mu_2\\
&=\int_S \exp\Big( \sum_{k\geq 1} \frac{y^k}{k}(\uch_k)\Big)_{12} \exp\Big(-\sum_k \frac{y^k}{k} \Delta f_k(u)\Big)_{23}\; \mu_2\\
&=\int_S p(\theta(y))_{12} \;\exp\Big( -\Delta\sum_{k\geq 1} \frac{y^k}{k}f_k(u) \Big)_{23}\;\mu_2.
\end{align*}
Here, we used indices to specify the position in the tensor product, i.e., $\mu_2=1 \otimes \mu \otimes 1$, etc. Since $\Delta^2 = t_1t_2\Delta$, we can compute the exponential term
\begin{align*}
    \exp & \Big( -\Delta\sum_{k\geq 1} y^k\frac{u^k - (u-t_1)^k - (u-t_2)^k + (u-t_1-t_2)^k}{kt_1t_2} \Big)\\
    & = 1+\frac{\Delta}{t_1t_2}\bigg(\exp\Big( -\sum_{k\geq 1} y^k\frac{u^k - (u-t_1)^k - (u-t_2)^k + (u-t_1-t_2)^k}{k} \Big)-1\bigg)\\
    & = 1 - \frac{\Delta}{t_1t_2}\left( \frac{(1-yu)(1-yu+yt_1+yt_2)}{(1-yu+yt_1)(1-yu+yt_2)}-1 \right)\\
    & = 1 - \frac{y^2\Delta}{(1-yu+yt_1)(1-yu+yt_2)}.
\end{align*}
Recall the operator Q
$Q^+: \Lambda(S) \otimes H^*(\Coh_\delta,\Q) \to \Lambda(S)$ from \eqref{RQ}. We have 
\[\sum_{m \in \Z} Q^+(z^mu^{m+i}\lambda) = z^{-i}\sum_{m\geq 0} z^mh_m(\lambda).\]
Putting everything together, and extending $Q^+$ linearly over $\Q[[x^{\pm 1}, y^{\pm 1}]]$ and using the relation 
\begin{align}\label{FS}
\sum_{n\in \Z} (ux)^n F(u)=\sum_{n\in \Z} (ux)^n F(x^{-1})
,\quad
F(u)\in A[[u^{-1},u]],
\end{align}
where $A$ is any ring,
we conclude that
\begin{align*}
    L_\lambda(x)L_\mu(y)(1) 
    &= Q^+\Big( \lambda\sum_{n\in \Z} (xu)^nR^+L_\mu(y)(1) \Big)\\
    &=\int_S Q^+\Big(\lambda_2 \sum_{n \in \Z} (xu)^n p(\theta(y))_{13} \big(1-\frac{y^2 \Delta_{23}}{(1-yu+yt_1)(1-yu+yt_2)}\big)\mu_3\Big)\\
    &=\int_S Q^+\Big(\lambda_2 \sum_{n \in \Z} (xu)^n p(\theta(y))_{13} \big(1-\frac{y^2 \Delta_{23}}{(1-yx^{-1}+yt_1)(1-yx^{-1}+yt_2)}\big)\mu_3\Big)\\
    & = \int_{S \times S}p(\theta(x))_{12}\;p(\theta(y))_{13}\;\Big( 1 - \frac{y^2\Delta}{(1-yx^{-1}+yt_1)(1-yx^{-1}+yt_2)} \Big)\;\lambda_2\;\mu_3\\
    & = m(\theta(x)\otimes\theta(y))(1-\Delta\Omega(x,y))(\lambda\otimes \mu)).
\end{align*}
\end{proof}

Now, by Lemma~\ref{lem:T-product}, we have
\begin{align*}
    [L_\lambda(x),L_\mu(y)](1) 
    &= -\int_{S\times S} p(\theta(x))_{12}\; p(\theta(y)))_{13}\;(\Delta \Omega(x,y)-\Delta \Omega(y,x))_{23}\; \lambda_2\;\mu_3\\
    &= -\theta(x)\theta(y)(\Omega(x,y)-\Omega(y,x))(\lambda\mu)
\end{align*}
which implies relation \eqref{W:e}. 
    
    Next, let us put $z=y^{-1}-x^{-1}$. A direct computation yields
    \begin{align*}
        (z^3-z(s_2\otimes 1 + 1\otimes s_2)/2-s_1\Delta)&(1-\Delta \Omega(x,y)) \\
        =&(z^3-z(s_2\otimes 1 + 1\otimes s_2)/2+s_1\Delta)(1-\Delta \Omega(y,x)).
    \end{align*}
    Unpacking the generating series and using Lemma~\ref{lem:T-product}, we see that the relation \eqref{W:f} holds when evaluated at $1$. Hence it holds in general.

    Let's finally turn to \eqref{W:g}. Similarly to Lemma~\ref{lem:T-product}, we have
    \begin{align*}
        L_{\lambda_1}(x_1)L_{\lambda_2}(x_2)L_{\lambda_3}(x_3)(1) =
        m(\theta(x_1)\otimes\theta(x_2)\otimes \theta(x_3))\prod_{i<j}(1-\Delta \Omega(x_i,x_j))_{ij}(\lambda_1\otimes\lambda_2\otimes\lambda_3).
    \end{align*}
    The product of $\Omega$ functions is well-defined as a Laurent series.
    Using this formula and expanding the products, one can show that
    \[
        [L_{\lambda_1}(x_1),[L_{\lambda_2}(x_2),L_{\lambda_3}(x_3)]](1) = m(\theta(x_1)\otimes\theta(x_2)\otimes \theta(x_3)) \Delta_{123}K(x_1,x_2,x_3)(\lambda_1\otimes\lambda_2\otimes\lambda_3),
    \]
    where 
    \begin{align*}
\Delta_{123} &= \Delta_{12}\Delta_{23}= \Delta_{12}\Delta_{13}= \Delta_{13}\Delta_{23},\\
    K(x_1,x_2,x_3) &= (1-\sigma_{23})(1+\sigma_{13})\Gamma(x_1,x_2,x_3),\\   
\Gamma(x_1,x_2,x_3)&=\Omega(x_1,x_2)\Omega(x_2,x_3)+\Omega(x_1,x_2)\Omega(x_1,x_3)+\Omega(x_1,x_3)\Omega(x_2,x_3)\\
&\quad-t_1t_2\Omega(x_1,x_2)\Omega(x_2,x_3)\Omega(x_1,x_3)
    \end{align*}
    and $\sigma_{ij}$ stands for the transposition of the indices $(i,j)$.
    In order to prove the relation \eqref{W:g}, we must show that the following Laurent series vanishes
    \begin{equation*}
    \begin{split}
        \sum_{w\in S_3}& w[L_{\lambda_1}(x_1),[L_{\lambda_2}(x_2),x_3^{-1}L_{\lambda_3}(x_3)]](1)\\ 
        &=(\theta(x_1)\otimes\theta(x_2)\otimes \theta(x_3))\Delta_{123}K'(x_1,x_2,x_3)(\lambda_1\otimes\lambda_2\otimes\lambda_3),
    \end{split}
    \end{equation*}
    where $K'(x_1,x_2,x_3) = \sum_{w\in S_3} w(x_3^{-1}K(x_1,x_2,x_3))$. We will show that $K'$ vanishes.
    Since
    \begin{align*}
    \sum_{w \in S_3} w x_3^{-1}(1-\sigma_{23})(1+\sigma_{13})&=\sum_w w (x_3^{-1}-\sigma_{23}x_2^{-1} + \sigma_{13}x_1^{-1}-\sigma_{23}\sigma_{13}x_2^{-1})\\ 
    &=\sum_w w (x_1^{-1} -2 x_2^{-1} + x_3^{-1})
    \end{align*}
    we obtain
    \begin{align*}
        K'&(x_1,x_2,x_3) = \sum_{w\in S_3} w\Big( (x_1^{-1}-2x_2^{-1}+x_3^{-1}) \Gamma(x_1,x_2,x_3)\Big).
    \end{align*}
    A direct computation yields
    \begin{align*}
        (x_1^{-1}-2x_2^{-1}+x_3^{-1})&(\Omega(x_1,x_2)\Omega(x_2,x_3)-t_1t_2\Omega(x_1,x_2)\Omega(x_2,x_3)\Omega(x_1,x_3))\\
        & = (x_3^{-1}-x_1^{-1})(\Omega(x_1,x_2)\Omega(x_1,x_3)-\Omega(x_1,x_3)\Omega(x_2,x_3)).
    \end{align*}
    Therefore, we can express $K'$ as follows
    \begin{equation*}
    \begin{split}
        K'(&x_1,x_2,x_3)\\
        &= 2\sum_{w\in S_3} w\left( (x_3^{-1}-x_2^{-1})\Omega(x_1,x_2)\Omega(x_1,x_3)+(x_1^{-1}-x_2^{-1})\Omega(x_1,x_3)\Omega(x_2,x_3) \right).
    \end{split}
    \end{equation*}
    The first term in parentheses is antisymmetric is $x_2$ and $x_3$, and the second one is antisymmetric in $x_1$ and $x_2$.
    Therefore their symmetrizations vanish. Hence $K'=0$.
    Thus the relation \eqref{W:g} holds, and the proof is complete. Above, we considered elements $L_n(\lambda)$ with $n<0$, but the only relations which we are interested in are those involving only $n\geq 0$. Proposition~\ref{prop:psi-T-rels} is proved.
\end{proof}

\medskip

\subsection{Level $r$ Fock space representation of $\DW{\c}$} Fix an integer $r \geq 0$. We set
$$\Lambda(S)_r=\Lambda(S)|_{\mathbf{r}=r}
, \quad 
{}_rL^\pm_n(\lambda)=L_{n+1 \mp r}^\pm(\lambda)
, \quad n \in \Z
, \quad
\lambda \in H^*(S,\Q)$$
Consider the formal series
$${}_rL_\lambda^\pm(z)=z^{-1\pm r}L_\lambda^\pm(z) \in \text{End}(\Lambda(S)|_{r})[z,z^{-1}].$$
The formal series ${}_rL_\lambda^+(z)$ specializes to the formal series $L^+_\lambda(z)$ 
in \eqref{fsL} for $r=1$.

\smallskip

\begin{proposition}\label{Prop:+-relationsfockspace} The following relation holds in 
    \begin{equation}\label{E:+-relationsFock}
        \big[{}_rL^+_{\lambda}(x), {}_rL^-_{\mu}(y)\big]_{++}=\left[\frac{1}{c_1x} \left( 1-\frac{\theta(x)}{\widetilde{\theta}(x)}\right) \delta\Big(\frac{y}{x}\Big)(\lambda\mu)\right]_{++} \in \End(\Lambda(S)|_{r})[[x,y]].
    \end{equation}
\end{proposition}
 \begin{proof} We begin by evaluating the left hand side of \eqref{E:+-relationsFock} at $1$.

 \smallskip

 \begin{lemma}\label{L:evaluationon1} We have
$$ \big[{}_rL^+_{\lambda}(x), {}_rL^-_{\mu}(y)\big]_{++} \cdot 1 =\left[\frac{1}{c_1x} \left( 1-\frac{\theta(x)}{\widetilde{\theta}(x)}\right) \delta\Big(\frac{y}{x}\Big)(\lambda\mu)\right]_{++}.$$
 \end{lemma}
 
 \begin{proof} The proof bears some resemblance to that of Lemma~\ref{lem:T-product}. We have
 $$L^-_\mu(y)\cdot 1=\sum_l y^lQ^-(\mu u^l)=\sum_l (-y)^l \tilde{e}_l(\mu)=\widetilde{\theta}(y)^{-1}(\mu)=\exp\Big( -\sum_{k \geq 1} \frac{y^k}{k} \tilde{p}_k\Big)(\mu).$$
A computation using \eqref{E:R+R-onpn} and \eqref{FS} yields
\begin{equation*}
    \begin{split}
        L^+_\lambda(x)L^-_\mu(y)\cdot 1&=\int_{S \times S} Q^+\left( \sum_{n \in \Z} x^nu^n \widetilde{\theta}(y)^{-1}_{13} \exp\Big(\sum_k \frac{(-y)^k}{k} f_k(-u)\Delta\Big)_{23}\lambda_2\mu_3\right)\\
        &=\int_{S \times S} Q^+\left( \sum_{n \in \Z} x^nu^n \widetilde{\theta}(y)^{-1}_{13} \exp\Big(\sum_k \frac{(-y)^k}{k} f_k(-x^{-1})\Delta\Big)_{23}\lambda_2\mu_3\right).
    \end{split}
\end{equation*}
Set
$$\overline{\Omega}(a,b)=\exp\Big( \sum_{k\geq 1} \frac{a^k}{k} \Delta f_k(b)\Big)=1 + \frac{\Delta a^2}{(1-ab)(1-a(b-c_1))}.$$
We get
\begin{equation}\label{E:proofdoublerel1}
    L^+_\lambda(x)L^-_\mu(y)\cdot 1=m(\theta(x) \otimes \widetilde{\theta}(y)^{-1})\overline{\Omega}(-y,-x^{-1})(\lambda \otimes \mu)
\end{equation}
and likewise
\begin{equation}\label{E:proofdoublerel2}
    L^-_\mu(y)L^+_\lambda(x)\cdot 1=m(\theta(x) \otimes \widetilde{\theta}(y)^{-1})\overline{\Omega}(x,y^{-1})(\lambda \otimes \mu).
\end{equation}
Observe that $\overline{\Omega}(-y,-x^{-1})=\overline{\Omega}(x,y^{-1})$ as rational functions.
However \eqref{E:proofdoublerel1} and \eqref{E:proofdoublerel2} should be expanded out in 
$\Lambda(S)((x))[[y]]$ and $\Lambda(S)((y))[[x]]$ respectively,
since by construction $L^\pm_n(\sigma)\cdot 1=0$ for $n <0$.
In other words, we have
$$\big[{}_rL^+_\lambda(x), {}_rL^-_\mu(y)\big]\cdot 1=x^{r-1}y^{-1-r}m(\theta(x) \otimes \widetilde{\theta}(y)^{-1}) \left(\overline{\Omega}(x,y^{-1})_+-\overline{\Omega}(x,y^{-1})_-\right)(\lambda \otimes \mu)$$
where $+$ and $-$ subscript indicate expansion in $\Lambda(S)((x))[[y]]$ and $\Lambda(S)((y))[[x]]$ respectively. 
Set $\delta(z)=\sum_{n\in\Z}z^n$. From the equality
$$\overline{\Omega}(x,y^{-1})=1 + \frac{\Delta}{c_1}\left( \frac{1}{x^{-1}-y^{-1}} - \frac{1}{x^{-1}-(y^{-1}-c_1)}\right)$$
we deduce that
\begin{equation*}
    \begin{split}
    (xy)^{-1}\left(\overline{\Omega}(x,y^{-1})_+-\overline{\Omega}(x,y^{-1})_-\right)=
\frac{\Delta}{c_1x}\Big( \delta\left( \frac{x(1-c_1y)}{y}\right) - \delta\left(\frac{y}{x}\right)\Big).
    \end{split}
\end{equation*}
 Thus we get
\begin{equation}\label{E:proofdoublerel3}
\big[{}_rL^+_\lambda(x), {}_rL^-_\mu(y)\big]\cdot 1=\int_{S\times S} \frac{\Delta_{23}}{c_1x} \big(A (x,y) + B(x,y)\big) \lambda_2\mu_3
\end{equation}
where
\begin{equation*}
    \begin{split}
        A(x,y)&=\Big(\frac{x}{y}\Big)^rm(\theta(x) \otimes \widetilde{\theta}(y)^{-1}) \;\delta\Big( \frac{x(1-c_1y)}{y}\Big),\\
        &=\Big(\frac{x(1-c_1y)}{y}\Big)^r m\Big(\theta(x) \otimes \theta\Big( \frac{y}{1-c_1y}\Big)^{-1}\Big)\;\delta\Big( \frac{x(1-c_1y)}{y}\Big),\\
        &=m(\theta(x) \otimes \theta(x)^{-1})\;\delta\Big( \frac{x(1-c_1y)}{y}\Big),\\
        B(x,y)&=-\Big(\frac{x}{y}\Big)^rm(\theta(x) \otimes \widetilde{\theta}(y)^{-1}) \;\delta\Big( \frac{y}{x}\Big),\\
        &=- m\Big(\theta(x) \otimes \widetilde{\theta}(x)^{-1}\Big)\;\delta\Big( \frac{y}{x}\Big).
    \end{split}
\end{equation*}
In simplifying $A(x,y)$ we used the following calculation
\begin{equation*}
  \begin{split}
      \widetilde{\theta}(y)(\lambda)&=\int_S \tau_{c_1} \exp\left( \sum_{k \geq 1} \frac{y^k}{k}p_k\right)\lambda\\
      &=\int_S \exp\left( \sum_{k \geq 1} \frac{y^k}{k} \sum_{i=0}^k \binom{k}{i}p_i c_1^{k-i}\right)\lambda\\
      &=\int_S \exp\left(r\sum_{k \geq 1}\frac{(c_1y)^k}{k}\right)\exp\left( \sum_{k \geq i \geq 1}  \binom{k-1}{i-1}\frac{y^k}{i}p_i c_1^{k-i}\right)\lambda\\
      &=(1-c_1y)^{-r}\int_S \exp\left( \sum_{i \geq 1} \frac{y^i}{i(1-c_1y)^i}p_i\right)\lambda\\
      &=(1-c_1y)^{-r} \theta\left( \frac{y}{1-c_1y}\right)(\lambda).
  \end{split}  
\end{equation*}
Substituting in \eqref{E:proofdoublerel3} and observing that
\begin{align*}
\left[\int_{S \times S} \frac{\Delta}{c_1x} A(x,y) \lambda_2\mu_3\right]_{++}
&=\left[\int_{S \times S} \frac{\Delta}{c_1x} \delta\Big( \frac{x(1-c_1y)}{y}\Big) \lambda_2\mu_3\right]_{++}=0,\\
\left[\int_{S \times S} \frac{\Delta}{c_1x} \delta\Big( \frac{y}{x}\Big)  \lambda_2\mu_3\right]_{++}&=0
\end{align*}
we easily deduce Lemma~\ref{L:evaluationon1}.
 \end{proof}

 \medskip

 In order to extend Lemma~\ref{L:evaluationon1} to $\Lambda(S)_r$, 
 we consider the commutation relation of $\psi_k(\nu)$ with
 $$C_{ij}(\lambda,\mu)=[{}_rL^+_i(\lambda), {}_rL^-_j(\mu)].$$
 Below, we drop the index $r$ for simplicity.
\begin{align*}
    \psi_k(\nu)C_{ij}(\lambda,\mu) & 
      = \psi_k(\nu) L_i^+(\lambda) L_j^-(\mu) - \psi_k(\nu) L_j^-(\mu) L_i^+(\lambda) \\
    & =  L_i^+(\lambda) L_j^-(\mu)\psi_k(\nu) - kL_i^+(\lambda) L_{j+k-1}^-(\nu\mu) + kL_{i+k-1}^+(\lambda\nu) L_j^-(\mu)  \\
    & \quad -  L_j^-(\mu)L_i^+(\lambda)\psi_k(\nu) - kL_j^-(\mu)L_{i+k-1}^+(\lambda\nu) + k L_{j+k-1}^-(\nu\mu) L_i^+(\lambda) \\
    & = C_{ij}(\lambda,\mu)\psi_k(\nu) + k C_{i+k-1,j}(\lambda\nu,\mu) - k C_{i,j+k-1}(\lambda,\nu\mu) 
\end{align*}
Thus
\begin{equation}\label{eq:psi-TT-comm}
    [\psi_k(\nu)/k,C_{ij}(\lambda,\mu)] = C_{i+k-1,j}(\lambda\nu,\mu) - C_{i,j+k-1}(\lambda,\nu\mu).
\end{equation}
Applying \eqref{eq:psi-TT-comm} to $1$ and using the fact that, thanks to Lemma~\ref{L:evaluationon1}, $C_{ij}(\lambda,\mu)\cdot 1$ only depends on $i+j$ and $\lambda\mu$  we see that
\begin{equation}\label{E:common1}
[\psi_k(\nu)/k,C_{ij}(\lambda,\mu)]\cdot 1=0.
\end{equation}
Using \eqref{eq:psi-TT-comm} and \eqref{E:common1} recursively, one gets that
$$[C_{ij}(\lambda,\mu),\psi_{k_1}(\nu_1)\cdots \psi_{k_l}(\nu_l)]\cdot 1=0$$
for any $(k_1,\nu_1), \ldots, (k_l,\nu_l)$. This proves Proposition~\ref{Prop:+-relationsfockspace}.
\end{proof}

\smallskip

Unraveling formula~\eqref{E:+-relationsFock}, we obtain, for $i\geq 0$ and $n-i\geq 0$,

    \begin{align}\label{eq:comm-pm}
        \Big[{}_r{L}_i^+(\lambda),{}_r{L}_{n-i}^-(\mu)\Big] = \sum_{0\leq j\leq k\leq n} (-1)^k\binom{r-k+j}{j+1}h_{n-k}e_{k-j}(c_1^{j}\lambda\mu),
    \end{align}
which highlights the dependence on $r$. 

\begin{remark}
    Note that the commutators $[L^+_i(\lambda),L^+_j(\mu)]$ are independent of $i-j$ only in the non-deformed case, see Remark~\ref{rmk:semidef-lin-param}.
\end{remark}

\medskip

\subsection{Fock space representations of $\DW{r}(S)$}\label{sec:Fockspacedefinition}

\begin{definition}\label{def:Fock}
The level $r$ Fock space of $S$ is
the graded vector space
\[\mathbf{F}^{(r)}(S)=\Lambda(S)_r\otimes \mathbb{C}[s,s^{-1}].
\]
\end{definition}

\smallskip

We may restate Propositions~\ref{prop:psi-T-rels} and \ref{Prop:+-relationsfockspace} as follows

\begin{corollary}\label{cor:Fock-is-module}
The assignment 
$$\psi_n(\lambda) \mapsto \text{mult.\;by\;}\psi_n(\lambda), \qquad 
T^+_n(\lambda)\mapsto{}_rL^+_n(\lambda) s, \qquad 
T^-_n(\lambda)\mapsto(-1)^{r+1}{}_rL_n(\lambda)s^{-1}$$
defines a graded $\DW{r}(S)$-module structure on $\mathbf{F}^{(r)}(S)$.\qed
\end{corollary}

\begin{proposition}\label{cor:big-W-faithful}
If $r>0$, then the representation of $\DW{r}(S)$ on $\mathbf{F}^{(r)}(S)$ is faithful.
\end{proposition}
\begin{proof} Recall that $Z(S)$ is the subalgebra of $\DW{r}(S)$ generated by $\psi_0(\lambda)$ for $\lambda \in H^*(S,\Q)$. Set 
$$Z(S)\cdot\mathbf{F}^{(r)}(S)\subset \mathbf{F}^{(r)}(S).$$
Since $Z(S)$ lies in the center of $\DW{r}(S)$, this is a $\DW{r}(S)$-submodule.
The action of $\DW{r}(S)$ on 
$$\mathbf{F}^{(r)}_\red(S)=\mathbf{F}^{(r)}(S)/Z(S)\mathbf{F}^{(r)}(S)$$ 
factors through $\DW{r}_\red(S)$. 
We have vector space isomorphisms
    \begin{align*}
        \mathbf{F}^{(r)}_\red(S)&= \Sym(H^*(S,\Q)\otimes t^2\Q[t]),\\
        \mathbf{F}^{(r)}(S) &= \Lambda^1(S)\otimes \mathbf{F}^{(r)}_\red(S),\\
    \Lambda^1(S) &= \Sym(H^{\geq 2}(S,\Q)\otimes t).
    \end{align*}
    Since $\Lambda^1(S)$ is the regular $Z(S)$-module, it suffices to prove that the action of $\DW{r}_\red(S)$ on $\mathbf{F}^{(r)}_\red(S)$ is faithful.
    Using Lemma~\ref{lem:Heis-faithful}, we only need to prove the faithfulness of its restriction to the Heisenberg subalgebra
    $U(\mathfrak{h}_S)$. The central charge being non-zero, this algebra has no non-trivial two-sided ideals by a standard argument, see e.g.~\cite[Thm.~2.1]{Coutinho}.
    Hence we may conclude.
\end{proof}

\medskip

\section{Deformed $W$-algebras of open surfaces}\label{sec:W-open-surf}

In this section we do not assume that $S$ is proper anymore, and define several versions of $W$-algebras, modeled on the cohomology and the cohomology with compact support of $S$. Throughout, we fix a smooth compactification $\iota: S \to \oS$. The $W$-algebras which we consider will end up being independent of this choice of compactification.

\medskip

\subsection{Positive halves} We begin with a general construction. Consider a graded ideal $I\subset H^*(\overline{S},\Q)$.
We have $\Delta(I)\subset I\otimes I$, because
\begin{align*}
(x\otimes y,\Delta(z))=(xy,z)=(x,yz)=0=(y\otimes x,\Delta(z)),
\end{align*}
for
$x\in I^\perp, y\in H^*(\oS,\Q)$ and $z\in I$.
Hence $I^\perp$ is also an ideal. Let $J$ be the quotient of $H^*(\oS,\Q)$ by $I$.

\smallskip

\begin{definition}\hfill
\begin{enumerate}[label=$\mathrm{(\alph*)}$,leftmargin=8mm,itemsep=1.2mm]
\item    Let $\Wc^+(I)$ be the smallest graded subalgebra of $W^+(\overline{S})$ containing $D_{m,0}(\lambda)$ for all $m\geq 0$, $\lambda\in I$, and stable under operators $\Ad(\psi_{l}(\mu))$ for all $l>0$, $\mu \in H^*(\oS,\Q)$.  We define $\Wc^-(I)\subset W^-(\overline{S})$ in the same way. 
\item     Let $\Wn^+(J)$ be the quotient of $W^+(\oS)$ by the two-sided ideal $\mathcal{I}^+$ generated by $\Wc^+(I)$. We define $\Wn^-(J)$ in the same way. 
\item Let $W^0(J)$ be the quotient of $W^0(\oS)$ by the ideal generated by elements $\psi_l(\lambda)$ 
for all $l > 0$, $\lambda \in I$. Thus $\c$ descends to a non zero element of $W^0(J)$.  
\end{enumerate}
\end{definition}

\begin{remark}\leavevmode\nolisttopbreak
\begin{enumerate}[label=$\mathrm{(\alph*)}$,leftmargin=8mm,itemsep=1.2mm]
    \item $\Wc^+(I)$ may differ from the subalgebra of $W^+(\oS)$ generated by $T^+_n(\lambda)$ with $n \geq 0$ and $\lambda\in I$. For instance, if $S=\mathbb{P}^2$ and $I=\Q [pt]$, we will show in Corollary~\ref{cor:nondef-bbW} that $\Wc^+(I)$ is a commutative algebra with basis given by monomials in $D_{m,n}(\pt)$, which is not generated by $\{D_{1,n}(\pt)\}_n$.
    \item The ideal $\mathcal{I}^+$ is generated by $T^+_n(\lambda)$ with $n \geq 0$ and $\lambda\in I$.
   \end{enumerate}
\end{remark}

\smallskip

Recall the elements $D_{m,n}(\lambda)\in W^+(\oS)$ in Proposition~\ref{prop:W-PBW}. They are not canonically defined unless $m\leq 1$ or $n\leq 1$.
In this subsection, we fix them to be 
\[
D_{m,n}(\lambda)=\frac{1}{m(n+1)}[\psi_{n+1}(1),D_{m,0}(\lambda)], \quad \lambda \in H^*(\oS,\Q),\quad m \geq 1
,\quad n \geq 1.
\]
The element $D_{m,n}(\lambda)$ belong to $\Wc^+(I)$ if $\lambda \in I$.

\begin{proposition}\label{prop:W-cpt-open} 
The Hilbert series of $\Wc^+(I)$ and $\Wn^+(J)$ are respectively equal to
\[
    P_{\Wc^+(I)}(z,w) = \Exp\left(\frac{P_I(z)z^{-2}w}{(1-z^2)(1-w)} \right),\qquad P_{\Wn^+(J)}(z,w) = \Exp\left(\frac{P_J(z)z^{-2}w}{(1-z^2)(1-w)} \right),
\]
where $P_I(z)=\sum_d \dim(I \cap H^d(\oS,\Q)) z^d$ and $P_J(z)=P_\oS(z)-P_I(z)$. 
\end{proposition}

\begin{proof} 
Let $A \subset W^+(\oS)$ be the subalgebra generated by all Lie words in 
$T_{n_1}(\lambda_1), \ldots, T_{n_s}(\lambda_s)$ for which $\prod_i \lambda_i \in I$. 
The definition of $D_{m,0}(\lambda)$ in \eqref{E:relinWforHeis}
implies that $\Wc^+(I) \subset A$. 
We'll prove that this is an equality. 

Let $\widetilde{W}^+(\oS)$ 
be the free algebra generated by the subset 
$$\{\tilde{T}_{n}(\lambda)\;:\; \lambda \in H^*(\oS,\Q), n\geq 0\},$$
modulo the relations $\tilde{T}_n(a\lambda + b\mu)=a\tilde{T}_n(\lambda) + b \tilde{T}_n(\mu)$. Thus $\widetilde{W}^+(\oS)$ is a free algebra in generators $\{\tilde{T}_{n}(\lambda)\}$, for $\lambda$ running through a basis of $H^*(\oS,\Q)$ and $n\geq 0$.
There is an obvious morphism 
$$\pi:\widetilde{W}^+(\oS) \to W^+(\oS)$$ whose kernel $R$ is generated by the 
relations \eqref{W:e}-\eqref{W:g} of \S\ref{sec:Fock-ops}. 
The order filtration $F_\bullet$ on $W^+(\oS)$ lifts to 
a filtration $\widetilde{F}_\bullet$ on $\widetilde{W}^+(\oS)$.
A Lie word $L(\tilde{T}_{n_1}(\lambda_1), \ldots, \tilde{T}_{n_s}(\lambda_s))$ belongs to 
$\widetilde{F}_n$ where $n=1-s+\sum_i n_i$ is the order of the word. 
Let $\widetilde{A} \subset \widetilde{W}^+(\oS)$ be the subalgebra generated by Lie words as above for which 
$\prod_i \lambda_i \in I$. Thus $\pi(\widetilde{A})=A$. We claim that 
\begin{equation}\label{E:proofW(I)}
\pi(\widetilde{F}_n \cap \widetilde{A})=F_n \cap A. 
\end{equation}
This is equivalent to the equality
\begin{equation}\label{E:proofW(I)2}
(\widetilde{F}_n \cap \widetilde{A}) + R=(\widetilde{A} + R) \cap \widetilde{F}_n + R.
\end{equation}
The only relations which do not preserve the order are those of type \eqref{W:f}. By the symbol of an element $x \in \widetilde{W}^+(\oS)$ we mean the class of $x$ in $\widetilde{F}_n/\widetilde{F}_{n-1}$, for $n$ the smallest integer for which $x\in \widetilde{F}_n$.
Now, if the symbol of a relation of type $(f)$ belongs to the associated graded $\text{Gr}_\bullet\widetilde{A}$ of $\widetilde{A}$ (with respect to the order filtration)
then the relation itself belongs to $\widetilde{A}$, because if $\lambda\mu \in I$ then $s_2\lambda\mu \in I$ and 
$c_1\Delta \lambda \mu \in I \otimes I$. Equations \eqref{E:proofW(I)2} and \eqref{E:proofW(I)} follow.

\smallskip

By \eqref{E:w0toGrW} and \eqref{E:w0toGrWLie} 
we have an isomorphism
$\Sym(\mathfrak{w}_0^+(\oS))\to \text{Gr}_\bullet W^+(\oS)$ and
there is a Lie algebra morphism $\mathfrak{w}^+_0(\oS) \to \text{Gr}_\bullet W^+(\oS)$. 
Since the symbol of a Lie word $L(T_{n_1}(\lambda_1), \ldots, T_{n_s}(\lambda_s))$ of order $n$
is a multiple of $D_{s,n}(\prod_i \lambda_i)$, we deduce from \eqref{E:proofW(I)} that
\begin{equation}\label{E:proofW(I)3}
\text{Gr}_n(A)=\Span\Big\{D_{m_1,n_1}(\lambda_1) \cdots D_{m_s,n_s}(\lambda_s)\;:\;\sum_i n_i = n, \lambda_1, \ldots, \lambda_s \in I\Big\}.
\end{equation}
Since $D_{m,n}(\lambda)\in\Wc^+(I)$ if $\lambda \in I$, we deduce that 
$\text{Gr}_\bullet (A) \subset \text{Gr}_\bullet (\Wc^+(I))$.
Hence $\text{Gr}_\bullet (A) = \text{Gr}_\bullet (\Wc^+(I))$.
Thus $A=\Wc^+(I)$. The formula for the Hilbert series of $\Wc^+(I)$ follows from \eqref{E:proofW(I)3}.

\smallskip

Let us now turn to the second equality. Let $W^+(J)$ denote the graded algebra defined by generators and relations as in \S\ref{sec:Fock-ops}, but with $J$ in place of $H^*(\oS,\Q)$. The results of \S\ref{sec:Fock-ops} may be repeated mutatis mutandis for $W^+(J)$. In particular, we have
$$P_{W^+(J)}(z,w) = \Exp\left(\frac{P_J(z)z^{-2}w}{(1-z^2)(1-w)}\right).$$
There is a canonical surjective morphism $W^+(\oS) \to W^+(J)$, which factors to a surjection 
$\Wn^+(J) \to W^+(J)$. 
Thus $P_{W^+(J)}(z,w) \leq P_{\Wn^+}(z,w)$. We will show that this map is an 
isomorphism by proving the reverse inequality. 
Let $\mathcal{I} \subset W^+(\oS)$ be the two-sided ideal generated by $\Wc^+(I)$. 
Recall that $D_{m,n}(\mu) \in \Gr_\bullet \Wc^+(I)$ for $\mu \in I$ and any $m\geq 1,$  $n \geq 0$.
Hence $\Sym(\mathfrak{w}_0^+(I)) \subset \Gr_\bullet \mathcal{I}$, where
$$\mathfrak{w}^+_0(I)=\Span\{D_{m,n}(\mu)\;:\; m \geq 1, n \geq 0,\; \mu \in I\}.$$
We deduce that
$$P_{\Wn^+(J)}(z,w) \leq P_{U(\mathfrak{w}^+_0(\oS))}(z,w) / P_{U(\mathfrak{w}^+_0(I))}(z,w)=P_{U(\mathfrak{w}^+_0(J))}(z,w)=P_{W^+(J)}(z,w)$$
which gives the desired reverse inequality.
\end{proof}

\medskip

In the course of the proof, we obtained the following characterization of $\Wc^\pm(I)$.

\begin{corollary}\label{Cor:wccharacterization} The subalgebra $\Wc^\pm(I)$ is generated by all Lie words
$$L(T^\pm_{n_1}(\lambda_1), \ldots, T^\pm_{n_s}(\lambda_s))$$
for which $\prod_i \lambda_i \in I$.\qed
\end{corollary}

\medskip

The commutator with $\psi_n(\lambda)$ preserves $\Wc^\pm(I)$ for any integer $n$ and any cohomology class
$\lambda\in H^*(\oS,\Q)$. 
This induces an action of $W^0(\oS)$ on both $\Wc^\pm(I)$ and $\Wn^\pm(J)$. 

\begin{lemma}\label{L:W0onWnWc} \hfill
\begin{enumerate}[label=$\mathrm{(\alph*)}$,leftmargin=8mm,itemsep=1.2mm]
\item
The action of $W^0(\oS)$ on $\Wn^\pm(J)$ factors through $W^0(J)$. 
\item
The action of $W^0(\oS)$ on $\Wc^\pm(I)$ factors through $W^0(H^*(\oS,\Q)/I^\perp)$. 
\end{enumerate}
\end{lemma}

\begin{proof}
The first statement is a consequence of the inclusion $[\psi_l(\mu)),W^\pm(\oS)] \subset \Wc^\pm(I)$ if $\mu \in I$, which may be checked on the generators $T_n(\lambda)$. The second statement is a consequence of the following claim: for any $l \geq 0$ and $\mu \in I^\perp$ we have
\begin{equation}\label{E:prooflemmaW0WcWn}
[\psi_l(\mu),\Wc^\pm(I)]=\{0\}.
\end{equation}
We sketch the proof of this claim, leaving the details to the reader. Let $\widetilde{W}^+(\oS), \widetilde{F}_\bullet$ 
and $\pi: \widetilde{W}^+(\oS) \to W^+(\oS)$ be as in the proof of Proposition~\ref{prop:W-cpt-open}, and let $R \subset \widetilde{W}^+(\oS)$ be the ideal of relations, $R_n=R \cap \widetilde{F}_n$. For a Lie word 
$L(\tilde{T}_{n_1}(\lambda_1), \ldots, \tilde{T}_{n_s}(\lambda_s))$ we put $c(L)=\prod_i \lambda_i \in H^*(\oS,\Q)$. 
From relations \eqref{W:e}, \eqref{W:f} and \eqref{W:g} one checks the following: for any $r \in R_n/R_{n-1}$ which is a linear combination of products of Lie words $L_1, \ldots, L_t$ such that $c(L_1)= \cdots = c(L_t)=\alpha$ there 
exists a lift $r' \in R_n$ of $r$ which is a linear combination of products of Lie words $L'_1, \ldots, L'_s$ satisfying 
$c(L'_i) \in \alpha H^*(\oS,\Q)$ for all $i$. In particular, for any two Lie words $L_1,L_2$ for which $c(L_1)=c(L_2)=0$ and $L_1-L_2 \in R_n/R_{n-1} \subset \widetilde{F}_n/\widetilde{F}_{n-1}$ we have 
\begin{equation}\label{E:proofLemmaWO2}
\pi(L_1-L_2) \in \Span\Big\{ \pi(L'_1 \cdots L'_s)\,:\, \forall\;i,\;c(L'_i)=0, \sum_i o(L_i)<n\Big\}.
\end{equation}
For any Lie word
$L\in \widetilde{W}^+(\oS)$, the symbol of $\pi(L)$ with respect to $F_\bullet$ is equal to a multiple of 
$D_{m,n}(\lambda)$ for some $m,n$ and $\lambda$. As $D_{m,n}(0)=0$, 
we deduce from \eqref{E:proofLemmaWO2} by induction on the order that for any Lie word $L$ with $c(L)=0$ 
we have $\pi(L)=0$. This yields \eqref{E:prooflemmaW0WcWn} as a particular case.
\end{proof}

\medskip

Now, we consider the ideals 
$$I_S=I(S)^\perp=H^*_c(S,\Q)
,\quad
I_S^\perp = I(S)=H^*(\oS,S,\Q)$$ and the quotient 
$$J_S=H^*(S,\Q) = H^*(\oS,\Q)\,/\,I_S^\perp.$$ 
Since $S$ is pure, the maps 
$\iota_!: H^*_c(S,\Q) \to H^*(\oS,\Q)$ and $\iota^* : H^*(\oS,\Q) \to H^*(S,\Q)$ are respectively injective and surjective.

\smallskip

\begin{definition}
We abbreviate
    \begin{gather*}
        \Wc^{\pm}(S)=\Wc^\pm(I_S), \quad \Wn^{\pm}(S)=\Wn^\pm(J_S), \quad W^0(S)=W^0(J_S) \simeq \Lambda(S).
    \end{gather*}
Set $\Wn^{\geq 0}(S)=W^0(S) \ltimes \Wn^+(S)$ and $  \Wc^{\geq 0}(S)=W^0(S) \ltimes \Wc^+(S).$
We define $\Wn^{\leq 0}(S)$ and $ \Wc^{\leq 0}(S)$ in the same way.
\end{definition}

\smallskip

Composing the inclusion $\Wc^\pm(S) \subset W^\pm(\oS)$ with the projection $W^\pm(\oS) \to \Wn^\pm(S)$ yields an algebra homomorphism 
 $$\varphi_S: \Wc^\pm(S) \to \Wn^\pm(S).$$
 
\smallskip

\begin{corollary}
    The Hilbert series of $\Wc^+(S), \Wn^+(S)$ are given by
\begin{align}
    P_{\Wc^+(S)}(z,w) &= \Exp\left(\frac{P_S^c(z)z^{-2}w}{(1-z^2)(1-w)} \right), \\
    P_{\Wn^+(S)}(z,w) &= \Exp\left(\frac{P_S(z)z^{-2}w}{(1-z^2)(1-w)} \right)   
\end{align} 
where $P^c_S(z)=\sum_n \dim(H_c^n(S,\Q))(-u)^n$ and $P_S(z)=\sum_n \dim(H^n(S,\Q))(-u)^n$.
\end{corollary}

\medskip

\subsection{Full $W$-algebras} There are four  $W$-algebras associated with $S$, namely
$$\wnn^{(\c)}(S), \quad \wnc^{(\c)}(S), \quad \wcn^{(\c)}(S), \quad \wcc^{(\c)}(S). $$
These depend on the choice of compactly supported or ordinary homology for each half.
They are defined as subquotients of $W^{(\c)}(\oS)$. 
For instance, $\wnn^{(\c)}(S)$ is the quotient of $W^{(\c)}(\oS)$ by the two-sided ideal generated by 
$\Wc^\pm(I_S^\perp)$ and $W^0(I_S^\perp)$, while $\wnc^{(\c)}(S)$ is the quotient of the subalgebra of 
$W^{(\c)}(\oS)$ generated by $W^\leq(\oS),$ $ \Wc^+(I_S)$ by its two-sided ideal generated by $W^0(I_S^\perp)$ 
and $\Wc^-(I_S^\perp)$. It is easy to prove that $\wnn^{(\c)}(S)$ may be presented just as $W^{(\c)}(\oS)$, 
where we replace $H^*(\oS,\Q)$ by $H^*(S,\Q)$ everywhere.

For any $\kappa \in \Q$, let $\wnn^{(\kappa)}(S)$ be the specialization of $\wnn^{(\c)}(S)$ at $\c=\kappa, e^{i\pi\mathbf{c}}=e^{i\pi\kappa}$, etc. Observe that $[T^-_0(\lambda),T_0^+(\mu)]=-e^{i\pi \c}\c$ when $\lambda\mu=[pt]$ so that $\c$ belongs to the two-sided ideal of $W^{(\c)}(\oS)$ generated by $\Wc^\pm(I_S^\perp)$. In other words, any specialization $\wnn^{(\kappa)}(S)$ at any $\c=\kappa \neq 0$ is zero.

\begin{proposition}\label{prop:triang-open}
For any $a,b \in \{\uparrow, \downarrow\}$, the multiplication yields an isomorphism
$$W^{(\c)}_{ab}(S) \simeq \begin{cases} W^-_a(S) \otimes W^0(S) \otimes W^+_b(S) \qquad &(a,b) \neq (\uparrow,\uparrow),\\ \Wn^-(S) \otimes W^0(S)|_{\c=0} \otimes \Wn^+(S) \qquad &(a,b) = (\uparrow,\uparrow)\end{cases}.$$
\end{proposition}

\begin{proof} Let $\mathcal{I}^{\geq}, $ $\mathcal{I}^{\leq}$ be the two-sided ideals of $W^\geq(\oS)$ and $W^\leq(\oS)$ generated by $\Wc^\pm(I_S^\perp)$ and $\{\psi_l(\lambda)\,:\, l >0, \lambda \in I_S^\perp\}$. 
Arguing as in the proof of Lemma~\ref{L:W0onWnWc} and using Proposition~\ref{prop:big-W-triang}, it is enough to check the following inclusions
\begin{align*}
\mathcal{I}^\geq \cdot W^-(\oS) &\,\subseteq\, W^-(\oS) \cdot \mathcal{I}^\geq + \c W^{(\c)}(\oS), \\
W^+(\oS) \cdot \mathcal{I}^\leq &\,\subseteq\, \mathcal{I}^\leq \cdot W^+(\oS) + \c W^{(\c)}(\oS), \\
\Wc^+(S) \cdot \Wc^-(S) &\,\subseteq\, \Wc^-(S) \cdot W^0(\oS)\cdot \Wc^+(S), \\
\Wc^+(S) \cdot \mathcal{I}^\leq &\,\subseteq\, \mathcal{I}^\leq \cdot \Wc^+(S) 
\end{align*}
These are in turn easily verified using \eqref{E:doublerelation}, Corollary~\ref{Cor:wccharacterization} together with the fact that $\mathcal{I}^\pm$ is generated by $T^\pm_n(\lambda)$ for $\lambda \in I_S^\perp$.

\end{proof}

\medskip

To finish, we mention the following analogue of Theorem~\ref{thm:big-W-undef}. For a pair $(a,b) \in \{\uparrow,\downarrow\}^2$, we define a Lie algebra $\mathfrak{w}^{(\c)}_{ab}(S)$ as follows. It has a basis given by a central element $\c$ and elements $D_{m,n}(\lambda)$ where $m\in\Z$, $n\geq 0$, and $\lambda$ belongs to $H^*(S,\Q)$ or $H^*_c(S,\Q)$ according to the following rule: 
\begin{itemize}
    \item[-] For $m=0$, $\lambda \in H^*(S,\Q)$;
    \item[-] For $m<0$, $\lambda \in H^*(S,\Q)$ if $a=\uparrow$ and $\lambda \in H^*_c(S,\Q)$ otherwise;
    \item[-] For $m>0$, $\lambda \in H^*(S,\Q)$ if $b=\uparrow$ and $\lambda \in H^*_c(S,\Q)$ otherwise.
\end{itemize}
The Lie bracket is given by
\[
    [D_{m,n}(\lambda),D_{m',n'}(\mu)] = (nm'-m'n) D_{m+m',n+n'-1}(\lambda\mu),
\]
where we have set $D_{m,-1}=\delta_{m,0}\c$, and where the product $\lambda \mu$ is induced either by the cup product on $H^*(S,\Q)$ or $H^*_{c}(S,\Q)$, or by the cap product $H^*(S,\Q) \otimes H^*_c(S,\Q) \to H_c^*(S,\Q)$ (or its composition with the natural map $H^*_c(S,\Q) \to H^*(S,\Q)$).

\smallskip

\begin{corollary}\label{cor:nondef-bbW}
    Assume that $s_2=0$ and $c_1\Delta=0$. There is an algebra isomorphism 
    $$\Phi:W_{ab}^{(\c)}(S)\simeq \begin{cases} U(\mathfrak{w}^{(\c)}_{ab}(S)) & (a,b) \neq (\uparrow,\uparrow)\\
    U(\mathfrak{w}^{(\c)}_{\uparrow\mathrel{\mspace{-2mu}}\uparrow}(S))|_{\c=0} & (a,b) = (\uparrow,\uparrow)
    \end{cases}$$
    which sends $\psi_n(\lambda)$ to $D_{0,n}(\lambda)$ and $T^\pm_n(\lambda)$ to $D_{\pm 1,n}(\lambda)$.
\end{corollary}
\begin{proof}
For the $\uparrow$-halves, this follows from the explicit identification $\mathcal{I}^\perp_S=\Span\{D_{m,n}(\lambda)\; : \pm m >0, n \geq 0, \lambda \in I^\perp_S\}$.
For the $\downarrow$-halves, we know that $\Wc^\pm(I_S)$ is preserved inside $W^\pm(\overline{S})$ under $\Ad_{T_0(1)}$ by Corollary~\ref{Cor:wccharacterization}.
In particular, we can apply the proof of Theorem~\ref{thm:W-undef} word for word to conclude.
Finally, the commutation relations between two halves are checked recurrently as in Proposition~\ref{prop:heis-in-double}.
\end{proof}

\smallskip

\begin{example}
    Let $S=\mathbb{A}^2$ together with the obvious action of a two-dimensional torus $T$.
    The Chern roots $t_1$, $t_2$ are the linear characters of $T$, and we have $\Delta = t_1t_2$.
    We identify $H^*_{T}(\mathbb{A}^2)=\mathbb{Q}[t_1,t_2]$
    and $H^*_{c,T}(\mathbb{A}^2) = t_1t_2 \mathbb{Q}[t_1,t_2]$.
    The difference between $\Wn^{+,T}(\mathbb{A}^2)$ and $\Wc^{+,T}(\mathbb{A}^2)$ is 
    multiplication of generators $T_i^+=T^+_i(1)$ by $t_1t_2$. We will abuse the notation and omit the subscripts.

Let us compare the output of our constructions with the affine Yangian $\ddot{Y}_{t_1,t_2}(\mathfrak{gl}_1)$ in~\cite{Tsymbaliuk}, see also \cite{AS}. 
 It is generated by elements $e_i,f_i,\widetilde{\psi}_i$ with $i\geq 0$, modulo the relations (Y0-Y6) in~\cite{Tsymbaliuk}. We define a morphism of algebras $\ddot{Y}_{t_1,t_2}(\mathfrak{gl}_1)\to \DW{\mathbf{c}}_T(\mathbb{A}^2)$ by sending
    \[
        e_i\mapsto T_i^+,\quad f_i\mapsto T_i^-.
    \]
Then $\widetilde{\psi}_n$ maps to the expression on the right hand side of~\eqref{eq:comm-pm}.
In particular $\widetilde{\psi}_i$ is different from the element $\psi_i\in \DW{\mathbf{c}}_T(\mathbb{A}^2)$.
    The only relations that are not obvious are (Y4-Y5).
    Since they are almost identical, let us concentrate on (Y4), which is
    \begin{align*}
        [\widetilde{\psi}_{i+3},e_j] - &3[\widetilde{\psi}_{i+2},e_{j+1}] + 3[\widetilde{\psi}_{i+1},e_{j+2}] - [\widetilde{\psi}_{i},e_{j+3}]\\
        &+ (t_1^2+t_1t_2+t_2^2)([\widetilde{\psi}_{i+1},e_{j}]-[\widetilde{\psi}_{i},e_{j+1}])+ t_1t_2(t_1+t_2)\{\widetilde{\psi}_{i},e_j\} = 0
    \end{align*}
    The relation (Y1) is the relation \eqref{W:f} in our Definition~\ref{def:W}.
    Using the relation $[e_i,f_j]=\widetilde{\psi}_{i+j}$ and applying the commutator with $f_i$ to (Y1), 
    proving (Y4) for $(i,j)$ is equivalent to proving it for $(i-1,j+1)$.
    Thus it suffices to check (Y4) for $i=0$.
    Using (Y4'), the formula we need to check is
    \[
        [[T^+_3,T^-_0],T^+_i] = 6T^+_{i+1}+2\mathbf{c}t_1t_2(t_1+t_2)T^+_i.
    \]
    Unpacking the right hand side of formula~\eqref{eq:comm-pm} for $n=3$, we can get an explicit formula for $[T^+_3,T^-_0]$ and check the relation above applying~\eqref{eq:pT-comm}.

    Using the triangular decomposition of both sides, it is easy to see that the map above is actually an isomorphism $\ddot{Y}_{t_1,t_2}(\mathfrak{gl}_1)\simeq \DW{\mathbf{c}}_T(\mathbb{A}^2)$. 
\end{example}

\medskip

\subsection{Fock space representations} In \S\ref{sec:Fockspacedefinition}, we defined a representation $\mathbf{F}^{(r)}(\oS)$ of $W^{(\c)}(\oS)$. We now consider variants of $\mathbf{F}^{(r)}(\oS)$ in the case of open surfaces.

The restriction $H^*(\overline{S},\Q)\to H^*(S,\Q)$ is surjective.
It yields a surjective map $\iota^*:\mathbf{F}^{(r)}(\overline{S})\to \mathbf{F}^{(r)}(S)$.
The kernel of $\iota^*$ is generated by the tautological classes $\psi_n(\lambda)$ with $\lambda \in I_S^\perp$.
The formulas~\eqref{eq:T-fla+}, \eqref{eq:T-fla-} and the definition of the operators $R^\pm, Q^\pm$ imply that $\ker(\iota^*)$ is a $\DW{r}(\overline{S})$-submodule of $\mathbf{F}^{(r)}(S)$ and that dually $T^\pm_n(\lambda)\mathbf{F}^{(r)}(\overline{S}) \subset\ker(\iota^*)$ for any $\lambda \in I_S^\perp$. From these observations, one deduces the following:

\smallskip

\begin{proposition}\label{Prop:FockspaceopenW} The action of $\DW{r}(\oS)$ on the Fock space $\mathbf{F}^{(r)}(\oS)$ descends/restricts to an action of $\wnc^{(r)}(S)$, $\wcn^{(r)}(S)$, $\wcc^{(r)}(S)$ on $\mathbf{F}^{(r)}(S)$ for any $r$, and to an action of $\wnn^{(0)}(S)$ on $\mathbf{F}^{(0)}(S)$. These representations are faithful when $r \neq 0$.
\end{proposition}
\begin{proof}
    Let us only address the faithfulness.
    We only need to consider $\wnc^{(r)}(S)$, $\wcn^{(r)}(S)$, since $\wcc^{(r)}(S)$ is a subalgebra of $\DW{r}(\oS)$.
    Similarly to Proposition~\ref{prop:heis-in-double}, these algebras contain a copy of the Heisenberg algebra $\mathfrak{h}_S$ of central charge $r$, where half of the generators are labeled by $H^*_c(S,\Q)$.
    Since $r>0$, we conclude by Lemma~\ref{lem:Heis-faithful}.
\end{proof}

\medskip

\section{Hecke patterns}\label{sec:Hecke-patterns}
In this section we introduce a general framework to construct modules over the algebras $\Hb_0(S)$ and $\Hb^c_0(S)$ by using a compactification $\oS$ of $S$ and by restricting the multiplication in $\Hb(\oS)$ to suitable substacks. For this we fix a smooth compactification $\oS$ of $S$. 

\subsection{Hecke patterns and Hecke correspondences}\label{sec:Heckepatterns} The following is a variation on the notion of a Hecke pattern which appears in \cite[\S 5]{KV}.

\begin{definition}
    A two-sided Hecke pattern on $\oS$ is a locally closed derived substack $X = \bigsqcup_\alpha X_\alpha$ of $\RCoh^{\geq 1}(\oS)$ satisfying the following properties:
    \begin{enumerate}[label=$\mathrm{(\alph*)}$,leftmargin=8mm]
        \item for any short exact sequence $0 \to \mathcal{E} \to \mathcal{F} \to \mathcal{G} \to 0$ with $\mathcal{G} \in \RCoh^0(\oS)$ and
        $\mathcal{F} \in X$ we have $\mathcal{E} \in X$, 
        \item for any short exact sequence $0 \to \mathcal{E} \to \mathcal{F} \to \mathcal{G} \to 0$ with $\mathcal{G} \in \RCoh^0(\oS)$, $\mathcal{E} \in X$ and $\mathcal{F} \in \RCoh^{\geq 1}(\oS)$ we have $\mathcal{F} \in X$.
    \end{enumerate}
    We call $X$ an $S$-weak Hecke pattern if the conditions (a-b) only hold for $\mathcal{G}\in \RCoh^0(S)$.
    We say that a Hecke pattern is $S$-strong if both conditions (a-b) imply that $\mathcal{G} \in \RCoh^0(S)$.
\end{definition}

We say that $X$ is of rank $r$ if $X_\alpha$ is nonempty only if $\text{rk}(\alpha)=r$.
Every Hecke pattern is a disjoint union of Hecke patterns of a well defined rank.
Note that the conditions of being a usual/$S$-strong/$S$-weak Hecke pattern imply the conditions~(\ref{eq:T-plus-condition},\ref{eq:T-minus-condition}), (\ref{eq:T-open-conditions}), (\ref{eq:T-cpt-conditions}) respectively.

\smallskip

\begin{remark} A substack $X \subset \RCoh^{\geq 1}(\oS)$ satisfying (a) alone 
is a left Hecke pattern, while a substack 
$X \subset\RCoh^{\geq 1}(\oS)$ satisfying (b) is a right Hecke pattern.
We can consider the $S$-weak/strong versions of these notions separately.
\end{remark}

\begin{example}\label{ex:Hecke-patterns}
    The property of being of dimension $\geqslant d$, $d=1,2$ is stable by passing to a subsheaf, and for any extension $0 \to \mathcal{E} \to \mathcal{F} \to \mathcal{G} \to 0$ with $\dim\mathcal{G}=0$, $\mathcal{E} \in \RCoh^{\geq d}$ and $\mathcal{F} \in \RCoh^{\geq 1}$ we have $\mathcal{F} \in \RCoh^{\geq d}$.
    Therefore $\RCoh^{\geq d}(\oS)$, $d=1,2$ is a Hecke pattern.
    One can similarly see that $\RCoh^{\geq 1}(S)$ is an $S$-strong Hecke pattern. More generally, any $S$-weak Hecke pattern contained in $\RCoh^{\geq 1}(S)$ is automatically $S$-strong.
    The collection of Hilbert schemes of points on $S$ (see \S\ref{sec:Hilbert}) is a left $S$-weak and right $S$-strong Hecke pattern.
\end{example}

\begin{lemma}\label{lem:strong-Hecke-equiv}
    A left Hecke pattern is $S$-strong if and only if $X\subset \RCoh^{\geq 1}(S)$.
    A right Hecke pattern is $S$-strong if and only if every sheaf in $X$ is locally free at any point in $\oS\setminus S$.
\end{lemma}
\begin{proof}
    The first claim follows from the fact that for any $x\in S$, any sheaf $\mathcal{E}$ with $\mathcal{E}_x\neq 0$ admits a surjection to $\mathcal{O}_x$.
    For the second claim, let $\mathcal{E}_1\subset \mathcal{E}$ be the maximal 1-dimensional subsheaf. Since $\Ext^1(\mathcal{O}_x,\mathcal{E}_1)\to\Ext^1(\mathcal{O}_x,\mathcal{E})$ is injective, $\mathcal{E}_1$ has to be supported away from $\oS\setminus S$. We conclude by observing that the cokernel of the double dual map $\mathcal{E}|_{\oS\setminus \supp(\mathcal{E}_1)}\hookrightarrow (\mathcal{E}|_{\oS\setminus \supp(\mathcal{E}_1)})^{\vee\vee}$ cannot have support in $\oS\setminus S$ either.
\end{proof}

\smallskip

An $S$-weak two-sided Hecke pattern $X$ yields two families of induction diagrams. 
We define
$$\widetilde{\RCoh}^\circ_{n\delta;\alpha}=\widetilde{\RCoh}_{n\delta;\alpha}(\oS) \underset{\RCoh_{n\delta}(\oS)}{\times} \RCoh_{n\delta}
,\quad
\widetilde{X}_{n\delta;\alpha}=\widetilde{\RCoh}^\circ_{n\delta;\alpha} \underset{\RCoh_{n\delta + \alpha}(\oS)}{\times} X_{n\delta +\alpha}.$$

Conditions (a) and (b) yield the following 
diagrams with Cartesian squares 
\begin{equation}\label{E:poshecke}
\begin{gathered}
\xymatrix{
\RCoh_{n\delta} \times \RCoh_\alpha(\oS) & \widetilde{\RCoh}^\circ_{n\delta;\alpha} \ar[r]^-{p^\circ_{n\delta,\alpha}} \ar[l]_-{q^\circ_{n\delta,\alpha}} & \RCoh_{\alpha+n\delta}(\oS)\\
\RCoh_{n\delta} \times X_\alpha \ar[u]& \widetilde{X}_{n\delta;\alpha} \ar[r]^-{\pi_{n\delta,\alpha}} \ar[l]_-{\kappa_{n\delta,\alpha}} \ar[u]& X_{\alpha+n\delta} \ar[u]
}\\
\xymatrix{
 \RCoh^{\geq 1}_\alpha(\oS) \times \RCoh_{n\delta}  & \widetilde{\RCoh}^{\circ,\geq 1}_{n\delta;\alpha-n\delta} \ar[r]^-{\overline{p}^\circ_{\alpha,n\delta}} \ar[l]_-{\overline{q}^\circ_{\alpha,n\delta}} & \RCoh^{\geq 1}_{\alpha-n\delta}(\oS)\\
 X_\alpha \times \RCoh_{n\delta}  \ar[u]& \widetilde{X}_{n\delta;\alpha-n\delta} \ar[r]^-{\overline{\pi}_{\alpha,n\delta}} \ar[l]_-{\overline{\kappa}_{\alpha,n\delta}} \ar[u]& X_{\alpha-n\delta} \ar[u]
}
\end{gathered}
\end{equation}

\smallskip

The maps $\pi_{n\delta,\alpha}$ and $\overline{\pi}_{\alpha,n\delta}$ may not be proper. Luckily, they factorise as
\begin{gather*}
    \widetilde{X}_{n\delta;\alpha} \xrightarrow{\pi'_{n\delta,\alpha}}  X_{\alpha+n\delta} \times \Sym^n(S) \xrightarrow{\pi''_{n\delta,\alpha}} X_{\alpha+n\delta},\\
    \widetilde{X}_{n\delta;\alpha-n\delta} \xrightarrow{\overline{\pi}'_{\alpha,n\delta}} X_{\alpha-n\delta} \times \Sym^n(S) \xrightarrow{\overline{\pi}''_{\alpha,n\delta}} X_{\alpha-n\delta}.
\end{gather*}

\begin{lemma} The following hold:
\begin{enumerate}[label=$\mathrm{(\alph*)}$,leftmargin=8mm]
\item the morphism $\kappa_{n\delta,\alpha}$ and $\overline{\kappa}_{\alpha,n\delta}$ are quasi-smooth,
\item the morphisms ${\pi}'_{n\delta,\alpha}$ and $\overline{\pi}'_{\alpha,n\delta}$ are proper and representable.
\end{enumerate}
\end{lemma}
\begin{proof}
Quasi-smoothness is preserved by derived base change. Hence $\kappa_{n\delta,\alpha}$ is a quasi-smooth morphism. Next, by construction $\widetilde{\RCoh}^{\geq 1}_{n\delta;\alpha-n\delta}$ is open in $\mathbb{V}(\RHom(\mathcal{E}_{\alpha},\mathcal{E}_{n\delta}))$ and as the sheaves over $S$ parametrized by $\RCoh^{\geq 1}$ have no zero-dimensional subsheaves, 
$$\Ext^2_S(\mathcal{E}_{\alpha},\mathcal{E}_{n\delta})|_{X_{\alpha} \times \RCoh_{n\delta}} = \Hom_S(\mathcal{E}_{n\delta},\mathcal{E}_{\alpha} \otimes K_S)|_{X_{\alpha} \times \RCoh_{n\delta}}=\{0\}.$$
Statement (a) follows. In order to prove (b), we introduce for all $\beta$ a full flag version of $\widetilde{X}_{n\delta;\beta}$ as follows:
$$\widetilde{X}_{\delta^n;\beta}=\{\mathcal{F}_n\subset \mathcal{F}_{n-1} \subset \cdots \subset \mathcal{F}_0\,:\, \mathcal{F}_0 \in X_{\beta+n\delta},\,\forall\; i, \mathcal{F}_i/\mathcal{F}_{i+1} \in \Coh_\delta\}.$$
Note that because $X$ is a Hecke pattern, each $\mathcal{F}_i$ belongs to $X$.
There are commutative diagrams
$$\xymatrix{
\widetilde{X}_{\delta^n;\alpha} \ar[r]^-{\gamma} \ar[d]_-{\varphi} & \widetilde{X}_{n\delta;\alpha} \ar[d]^-{\pi'_{n\delta;\alpha}}  & \widetilde{X}_{\delta^n;\alpha-n\delta} \ar[r]^-{\overline{\gamma}} \ar[d]_-{\overline{\varphi}} & \widetilde{X}_{n\delta;\alpha-n\delta} \ar[d]^-{\overline{\pi}'_{n\delta;\alpha}}\\
X_{\alpha+n\delta} \times S^n \ar[r]^-{\Id\times t} & X_{\alpha+n\delta} \times \Sym^n(S)  & X_{\alpha-n\delta} \times S^n \ar[r]^-{\Id\times t} & X_{\alpha-n\delta} \times \Sym^n(S)
}$$
where $\gamma, \overline{\gamma}, \varphi, \overline{\varphi}$ are obvious forgetful maps and $t$ is the projection. The map $t$ being finite, it is proper and representable. Moreover, as any length $n$ sheaf on $S$ admits a full flag of subsheaves, the morphisms $\gamma, \overline{\gamma}$ are proper, representable and surjective. Thus to prove that $\pi'_{n\delta,\alpha}$ and $\overline{\pi}'_{\alpha,n\delta}$ are proper and representable, it suffices to show that the same
holds for the maps $\varphi, \overline{\varphi}$.
In order to show this we consider the chains of forgetful morphisms 
\begin{equation*}
\xymatrix{
\widetilde{X}_{\delta^n; \beta} \ar[r]^-{\overline{\varphi}_{1}} & \widetilde{X}_{\delta^{n-1};\beta} \times S \ar[r]^-{\overline{\varphi}_{2}} & \cdots 
\ar[r]^-{\overline{\varphi}_{n-1}} & \widetilde{X}_{\delta;\beta} \times S^{n-1} \ar[r]^-{\overline{\varphi}_{n}} & X_{\beta} \times S^n
}
\end{equation*}
and
\begin{equation*}
\xymatrix{
\widetilde{X}_{\delta^n;\beta} \ar[r]^-{\varphi_{0}} & \widetilde{X}_{\delta^{n-1};\beta+\delta} \times S \ar[r]^-{{\varphi}_{1}} & \cdots 
\ar[r]^-{\varphi_{n-2}} & \widetilde{X}_{\delta;\beta+(n-1)\delta} \times S^{n-1} \ar[r]^-{\varphi_{n-1}} & X_{\beta +n\delta} \times S^n
}.
\end{equation*}
Observe that the composition of the first and the second chain for $\beta=\alpha-n\delta$ yields $\varphi$ and $\overline{\varphi}$ respectively. Thus we are reduced to checking that each $\varphi_i$ and $\overline{\varphi}_i$ is proper and representable.
By definition of a two-sided Hecke pattern, $\varphi_i$ and $\overline{\varphi}_i$ are respectively obtained by base change from the maps
$$\rho': \widetilde{\RCoh}^{\geq 1}_{\delta;\beta+i\delta} \to \RCoh^{\geq 1}_{\beta+(i+1)\delta} \times S, \qquad 
\overline{\rho}': \widetilde{\RCoh}^{\geq 1}_{\delta;\beta+(n-i)\delta} \to \RCoh^{\geq 1}_{\beta+(n-i)\delta} \times S.
$$
The maps $\rho', \overline{\rho}'$ above are respectively the projections from the projectivization 
$\mathbb{P}(\mathcal{E}_{\beta+(i+1)\delta})$ and $\mathbb{P}(\mathcal{E}_{\beta+(n-i)\delta}^\vee \otimes K_S[1])$. 
Since $\mathcal{E}$ and $\mathcal{E}^\vee \otimes K_S[1]$ are both of perfect amplitude $[0,1]$ over $X \times S$ it 
follows that $\rho',\overline{\rho}'$ are proper and representable (see \cite[Lemma~5.4]{QJiang}). The same therefore 
holds for $\varphi_i, \overline{\varphi}_i$, and we are done.
\end{proof}

\smallskip

Let $X=\{X_\alpha\}$ be an $S$-weak two-sided Hecke pattern. By base change from the 
diagrams~\eqref{E:poshecke}, we may analyze length one Hecke correspondences for $X$ 
in the same fashion as for $\RCoh$. 
We will say that the two-sided Hecke pattern $X$ is \emph{regular} if there exist, for each $\alpha$, an open cover 
$X_\alpha=\bigcup_i U_\alpha^{(i)}$ and locally free resolutions $\mathcal{E}_\alpha^{(i)}$ of length two for which the 
condition \eqref{E:assumption} holds when restricted to $X$. Equivalently, $X$ is regular if and only if the morphisms 
$\kappa_{\delta,\alpha}^{cl}$ and $\overline{\kappa}_{\alpha,\delta}^{cl}$ 
in \eqref{E:poshecke} are lci and of the same dimension as their derived enhancements.

\medskip

\subsection{Hecke patterns and Hecke operators}
\label{sec:Heckeactions}
Let $X$ be an $S$-weak two-sided Hecke pattern. 
Set
$$\mathbf{V}(X)_{\alpha+\Z\delta}=\bigoplus_{n \in \mathbb{Z}} H_*(X_{\alpha+n\delta}, \Q).$$
Let $r$ be the projection 
$$r : H^c_*(\Sym^n(S),\Q) \to 
H^c_0(\Sym^n(S),\Q)=\Q.$$ 
We define the maps
\begin{equation}\label{E:twosidedheckeaction}
\begin{split}    
m_{\alpha,-n\delta}=r\circ(\overline{\pi}'_{\alpha,n\delta})_*\circ(\overline{\kappa}_{\alpha,n\delta})^!:H_*(X_{\alpha},\Q) \otimes H^c_*(\Coh_{n\delta},\Q) &\to H_*(X_{\alpha-n\delta},\Q),\\
m_{n\delta,\alpha}=r\circ(\pi'_{n\delta,\alpha})_*\circ
(\kappa_{n\delta,\alpha})^!:H^c_*(\Coh_{n\delta},\Q) \otimes H_*(X_\alpha,\Q) &\to H_*(X_{\alpha+n\delta},\Q).
\end{split}
\end{equation}
These maps have cohomological degree $2\langle \alpha,n\delta\rangle$ and $-2\langle n\delta,\alpha\rangle$ respectively.
The proof of the following proposition is analogous to Theorem~\ref{thm:fullCOHA-KV}.

\smallskip

\begin{proposition}\label{prop:negheckeactionsdef} 
The map $m_{n\delta,\alpha+m\delta}$ defines a left action 
$$\Psi^+_X: \Hb^c_0(S)\otimes\mathbf{V}(X)_{\alpha + \Z\delta}\to\mathbf{V}(X)_{\alpha + \Z\delta}.$$
The map $m_{\alpha+m\delta,-n\delta}$ defines a right action 
$$\Psi^-_X:\mathbf{V}(X)_{\alpha + \Z\delta}\otimes \Hb^c_0(S)\to\mathbf{V}(X)_{\alpha + \Z\delta}.$$
If $X$ is only a left/right $S$-weak Hecke pattern, we only have the left/right action.
\qed
\end{proposition}

\medskip

\begin{remark}\label{rmk:proper-action-COHA}\leavevmode\nolisttopbreak
\begin{enumerate}[label=$\mathrm{(\alph*)}$,leftmargin=8mm,itemsep=1.2mm]
\item 
Suppose that $X$ is a usual Hecke pattern.
We can write diagrams~\eqref{E:poshecke} with $\RCoh_{n\delta}$ replaced by $\RCoh_{n\delta}(\oS)$. In this case the maps $\pi_{n\delta,\alpha}$, $ \overline{\pi}_{\alpha,n\delta}$ are proper for all $n$ and $\alpha$, 
and the actions $\Psi^\pm_X$ can be lifted to $\Hb_0(\oS)$.
\item 
Similarly, if $X$ is an $S$-strong Hecke pattern, the maps $\pi_{n\delta,\alpha}$, $ \overline{\pi}_{\alpha,n\delta}$ become proper, and so the actions $\Psi^\pm_X$ can be further lifted to $\Hb_0(S)$.
\item
For $X=\RCoh^{\geq 1}$ and $n=1$, the maps $m_{\alpha,-n\delta},$  $m_{n\delta,\alpha}$ are the negative and positive Hecke operators defined in \S\ref{sec:derived Hecke}.
\end{enumerate}
\end{remark}

\medskip

By Example~\ref{ex:Hecke-patterns}, we have both a left and a right $\mathbf{H}_0(S)$-module structure on the space 
$$\Hb(S)^{\geq d}_{\alpha+\mathbb{Z}\delta}=\bigoplus_{n\in \mathbb{Z}}H_*(\Coh^{\geq d}_{\alpha+n\delta},\Q).$$
Similarly, we have left and right $\mathbf{H}_0(\oS)$-actions on $\Hb(\oS)^{\geq d}_{\alpha+\mathbb{Z}\delta}$.

\medskip

\subsection{Hecke patterns and tautological classes} 
Let $X$ be an $S$-weak two-sided Hecke pattern. Let 
$$\ev_\alpha : \Lambda(\oS) \to H^*(X_\alpha, \Q)$$
denote the restriction of tautological classes 
on $\RCoh_\alpha(\oS)$ to $X_\alpha$. 
This defines a representation 
$$\bullet:\Lambda(\oS)\times\mathbf{V}(X)_{\alpha+\Z\delta}\to\mathbf{V}(X)_{\alpha+\Z\delta}$$
Recall the involution $\upsilon$ in~\eqref{eq:involution}.
For any $c \in \Lambda(\oS),$ $u \in \Hb^c_0(S)$ and $z \in \mathbf{V}(X)_{\alpha+\Z\delta}$ we have
\begin{align*}
    c \bullet m_{n\delta,\alpha}(u \otimes z)&=\sum (-1)^{|c^{(2)}_i|\cdot |u|} m_{n\delta,\alpha}\Big((c^{(1)}_i \bullet u) \otimes (c^{(2)}_i \bullet z)\Big), \\
    c \bullet m_{\alpha,-n\delta}(u \otimes z)&=\sum (-1)^{|c^{(2)}_i|\cdot |u|} m_{\alpha,-n\delta}\Big((\upsilon(c^{(1)}_i) \bullet u) \otimes (c^{(2)}_i \bullet z)\Big).
\end{align*}
Thus, the left and right $\Hb^c_0(S)$-actions on $\mathbf{V}(X)_{\alpha+\Z\delta}$ extend to actions of 
the algebra $\widetilde{\Hb}^c_0(S)$ defined in \eqref{tildeHc}. 
The proof is identical to that of Proposition~\ref{prop:modulealgebra}.
Analogous statements hold for usual and $S$-strong Hecke patterns, and the actions defined in Remark~\ref{rmk:proper-action-COHA}.

\medskip

\subsection{The case of regular Hecke correspondences} 
Let $X$ be a two-sided Hecke pattern.
Set
$$H_*^\taut(X_\alpha,\Q)=\Lambda(\oS) \bullet [X_\alpha^{cl}]
,\quad
\mathbf{V}^\taut(X)_{\alpha+\Z\delta}=\bigoplus_{n\in\Z} H_*^\taut(X_{\alpha+n\delta},\Q).$$ 
We consider the linear map
$$\ev':\mathbf{F}(\oS)\to \mathbf{V}^\taut(X)_{\alpha+\Z\delta}
,\quad
xu^n\mapsto x\bullet[X_{\alpha+n\delta}^{cl}]
,\quad 
x\in\Lambda(\oS).
$$
Let $\End_X(\mathbf{F}(\oS))$ be the subspace of all 
endomorphisms of $\mathbf{F}(\oS)$
preserving the kernels of the maps $\ev'$.
Propositions~\ref{P:negut}, \ref{P:negutdual} and \ref{prop:psi-T-rels} yield the following.

\smallskip

\begin{proposition}\label{P:regularHP} Let $X$ be a regular two-sided Hecke pattern of rank $r$.  
\begin{enumerate}[label=$\mathrm{(\alph*)}$,leftmargin=8mm,itemsep=1.2mm]
\item 
There is a commutative diagram of homomorphisms 
$$\xymatrix{W^{\pm}(\oS) \ar[r] \ar@/^1pc/[rr]^-{\Phi^\pm} \ar[dr]_-{\Phi^\pm_X} & \End_X(\mathbf{F}(\oS)) \ar[r] \ar[d]^-{\ev'} & \End(\mathbf{F}(\oS))\\
& \End(\mathbf{V}^\taut(X)_{\alpha+\Z\delta})&}$$
in which $\Phi^\pm, \Phi^\pm_X$ are algebra homomorphisms.
\item
For any $\xi \in H^*(\oS)$ and $n \geq 0$, we have
\begin{align*}
    \Phi^+_X(T_n(\xi))=T_+(\xi u^{n-r+1})=\Psi^+_X(\xi u^{n-r+1} \cap [\Coh_\delta]),\\
    \Phi^-_X(T_n(\xi))=T_-(\xi u^{n+r+1})=\Psi^-_X(\xi u^{n+r+1} \cap [\Coh_\delta]).
\end{align*}
\item
The maps $\Phi_X^{\pm}$ glue into an algebra homomorphism $\Phi_X: \DW{r}(\oS) \to \End(\mathbf{V}^\taut(X)_{\beta+\Z\delta})$.
\item 
If $X$ is $S$-strong, then the assumptions above hold with $\oS$ replaced by $S$, and we get an algebra homomorphism $\Phi_X: \Wnn{r}(S) \to \End(\mathbf{V}^\taut(X)_{\beta+\Z\delta})$.
\end{enumerate}
\end{proposition}

When $X$ is only an $S$-weak Hecke pattern, regularity is not sufficient to ensure the actions of $\Wc^\pm(S)$ in our setup.
The reason is that $\Wc^\pm(S)$ is \emph{defined} as a subalgebra of $W^\pm(\oS)$; furthermore, it is typically not generated by degree 1 elements.
Because of this, we will have to postpone the proof of an analogous statement until the end of \S\ref{sec:Hilbert}, see Corollary~\ref{cor:regularHPcompact}.

\smallskip

Thanks to Proposition~\ref{prop:Negutlemmavirtual}, a result in all points analogous to Proposition~\ref{P:regularHP} holds without any regularity assumptions if we replace $\mathbf{V}(X)^{\taut}$ by its virtual cousin
$$\mathbf{V}^{\vtaut}(X_\alpha)=\Lambda(\oS) \bullet [X_\alpha].$$

\subsection{Base change for Hecke patterns} 
Let $S$ be a smooth surface and let $X$ be an $S$-weak two-sided Hecke pattern on $S$.

\begin{proposition}\label{prop:basechangeHP}  \leavevmode\nolisttopbreak
\begin{enumerate}[label=$\mathrm{(\alph*)}$,leftmargin=8mm]
\item Let $\iota: S^\circ \to S$ be an open embedding. Then $X$ is an $S^\circ$-weak two-sided 
Hecke pattern on $S^\circ$ and we have $\Psi^{\pm}_{X} \circ \iota_{!} =\Psi^{\pm}_{X}: \Hb^c_0(S^\circ) \to \End(
\mathbf{V}(X)_{\alpha+\Z\delta})$,

\item Let $j:X^\circ \to X$ be an open immersion and assume that $X^\circ$ is 
also an $S$-weak two-sided Hecke pattern on $S$. Then we have a commutative diagram
$$\xymatrix{ \Hb_0^c(S) \ar[r] \ar@/^1pc/[rr]^-{\Psi^{\pm}_X} \ar[dr]_-{\Psi^{\pm}_{X^\circ}} & \End_{X^\circ}(\mathbf{V}(X)_{\alpha+\Z\delta}) \ar[r] \ar[d]^-{\text{res}} & \End(\mathbf{V}(X)_{\alpha+\Z\delta})\\& \End(\mathbf{V}(X^\circ)_{\alpha+\Z\delta})&}$$
where $\End_{X^\circ}(\mathbf{V}(X)_{\alpha+\Z\delta})$ is the subset of endomorphisms preserving the kernel of the map 
$j^*: \mathbf{V}(X)_{\alpha+\Z\delta} \to \mathbf{V}(X^\circ)_{\alpha+\Z\delta}$.
\end{enumerate}   
\end{proposition}

\begin{proof} 
We treat the case of $\Psi^+$, the other being similar. Consider the diagram
$$\xymatrix{\RCoh_{l\delta}(S) \times X_\alpha & \widetilde{X} \ar[r]^-{\pi'_{l\delta,\alpha}} \ar[l]_-{\kappa_{l\delta,\alpha}} & X_{\alpha+l\delta}\times \Sym^l(S)\\
\RCoh_{l\delta}(S^\circ) \times X_\alpha \ar[u] & \widetilde{Z} \ar[u] \ar[r]^-{\pi^\#_{l\delta,\alpha}} \ar[l]_-{\kappa^\#_{l\delta,\alpha}} & X_{\alpha+l\delta}\times \Sym^l(S^\circ)\ar[u]}$$
where $\widetilde{Z}=\widetilde{X} \times_{X_{\alpha+l\delta} \times \Sym^l(S)} (X_{\alpha+l\delta} \times \Sym^l(S^\circ))$ and the vertical maps are all open embeddings. Both squares are cartesian by definition of Hecke patterns. Statement (a) follows from base change properties of the morphism $\iota_{!}$, see Proposition~\ref{P:basechange}. Now consider the diagram
$$\xymatrix{\RCoh_{l\delta}(S) \times X_\alpha  & \widetilde{Z}  \ar[r]^-{\pi^\#_{l\delta,}} \ar[l]_-{\kappa^\#_{l\delta,\alpha}} & X_{\alpha+l\delta}\times \Sym^l(S)\\
\RCoh_{l\delta}(S) \times X^\circ_\alpha\ar[u]& \widetilde{X}^\circ \ar[u]\ar[r]^-{\pi^{'\circ}_{n\delta,\alpha}} \ar[l]_-{\kappa^\circ_{l\delta,\alpha}} & X^\circ_{\alpha+l\delta}\times \Sym^l(S)\ar[u]}.$$
Again, both squares are cartesian, and the vertical arrows are all open embeddings induced by $j$. Statement (b) follows from proper base change in hyperbolic homology, see Lemma~\ref{L:projformula}.
\end{proof}

\medskip

\section{Action on Hilbert schemes}\label{sec:Hilbert}
In this section we construct actions of $\Hb_0(S)$ and $ \Hb^c_0(S)$ on the homology of the Hilbert scheme of points 
on $S$, and we explicitly describe the action on tautological classes using the results of \S\ref{sec:derived Hecke}. 
We assume everywhere that the surface $S$ is pure.

\subsection{The Hilbert scheme and stack} 
Let $\Pic$ be the derived stack of invertible coherent sheaves on $S$, $B\Gm \to \Pic$ the closed immersion of the 
substack parametrizing trivial invertible sheaves, and $\mathfrak{Perf}$ the derived stack of perfect complexes. 
 There is a morphism of derived stacks $\RCoh^{\geq 2}_{1,-n\delta} \to \mathfrak{Perf}_{1,-n\delta}$.
 Composed with the perfect determinant map $\mathfrak{Perf}_{1,-n\delta} \to \Pic$ defined in~\cite{STV15}, 
 it yields a morphism of derived stacks $\RCoh^{\geq 2}_{1,-n\delta} \to \Pic$.
We define the Hilbert stack of $S$ to be the derived fiber product 
$$\Hilb_n= \RCoh^{\geq 2}_{1,-n\delta} \underset{\Pic}{\times} B\Gm$$
We write $\Hilb=\bigsqcup_n \Hilb_n$.
Let $\Hilbrm_n$ be the Hilbert scheme of $n$ points on $S$, whose points parametrize ideal sheaves 
$\mathcal{I} \subset \mathcal{O}_{S}$ and of colength $n$. It is the coarse moduli space of $\Hilb_n$.

\medskip

\begin{lemma}\label{lem:hilbproperties1} The following statements hold:
\begin{enumerate}[label=$\mathrm{(\alph*)}$,leftmargin=8mm,itemsep=.5em]
\item $\Hilb_n$ is isomorphic to its classical truncation,
\item There is an isomorphism of stacks $\Hilb_n\simeq \Hilbrm_n \times B\mathbb{G}_m$.
\end{enumerate}
\end{lemma}

\begin{proof}
Since the morphism $\RCoh^{\geq 2}_{1,-n\delta} \to \Pic$ is $\Pic^{cl}$-equivariant, we have
\begin{equation*}
 \begin{split}  
\text{vdim}(\RHilb_n)&=\text{vdim}(\RCoh^{\geq 2}_{1,-n\delta})-\text{vdim}(\Pic)-1\\
 &=\langle [\mathcal{O}_{S}] -n\delta\,,\, [\mathcal{O}_{S}]-n\delta \rangle - \langle [\mathcal{O}_{S}], [\mathcal{O}_{S}] \rangle -1\\
 &=2n-1.
 \end{split}
\end{equation*} 
The classical truncation of $\RHilb_n$ parametrizes colength $n$ ideal subsheaves $\mathcal{I} \subset \mathcal{O}_{S}$. 
It is smooth and of dimension $2n-1$. 
The stack $\RHilb_n$ is quasi-smooth. 
Part (a) follows using the fact that a quasi-smooth derived stack with smooth classical truncation and whose virtual dimension coincides with that of its classical truncation is underived. 

We now turn to (b). 
The stack $\Hilb_n$ parametrizes some simple rank 1 sheaves over the proper surface $\oS$. 
Hence it is a $\mathbb{G}_m$-gerbe over its coarse moduli space $\Hilbrm_n$.
Fixing a splitting of this gerbe is the same as choosing a universal sheaf $\mathcal{U}_n$ on $\Hilbrm_n \times S$. 
But such a canonical sheaf is given by a subsheaf of $\mathcal{O}_{\Hilbrm_n} \boxtimes \mathcal{O}_{S}$.
\end{proof}

\smallskip
Fix a smooth compactification $\iota: S \to \oS$. 
We define $\RHilb_n(\oS)$, $\Hilbrm_n(\oS)$ and $\Pic(\oS)$ as above.
Let $\rho$ be the degree one line bundle on $\Hilb_n(\oS)$ pulled back from $B\mathbb{G}_m$. 
Let $\mathcal{U}_n$ and $\mathcal{I}_n= \mathcal{U}_n \boxtimes \rho$  
be the universal ideal sheaves in
$Coh(\Hilbrm_n(\oS) \times \oS)$ 
and
$Coh(\Hilb_n(\oS) \times \oS)$ 
respectively. Let $\ev$ denote the evaluation  morphism \eqref{eval}
$$\ev:\Lambda(\oS)\to H^*(\RHilb_n(\oS),\Q).$$

\medskip

\begin{lemma}\label{lem:hilbproperties2} We have 
$c_1(\rho)
=\ev(p_1(\pt))
=\int_{\oS} c_1(\mathcal{I}_n) \cdot \pt$.
\end{lemma}
 
\begin{proof}
Since $\mathcal{I}_n$ is of generic rank one, we have 
$c_1(\mathcal{I}_n)=c_1(\mathcal{U}_n) + c_1(\rho)$. 
We may write 
$$c_1(\mathcal{I}_n)=\ev(p_1(\pt)) \otimes 1 + y
,\quad
y\in H^*(\Hilbrm_n(\oS),\Q) \otimes H^{>0}(\oS,\Q).$$ 
Thus, given any point $s$ in $\oS$, we have 
$$\ev(p_1(\pt))=c_1(\rho) + c_1(\mathcal{U}_n|_{\Hilbrm_n(\oS) \times \{s\}}).$$
However, since $\mathcal{U}_n$ is isomorphic to $\mathcal{O}_{\Hilbrm_n(\oS) \times S}$ outside of a closed subset of codimension $2$, we have
\[
    c_1(\mathcal{U}_n) = c_1(\det \mathcal{U}_n) = c_1(\det \mathcal{O}) = 0,
\]
and so $c_1(\mathcal{U}_n|_{\Hilbrm_n(\oS) \times \{s\}}) = 0$.
\end{proof}

\smallskip

Hence, we have an isomorphism
\begin{equation}\label{E:Hilbident}
H_*(\Hilbrm_n(\oS),\Q) = H_*(\Hilb_n(\oS),\Q)\,/\, 
p_1(\pt) \bullet H_*(\Hilb_n(\oS),\Q).
\end{equation}
We define
\begin{align*}
&V(\oS)=\bigoplus_n H_*(\Hilbrm_n(\oS), \Q),\\
&\mathbf{V}(\oS)=\bigoplus_n H_*(\RHilb_n(\oS), \Q).
\end{align*}
We define $V(S)$ and $\mathbf{V}(S)$ similarly.
Unless $S=\oS$, the cohomology group $H^4(S,\Q)$ vanishes,
hence the class $p_1(\pt)$ vanishes in $\Lambda(S)$. 
Nevertheless, Lemma~\ref{lem:hilbproperties2} and \eqref{E:Hilbident} hold if 
we use instead of $p_1(\pt)$ 
the restriction to $H^*(\Hilb_n,\Q)$ of the class
$p_1(\pt)$ in $H^*(\Hilb_n(\oS),\Q)$.
This is not in conflict with Lemma~\ref{lem:eval-open-restriction}, 
since its conditions are satisfied for $\Hilbrm_n$, but not $\Hilb_n$.

\medskip

\subsection{Purity and generation of the cohomology by tautological classes} In this section we collect some facts on purity and tautological classes on Hilbert schemes and Hilbert stacks.

\smallskip

\begin{lemma}\label{L:pureHilb}The stack $\Hilb_n$ is pure for all $n$.
\end{lemma}

\begin{proof} It is enough to prove that $\Hilbrm_n$ is pure. 
For any partition $\lambda=(1^{m_1}2^{m_2} \ldots)$ we set $$\Sym^\lambda(S)=\Sym^{m_1}(S) \times \Sym^{m_2}(S) \times \cdots.$$
By \cite[Thm.~2]{GS93} there is an isomorphism of mixed Hodge structures 
$$H^{i+2n}(\Hilbrm_n,\Q)\otimes \Q(n)=\bigoplus_\lambda H^{i+2l(\lambda)}(\Sym^\lambda(S),\Q)\otimes \Q(l(\lambda)),$$
where $\lambda$ runs among all partitions of size $n$. 
Hence, it is enough to prove that each symmetric power $\Sym^m(S)$ is cohomologically pure. 
We have an 
isomorphism of Hodge stuctures 
$$\pi^*:H^*(\Sym^m(S),\Q) \stackrel{\sim}{\to} H^*(S^m,\Q)^{\mathfrak{S}_m}.$$ The lemma follows.
\end{proof}

\begin{corollary}\label{C:generationHilb}
For any $n$, the cohomology of $\Hilb_n$ is generated by tautological classes,
i.e., the evaluation map $\Lambda(\oS) \to H^*(\Hilb_n,\Q)$ is onto.
The same holds for the evaluation map $\Lambda(S) \to H^*(\Hilbrm_n,\Q)$.
\end{corollary}

\begin{proof}
The evaluation map factors through $H^*(\Hilb_n(\oS),\Q)$.
The map $\Lambda(\oS) \to H^*(\Hilb_n(\oS),\Q)$ is surjective by~\cite[Thm.~7.5]{Q18} and Lemma~\ref{lem:hilbproperties1}.
The restriction map $H^*(\Hilb_n(\oS),\Q)\to H^*(\Hilb_n,\Q)$ is surjective since the cohomology of $\Hilb_n$ is pure, and $\Hilb_n(\oS)$ is its smooth compactification.
This proves the first claim.
For the second claim, note that the evaluation map factors through $\Lambda(S)$ by Lemma~\ref{lem:eval-open-restriction}.
\end{proof}

By Corollary~\ref{C:generationHilb}, 
the evaluation morphisms 
$\Lambda(S) \to H^*(\Hilbrm_n,\Q)$ 
yield a surjective map
$$\ev':\mathbf{F}(S) \to V(S).$$
The same holds for the surface $\oS$. 
There is a commuting diagram
$$\xymatrix{\mathbf{F}(\oS) \ar[r]^-{\underline{\iota}^*} \ar[d]_-{\ev'} & \mathbf{F}(S) \ar[d]_-{\ev'}\\
{V}(\oS) \ar[r]^-{\underline{\iota}^*} & V(S).}$$

\subsection{COHA actions on $\RHilb$}\label{sec:HilbregularHP} We endow the spaces $\mathbf{V}(\oS),$ $ \mathbf{V}(S)$ with actions of the algebras $\Hb_0(\oS),$ $ \Hb_0(S)$ and their compact versions. This is a reformulation of Nakajima's and Lehn's constructions, see~\cite{NakLectures}, \cite{Lehn}.
Let  $\Hilb_{k-n,k}(\oS)$ be the flag Hilbert stack
\[
    \Hilb_{k-n,k}(\oS) = \{\mathcal{J}\subset\mathcal{I}\subset\mathcal{O}_S
    \,:\, \len(\mathcal{I}) = k-n,\; \len(\mathcal{J}) = k\}
\]

\medskip

\begin{proposition}\label{prop:Hilb-is-regular}
    The stack $\Hilb(\oS)$ is a regular two-sided Hecke pattern. 
    The stack $\Hilb$ is a regular right $S$-strong and left $S$-weak Hecke pattern. 
\end{proposition} 

\begin{proof} The proof of the fact that $\Hilb$ is a two-sided Hecke pattern is the same for $S$ and $\oS$ and relies upon the following simple observation.

\begin{lemma}\label{L:doubledual} Let $0\to \mathcal{F} \stackrel{a}{\to} \mathcal{G} \to \mathcal{T} \to 0$ be a short exact sequence of coherent sheaves on a smooth surface, with $\mathcal{F},\mathcal{G}$ torsion-free and $\mathcal{T}$ zero-dimensional. Then $a^{\vee\vee} :\mathcal{F}^{\vee\vee} \simeq \mathcal{G}^{\vee\vee}$. 
\end{lemma}

\begin{proof} It suffices to consider the case when $\mathcal{T}$ is of length one. Applying the derived duality functor $\mathbb{D}$ yields the long exact sequence
$$\to H^0(\mathbb{D}\mathcal{T}) \to H^0(\mathbb{D}\mathcal{G})\stackrel{\mathbb{D}a}\to H^0(\mathbb{D}\mathcal{F})\to H^1(\mathbb{D}\mathcal{T})\to.$$
Since $\mathbb{D}\mathcal{T} \simeq \mathcal{T}[-2]$, we obtain an isomorphism $\mathcal{G}^\vee=H^0(\mathbb{D}\mathcal{G}) \simeq H^0(\mathbb{D}\mathcal{F})=\mathcal{F}^\vee$. 
Note that because $\mathcal{F},\mathcal{G}$ are torsion-free, $\mathcal{F}^\vee, \mathcal{G}^\vee$ are vector bundles.
\end{proof}

Let $\mathcal{I} \subset \mathcal{O}_S$ be an ideal sheaf of finite colength and let $\mathcal{T}$ be a finite length sheaf.
For any short exact sequence 
$$0 \to \mathcal{J} \to \mathcal{I} \to \mathcal{T}\to 0,$$ 
the sheaf $\mathcal{J}$ is a finite colength ideal sheaf.
If $\mathcal{J} \in \RCoh^{\geq 1}$, for any short exact sequence 
$$0 \to \mathcal{I} \to \mathcal{J} \to \mathcal{T}\to 0$$  
the sheaf $\mathcal{J}$ is torsion-free, because
otherwise it would contain a one-dimensional subsheaf $\mathcal{E}$, whose 
support would intersect $S \backslash \supp(\mathcal{T})$, contradicting the fact that $\mathcal{I}$ is torsion-free.
Finally, by Lemma~\ref{L:doubledual}, there is a canonical isomorphism 
$\mathcal{J}^{\vee \vee} \simeq \mathcal{I}^{\vee \vee}$ hence $\mathcal{J} \in \RHilb$ as wanted. 
The $S$-strongness on the right follows from Lemma~\ref{lem:strong-Hecke-equiv}.

Let us now prove the regularity of these Hecke patterns. Again, the argument is the same for $S$ and $\oS$, 
we will only treat the latter case.
It is well-known that $\Hilbrm_n(\oS)$ and $\Hilbrm_{n,n+1}(\oS)$ are both smooth and connected, of respective 
dimensions $2n$ and $2n+2$. 
Since $\Hilb_n(\oS)$ and $\Hilb_{n,n+1}(\oS)$ are $\mathbb{G}_m$-gerbes over 
$\Hilbrm_n(\oS)$ and $\Hilbrm_{n,n+1}(\oS)$, the former are smooth, irreducible and
\begin{equation}\label{E:dimHilb}
\dim(\Hilb_n(\oS))=2n-1, \qquad \dim(\Hilb_{n+1}(\oS))=\dim(\Hilb_{n,n+1}(\oS))=2n+1.
\end{equation}
Note that $\Hilb_n(\oS)$ is of finite type and embeds in the derived stack
$\RCoh_{1,-n\delta}^{\geq 1}(\oS)$ for any $n$. 
It follows that we may find global resolutions for $\mathcal{E}_{1,-n\delta}$ and for 
$\RHom(\mathcal{E}_\delta, \mathcal{E}_{1,-(n+1)\delta})[1]$. 
Here we follow the notation of \S\ref{sec:derived Hecke}.
The tautological sheaf $\mathcal{E}_{1,-n\delta}$ on $\Hilb_n(\oS) \times \oS$ is of rank one. 
The section $s$ in \eqref{section-s} is thus regular if and only if 
$$\dim(\Hilb_{n,n+1}(\oS))
=\dim(\Hilb_n(\oS) \times \oS) + \rk(\mathcal{E}_{1,-k\delta})-1
=\dim(\Hilb_n(\oS) \times \oS).$$ 
Likewise, the section $s'$ in \eqref{section-s'} is regular if and only if 
\begin{equation*}
\begin{split}
\dim(\Hilb_{n,n+1}(\oS))
&=\dim(\Coh_\delta(\oS) \times \Hilb_{n+1}(\oS)) -\langle \delta, (1-(n+1)\delta)\rangle\\
&=\dim(\Coh_\delta(\oS) \times \Hilb_{n+1}(\oS))-1.
\end{split}
\end{equation*}
Both of these equalities follow from \eqref{E:dimHilb}.
\end{proof}

\smallskip

\begin{remark}
    Let $S$ be projective, and $H$ an ample divisor.
    Suppose that the assumptions A and S of~\cite{Negut} hold. Fix $r>0$ and $c\in H^2(S,\Z)$.
    A proof similar to Proposition~\ref{prop:Hilb-is-regular} yields the regularity of Hecke pattern $\mathcal{M}_{r,c}$, which is the moduli of $H$-stable torsion-free sheaves on $S$ of rank $r$ and first Chern class $c$.
\end{remark}

\medskip

For simplicity, we will denote the Hecke patterns $\Hilb$ and $\Hilb(\oS)$ by $\Hpat$ and $\oHpat$ respectively.
Proposition~\ref{prop:negheckeactionsdef} yields two representations
\begin{align*}
& \Psi^+_{\Hpat} :\widetilde{\Hb}^c_0(S) \to \End(\mathbf{V}(S)),\\ 
& \Psi^-_{\Hpat} :\widetilde{\Hb}_0(S) \to \End(\mathbf{V}(S))
\end{align*}
such that the subspace $H^c_*(\Coh_{k\delta},\Q)$ maps into  
$$\prod_n\text{Hom}\left(H_*(\Hilb_{n},\Q),H_*(\Hilb_{n\mp k},\Q) \right).$$
We get similar representations for $\oS$. 
Since both $\Hilb$ and $\Hilb(\oS)$ are right Hecke patterns, by Remark~\ref{rmk:proper-action-COHA}(a) 
we can lift $\Psi_{\Hpat}^-$ to an action of the algebra $\widetilde{\Hb}_0(S)$.
By Proposition~\ref{P:regularHP}, the
regularity of the Hecke patterns yields representations
\begin{align*}
   & \Phi^{\pm}_{\oHpat}: W^\pm(\oS) \to\End(\mathbf{V}^\taut(\oS))= \End(\mathbf{V}(\oS)),\\ 
   &\Phi^{-}_{\Hpat}: \Wn^-(S) \to\End(\mathbf{V}(S)),
\end{align*}
For $\oS$ they glue to a representation 
$$ \Phi_{\oHpat}:\DW{1}(\oS)\to \End(\mathbf{V}(\oS)).$$
Lemma \ref{lem:hilbproperties1} yields an isomorphism
$$V(\oS) = \mathbf{V}(\oS) \,/\, p_1(\pt) \bullet \mathbf{V}(\oS)$$
Since the class $p_1(\pt)=\psi_0(\pt)$
belongs to the center of $W(\oS)$, 
the representations $\Phi^{\pm}_{\oHpat}$ descend to representations 
\begin{align*}
 \Phi^{\pm}_{\oS}: W^\pm(\oS) \to \End(V(\oS)).
\end{align*}
Since $\ev_{n\delta}(p_1(\pt))=0$ for any integer $n>0$, see Example~\ref{ex:p1pt}, the representations $\Psi^{\pm}_{\oHpat}$ descend to representations 
\begin{align*}
\Psi^\pm_{\oS} :\Hb_0(\oS) \to \End(V(\oS)).
\end{align*}
Similar claims hold for the representations $\Psi_{\Hpat}^\pm$ and $\Phi_{\Hpat}^-$ associated with $S$.

\medskip

\subsection{Nakajima operators and COHA actions} 
Let us briefly recall 
the construction of the Nakajima operators, see~\cite{NakLectures} 
and \cite{Lehn} for details. We'll relate them to the action of 
$\Hb_0(S)$,  $\Hb^c_0(S)$. 
We begin with the case of a proper surface $\oS$.
For each $k\geqslant 0$ and $l \geqslant 1$, 
we consider the reduced subscheme 
$$Z_{k+l,k}(\oS)\subset \Hilbrm_{k+l}(\oS) \times \Hilbrm_{k}(\oS)$$ parametrizing pairs of ideal sheaves $(\cI,\cJ)$ with $\cJ \supset \cI$ for which the support
$\supp(\cJ/\cI)$ consists of a single point. 
There is a support map 
$$s: Z_{k+l,l}(\oS) \to \oS
,\quad
s(\cJ,\cI)=\supp(\cJ/\cI).$$
This allows us to view $Z_{k+l,k}(\oS)$ as a subscheme of 
$\Hilbrm_{k+l}(\oS) \times \oS \times \Hilbrm_{k}(\oS)$. 
For any subset $I \subset \{1,2,3\}$ let $p_I$ be
the projection to the factors in $I$.
For each $\lambda \in H^*(\oS,\Q)$ and $l \geq 1$, the Nakajima operator $\mathfrak{q}_l(\lambda) \in \End(V(\oS))$ is
\begin{equation*}
\begin{split}
\mathfrak{q}_{l}(\lambda) : H_*(\Hilbrm_k(\oS),\Q) &\to H_*(\Hilbrm_{k+l}(\oS),\Q)\\
c &\mapsto p_{1*}(p_{23}^*((\lambda \cap [\oS])\otimes c) \cap [Z_{k+l,k}(\oS)])
\end{split}
\end{equation*}
Note that the restriction of $p_1$ to $Z_{k+l,k}(\oS)$ is proper. 
Exchanging the roles of $\Hilbrm_k(\oS)$ and $\Hilbrm_{k+l}(\oS)$ and using the isomorphic subscheme 
$Z_{k,k+l} (\oS)\subset \Hilbrm_{k}(\oS) \times \oS \times \Hilbrm_{k+l}(\oS)$ in place of $Z_{k+l,k}(\oS)$, 
we get the operator $\mathfrak{q}_{-l}(\lambda) \in \End(V(\oS))$ given by
\begin{equation*}
\begin{split}
\mathfrak{q}_{-l}(\lambda) : H_*(\Hilbrm_{k+l}(\oS),\Q) &\to H_*(\Hilbrm_{k}(\oS),\Q)
\\
c &\mapsto (-1)^l p_{1*}(p_{23}^*((\lambda \cap [\oS])\otimes c) \cap [Z_{k,k+l}(\oS)])
\end{split}
\end{equation*}

By \cite[\S 8]{NakLectures}, the operators $\mathfrak{q}_n(\lambda)$ with $ \lambda \in H^*(\oS,\Q),$ 
$ n \in \mathbb{Z} \backslash\{0\}$ 
generate an action of the Heisenberg algebra $\Heis_\oS$ modeled on $H^*(\oS,\Q)$ with the relations
\eqref{E:Heis-bracket} with central charge $C=1$.
The space $V(\oS)$ is isomorphic to the Fock space representation of $\Heis_\oS$.
The action of $U(\Heis_\oS)$ on $V(\oS)$ is faithful.

Now, assume that $S$ is any quasi-projective smooth surface. 
We may apply the exact same construction, with the following modifications: 
the operators $\mathfrak{q}_{-l}(\lambda)$ are labeled by classes $\lambda \in H^*_c(S,\Q)$
and the commutators are defined using the intersection pairing 
$H^*(S,\Q) \otimes H^*_c(S,\Q) \to \Q$, see \cite[\S 8]{NakLectures}.

\smallskip

\begin{proposition}\label{E:NakOpsasCOHA} Let $S$ be any quasi-projective smooth surface. For any $l >0$ and any $\lambda \in H^*(S,\Q)$, $\mu \in H^*_c(S,\Q)$ there exist elements $E_{l}(\lambda) \in \Hb_0(S)$ and $E_{-l}(\mu) \in \Hb_0^c(S)$ such that 
$$\Psi^-_S(E_{l}(\lambda))=\mathfrak{q}_l(\lambda)
,\quad
\Psi^+_S(E_{-l}(\mu))=\mathfrak{q}_{-l}(\mu).$$
\end{proposition}

\begin{proof} We begin with the operators $\mathfrak{q}_{-l}(\mu)$. Let $h: S \to \Sym^l(S)$ be the diagonal embedding. Set 
$$\RCoh_{l\delta}^{pt} =\RCoh_{l\delta} \times_{\Sym^l(S)} S. $$
Let $s: \RCoh_{l\delta}^{pt} \to S$ the projection and $t:\RCoh_{l\delta}^{pt} \to\RCoh_{l\delta}$ the closed immersion. 
Given $k \geq l$, let $\widetilde{\mathfrak{Z}}_{k,k-l}$ be the derived stack parametrizing inclusions $\cI \subset \cJ$ of ideal sheaves of colength $k$ and $k-l$, such that $\cJ/\cI$ is supported at a single point. 
We have the following commutative diagram with Cartesian squares
$$\xymatrix{\RCoh_{l\delta} \times \RHilb_{k} & \widetilde{\RHilb}_{k,k-l}\ar[l]_-{\kappa} \ar[r]^-{\pi'} & \RHilb_{k-l} \times \Sym^l(S)\\
\RCoh_{l\delta}^{pt} \times \RHilb_{k} \ar[u]^-{t \times \Id}& \widetilde{\mathfrak{Z}}_{k,k-l} \ar[u]^-{i'}\ar[l]_-{\kappa^{pt}} \ar[r]^-{(\pi')^{pt}} & \RHilb_{k-l} \times S \ar[u]^-{i''}
}$$
The map $\kappa^{pt}$ is quasi-smooth, because $\kappa$ is quasi-smooth. 
The class $\mu \in H^c_*(S,\Q)$ yields a class 
\[
    s^*(\mu) \in H^*(\Coh_{l\delta}^{pt}/S,\Q) = H_*(\Coh_{l\delta}^{pt}/\Sym^l(S),\Q);
\]
note that the equality holds by properness of $h$ and~\eqref{eq:hyp-change-of-S}.
We define 
$$E_{-l}(\mu)=(-1)^lt_!(s^*(\mu) \cap [\Coh_{l\delta}^{pt}]) \in H_*(\Coh_{l\delta}/\Sym^l(S),\Q).$$ 
The proper base change in Proposition~\ref{P:basechange} implies that $$(\pi')_!\kappa^!(t \times \Id)_!=(\pi')_!i'_!(\kappa^{pt})^!=i''_!(\pi')^{pt}_!(\kappa^{pt})^!.$$ Composing with the projection $r$, we get the relation $$\Psi^+_{\Hpat}(E_{-l}(\mu))(c)=(-1)^lr \circ (\pi')^{pt}_!(\kappa^{pt})^!((s^*(\mu) \cap [\Coh_{l\delta}^{pt}]) \otimes c)
,\quad
c \in H_*(\RHilb_k,\Q).$$
After pulling everything back from $\RHilb$ to $\Hilbrm$, we have
$$(\kappa^{pt})^!\left((s^*(\mu) \cap [\Coh_{l\delta}]) \otimes c\right)=\left((\mu \cap [S]) \otimes c \otimes [\Hilbrm_{k-l}]\right) \cap [Z_{k,k-l}].$$
Applying the proper pushforward to the projection to the factor $\Hilbrm_{k-l}$,
we get the equality 
$$\Psi^+_{\Hpat}(E_{-l}(\mu))(c)=\mathfrak{q}_{-l}(\mu)(c).$$

We now turn to the case of operators $\mathfrak{q}_l(\lambda)$. Using the cartesian diagram
$$\xymatrix{\RHilb_k(\oS) \times \RCoh_{l\delta}(\oS) & \widetilde{\RHilb}_{k,k+l}(\oS) \ar[r]^-{\overline{\pi}} \ar[l]_-{\overline{\kappa}} & \RHilb_{k+l}(\oS)\\
\RHilb_k \times \RCoh_{l\delta}\ar[u] & \widetilde{\RHilb}_{k,k+l} \ar[u]\ar[r]^-{\overline{\pi}^\circ} \ar[l]_-{\overline{\kappa}^\circ} & \RHilb_{k+l} \ar[u]
}$$
in which all the vertical arrows are open embeddings, we are reduced to the case of a proper surface $S=\oS$. For this we may repeat the arguments used in the case $\mathfrak{q}_{-l}(\mu)$ above.     
\end{proof}

\medskip

\subsection{The Heisenberg subalgebra and Nakajima operators} In this paragraph, we will identify the action of the Heisenberg subalgebra $\Heis_S$ on $V(S)$ with the action of Nakajima operators. 
We begin with the case of a proper surface $\oS$, in which case we may consider both actions $\Phi^{\pm}_{\oS}$.

\medskip

\begin{proposition}\label{prop:heis=nak} For any $\lambda \in H^*(\oS,\Q)$ and $l\geq 1$ we have the following formulas in $\End(V(\oS))$
\begin{equation}\label{E:Nak=Heis}
\Phi_{\oS}^+(D_{l,0}(\lambda))=(-1)^l\mathfrak{q}_{-l}(\lambda)
,\quad
\Phi^-_{\oS}(D_{l,0}(\lambda))=\mathfrak{q}_{l}(\lambda).
\end{equation}
Both representations $\Phi^{\pm}_{\oS}$ of $U(\Heis^+_\oS)$ are faithful.
\end{proposition}
\begin{proof} Assume first that $l=1$. Consider the diagram
$$\xymatrix{\Coh_\delta(\oS) \times \Hilb_{k+1}(\oS) & \widetilde{\Hilb}_{k+1,k}(\oS) \ar[r]^-{p} \ar[l]_-{q} & \Hilb_{k}(\oS)\\
\oS \times \Hilbrm_{k+1}(\oS) \ar[u]^-{j_1} & Z_{k+1,k}(\oS) \ar[u]^-{j_2} \ar[r]^-{p'} \ar[l]_-{q'} & \Hilbrm_{k}(\oS) \ar[u]^-{j_3}
}$$
The vertical arrows are induced by the maps 
$$\oS \to \oS \times B\mathbb{G}_m =\Coh_\delta(\oS)
,\quad
\Hilbrm_k(\oS) \to \Hilbrm_k(\oS) \times B\mathbb{G}_m=\Hilb_k(\oS).$$ 
The right square is cartesian.  
Moreover, the stacks $\widetilde{\Hilb}_{k+1,k}(\oS)$ and $Z_{k+1,k}(\oS)$ 
are smooth, and the maps $q,$ $q'$ are lci. 
For $c \in H_*(\Hilbrm_{k+1}(\oS),\Q)$ and $\lambda \in H^*(\oS,\Q)$, we have 
$$\mathfrak{q}_{-1}(\lambda)(c)=-p'_*((q')^!((\lambda \cap [\oS])\otimes c)).$$ 
Note that the map $H_*(\RHilb_k(\oS),\Q) \to H_*(\Hilbrm_k(\oS),\Q)$ 
is the pullback by the morphism $\Hilbrm_k(\oS) \to \RHilb_k(\oS)$. 
By base change, we have
\begin{equation}
    \begin{split}
\Phi_{\oS}^+(D_{1,0}(\lambda))(c)&=j_3^*(p_*q^!(\lambda \cap [\oS] \otimes c))\\
&=p'_*j_2^*q^!(\lambda \cap [\oS] \otimes c))\\
&=p'_*(q')^!j_1^*(\lambda \cap [\oS] \otimes c))\\
&=-\mathfrak{q}_{-1}(\lambda)(j_1^*(c))
 \end{split}
\end{equation}
as wanted. 
Since the map $\widetilde{\Hilb}_{k+1,k}(\oS) \to \Coh_\delta(\oS) \times \Hilb_k(\oS)$ is lci, 
the refined Gysin pullback is well-defined without any need to consider derived enhancements.

To extend the above relation to arbitrary $l >1$ we use Lehn's formulas~\cite[(2)]{Lehn}.
Put $\mathcal{U}=\pi_*(\mathcal{T})$ where $\mathcal{T}$ is the universal subscheme on $\Hilbrm(\oS) \times \oS$ and $\pi$ is the projection along $\oS$.
Setting $\mathfrak{d}=c_1(\mathcal{U})$, for each
$\lambda \in H^*(\oS,\Q)$, $m \geq 1$ we have
\begin{equation}\label{E:relinLehnforHeis}
    \begin{split}
    [\mathfrak{d},\mathfrak{q}_{\pm 1}(1)]&=\mathfrak{q}'_{\pm 1}(1) \\
    [\mathfrak{q}'_{\pm 1}(1), \mathfrak{q}_{\pm m}(\lambda)]&=-m\mathfrak{q}_{\pm (m+1)}(\lambda).
    \end{split}
\end{equation}
Observing that $\mathfrak{d}=-\frac{1}{2}\psi_2(1)$ and comparing~\eqref{E:relinWforHeis} with~\eqref{E:relinLehnforHeis} we deduce the statement by induction on $m$. 
The difference in signs is due to the fact that $\Phi_{\oS}^-$ is a right representation.
\end{proof}

\smallskip

We now turn to the case of an arbitrary cohomologically pure\footnote{When the surface $S$ is not pure, the same construction yields an identification of the action of the element $D_{m,0}(\alpha)$ for $\alpha \in H^*_{\mathrm{pure}}(S,\Q)$ with the corresponding Nakajima operators $\mathfrak{q}_m(\alpha)$.} surface $S$, where only $\Phi^-_S$ is 
defined. Fixing $\iota: S \to \oS$, the operators $\mathfrak{q}_l(\lambda)$ for $l >0$ and 
$\lambda \in H^*(\oS,\Q)$ 
are easily seen to be compatible with the restriction maps from 
$\underline{\iota}^*:H_*(\Hilbrm_k(\oS),\Q) \to H_*(\Hilbrm_k(S),\Q)$
 in the sense that there is a commutative diagram
$$\xymatrix{H_*(\Hilbrm_k(\oS),\Q) \ar[d]^-{\mathfrak{q}_l(\lambda)} \ar[r]^-{\underline{\iota}^*} & H_*(\Hilbrm_k,\Q) 
\ar[d]^-{\mathfrak{q}_l(\lambda_{|S})}\\
H_*(\Hilbrm_{k+l}(\oS),\Q) \ar[r]^-{\underline{\iota}^*} & H_*(\Hilbrm_{k+l},\Q) 
}.$$
We claim that the same base change formulas hold for the operators $\Phi^-(D_{m,0}(\lambda))$ for $m \geq 1$. Indeed, it follows from the cartesian diagram
$$\xymatrix{\RHilb_k(\oS) \times \RCoh_{\delta} & \widetilde{\Hilb}_{k,k+1}(\oS) \ar[l]_-{\overline{\kappa}} \ar[r]^-{\overline{\pi}'} & \RHilb_{k+1}(\oS) \times \oS\\
\RHilb_k \times \RCoh_{\delta} \ar[u]& \widetilde{\Hilb}_{k,k+1}\ar[u] \ar[l]_-{\overline{\kappa}^\circ} \ar[r]^-{\overline{\pi}^{\circ'}} & \RHilb_{k+1} \times S\ar[u]
}$$
and open base change that there is a commutative diagram
$$\xymatrix{H_*(\Hilbrm_k(\oS),\Q) \ar[d]^-{D_{1,n}(\lambda)} \ar[r]^-{\underline{\iota}^*} & H_*(\Hilbrm_k,\Q) 
\ar[d]^-{D_{1,n}(\lambda_{|S})}\\
H_*(\Hilbrm_{k+1}(\oS),\Q) \ar[r]^-{\underline{\iota}^*} & H_*(\Hilbrm_{k+1},\Q) 
}$$
for any $n$. Since the collection of elements $D_{1,k}(\lambda)$ generates $\Wn^-(\oS)$, and in particular $\Heis^-_\oS$, we deduce the following:

\begin{corollary}\label{Cor:faithfulHeis} For any pure surface $S$ and any pair $(m,\lambda)$ we have $$\Phi_{S}^-(D_{m,0}(\lambda))=\mathfrak{q}_m(\lambda) \in \End({V}(S)).$$
In particular, the representation $\Phi^-_{S}$ of $U(\Heis^-_S)$ is faithful.
\end{corollary}

We are now in position to prove the following: 

\begin{proposition}\label{P:faithfullactionWHilb} For any pure surface $S$, the representation
$\Phi^-_{\Hpat}$ on $\mathbf{V}(S)$ is indeed a faithful representation of $\Wn^-(S)$. 
If $S$ is proper then the same holds for the representation $\Phi^+_{\Hpat}$ on $\mathbf{V}(S)$.
\end{proposition}
\begin{proof} Let $I=\text{Ker}((\Phi^{-}_{S})_{|\Wn^-(S)})$. By definition, it is a two-sided ideal of $\Wn^-(S)$. If non-zero, $I$ must have a non-zero intersection with $U(\Heis_S)$ by Lemma~\ref{lem:Heis-faithful}. But this would contradict Corollary~\ref{Cor:faithfulHeis}.  The second statement is proved in the same fashion.
\end{proof}

\medskip

\subsection{Comparison between COHAs and $W$-algebras}\label{sec:Tmain}

In this section we prove Theorem~\ref{thmB}. We will begin with the case of $\Hb_0(S)$.
Then we deduce the case of $\Hb^c_0(S)$ using the morphism 
$$\underline{\iota}_! : \Hb^c_0(S) \to \Hb_0(\overline{S}).$$

\medskip

\begin{proof}[Proof of Theorem~\ref{thmB} for $\Hb_0(S)$]
    Let $\Hb'_0(S)\subset\Hb(S)$ be the subalgebra generated by $\Hb_0(S)[1,-]$.
    Set $\widetilde{\Hb}'_0(S)=\Lambda(S)\ltimes \Hb'_0(S)$. 
    By \S\ref{sec:HilbregularHP} there are homomorphisms of algebras 
    \begin{equation}\label{diag:heckeactionhilb}
    \xymatrix{ W^{\leq}(S)  \ar[r]^-{\Phi^-} &  \End(\mathbf{V}(S))& \widetilde{\Hb}_0(S)^\op  \ar[l]_-{\Psi^-}}
    \end{equation}
    The map $\Phi^-|_{W^-(S)}$ is injective by Proposition~\ref{P:faithfullactionWHilb}.
    Recall that $\RHilb$ is a Hecke pattern of rank $r=1$.
    For any $n \geq 0$, $m>0$ and $\lambda \in H^*(S,\Q)$ we have
    \begin{equation}\label{E:actionCOHA=Tn}
    \Psi^-(\lambda u^n \cap [\Coh_\delta(S)])=\Phi^-(T_n(\lambda)), \qquad \Psi^-(\psi_m(\lambda))=\Phi^-(\phi_m(\lambda))
    \end{equation}
     By \eqref{E:actionCOHA=Tn}, we have 
     $$\Phi^-(W^-(S)) =\Psi^-(\Hb'_0(S)^\op).$$ 
     Hence, for any $n$ and $l$, we have the chain of inequalities 
    $$\dim(W^+(S)[n,l])  \;\leq\; \dim(\Hb'_0(S)[n,l]) \;\leq\; \dim(\Hb_0(S)[n,l]).$$
    By Theorems~\ref{T:KV} and~\ref{T:W(S)}, we also have 
    $$\dim(W^+(S)[n,l])=\dim(\Hb_0(S)[n,l]).$$ 
    Thus the inequalities above are equalities. 
    Hence, we have $\Hb_0(S)=\Hb'_0(S)$, i.e., the algebra
    $\Hb_0(S)$ is generated in degree one.
    Further, the map $\Psi^-|_{\Hb_0(S)}$ is injective, 
    and the morphism $T_n(\lambda) \mapsto (\lambda u^n) \cap [\Coh_\delta(S)]$ extends to the desired isomorphism of algebras 
    $$\Theta_{S}:W^+(S) \simeq \Hb_0(S).$$
    This isomorphism extends to $W^\geq(S) \simeq \widetilde{\Hb}_0(S)$.
\end{proof}

\medskip

\begin{proof}[Proof of Theorem~\ref{thmB} for $\Hb_0^c(S)$]
    Recall that $S$ is pure.
    Fix a smooth compactification $\iota: S \to \oS$.
    Consider the algebra homomorphism $$\underline{\iota}_!:\Hb_0^c(S) \to \Hb_0(\oS) \simeq W^+(\oS).$$
    By Proposition~\ref{prop:Hilb-is-regular}, the Hecke pattern $\RHilb(\oS)$ is a two-sided $S$-weak Hecke pattern. So we have a left action 
    $$\Psi^+_{\oHpat} : \Hb^c_0(S) \to \End(\mathbf{V}(\oS)).$$ 
By Proposition~\ref{prop:basechangeHP}, there is a commutative diagram
    $$\xymatrix{\Hb^c_0(S) \ar[r]^-{\underline{\iota}_!} \ar[dr]_(.4){\Psi^+_{\oHpat}}&\Hb^c_0(\oS) \ar[d]^(.35){\Psi^+_{\oHpat}} & W^+(\oS) \ar[l]_-{\Theta_{\oS}} \ar[dl]^(.4){\Phi^+_{\oHpat}}\\ & \End(\mathbf{V}(\oS)) &
    }$$
Put $A=\Psi^+_{\oHpat}(\Hb_0^c(S))$. By Proposition~\ref{E:NakOpsasCOHA}, 
$A$ contains the subalgebra generated by the operators 
$$\mathfrak{q}_{-l}(\lambda)
,\quad
l>0
,\quad
\lambda \in H^*_c(S,\Q)
$$ 
as well as the operators 
$$\Phi^+_{\oHpat}(D_{1,m}(\lambda))
,\quad
m \geq 0
,\quad
\lambda \in H^*_c(S,\Q)$$ which arise as length one compactly supported Hecke operators.
    We know that $\Theta_{\oS}$ is an isomorphism.
    By Proposition~\ref{P:faithfullactionWHilb}, the map $\Phi^+_{\oHpat}$ is injective.
Thus, the definition of $\Wc^{+}(S)$ yields
    $$\Phi^+_{\oHpat}(\Wc^{+}(S))\subset A.$$ 
Hence, the graded dimension of $A$ is bounded below by that of $\Wc^{+}(S)$. 
By Theorem~\ref{T:KVcompact} there is an equality between these graded dimensions.
We deduce that we have a chain of isomorphisms
    $$\xymatrix{\Hb^c_0(S) \ar[r]^-{\Psi^+_{\oHpat}}& A \ar@{=}[r] & \Phi^+_{\oHpat}(\Wc^{+}(S))& \Wc^{+}(S) \ar[l]_-{\Phi^+_{\oHpat}}}$$
    as desired.
    This also shows that the map $\underline{\iota}_!$ is injective and concludes the proof of Theorem~\ref{thmB}.   
\end{proof}

\begin{remark}
    The proof of Theorem~\ref{thmB} above goes through almost verbatim in the case when $S$ is equipped with an action of an algebraic torus $T$.
    The only thing that we need in addition are equivariant counterparts of Theorems~\ref{T:KV},~\ref{T:KVcompact}, which are provided by remarks after said theorems.
\end{remark}

\smallskip

\begin{corollary}\label{cor:regularHPcompact} Let $S$ be a pure surface, $X$ a two-sided Hecke pattern of rank $r$, and $X^\circ\subset X$ an open two-sided $S$-weak Hecke subpattern.
There is a commutative diagram 
$$\xymatrix{\Wc^\pm(S) \ar[r] \ar@/^1pc/[rr]^-{\Phi^{\pm}} \ar[dr]_-{\Phi^{\pm}_X} & \End_X(\mathbf{F}^{(r)}(\oS)) \ar[r] \ar[d]^-{\ev} & \End_{X^\circ}(\mathbf{F}^{(r)}(\oS))\\
& \End(\mathbf{V}^\taut(X))&}$$
which glues into an action of $\Wcc{r}(S)$ on $\mathbf{V}^\taut(X)$ provided that the conditions of Lemma~\ref{lem:eval-open-restriction} are satisfied. 
\end{corollary}

\begin{proof}
    We have an analogous diagram for $W^\pm(\oS)$ by Proposition~\ref{P:regularHP}.
    For the first claim, it suffices to show that $\Wc^\pm(S)$ lands in $\End_{X^\circ}(\mathbf{F}^{(r)}(\oS))$ under $\Phi^\pm$.
    This follows from Proposition~\ref{prop:negheckeactionsdef} and the fact that $\Wc^+(S)\simeq \Hb_0^c(S)$ as subalgebras of $W^+(\oS)$.
    Adding tautological classes in, à priori we get actions of $W^0(\oS) \ltimes \Wc^\pm(S)$.
    However, the map $\mathbf{F}^{(r)}(\oS)\to \mathbf{V}^\taut(X)$ factors through $\mathbf{F}^{(r)}(S)$ by Lemma~\ref{lem:eval-open-restriction}, and so the two actions above glue to $\Wcc{r}(S)$.
\end{proof}

In particular, we have an action of $\Wnc{1}(S)$ on $V(S)$, which descends to a faithful action of $W_{\uparrow\mathrel{\mspace{-2mu}}\downarrow,\red}^{(1)}$ by Lemma~\ref{lem:Heis-faithful}.

\begin{remark}
    By Lemma~\ref{lem:strong-Hecke-equiv}, any two-sided $S$-weak Hecke pattern which is $S$-strong on the left is automatically $S$-strong on the right. In particular, the action of $\Wcn{r}(S)$ always has central charge $r=0$ and lifts to the action of $\Wnn{0}(S)$.
    However, we expect that one can obtain actions of $\Wcn{r}(S)$ with non-zero central charge out of Hecke patterns living in other hearts of $D^b Coh(S)$.
\end{remark}

\medskip

\section{Action on Higgs bundles}\label{sec:Higgs}

In this section we show how to apply the machinery developed in this paper to obtain explicit formulas for the action of Hecke operators on the homology of the stack of Higgs bundles on a smooth projective curve $C$.

\subsection{The stack of Higgs bundles}\label{S:Higgsbundles} Let us consider $S=T^*C$, where $C$ is a smooth connected projective curve of genus $g$. Let $\overline{S}=\mathbb{P}(\Omega_C\oplus \mathcal{O}_C)$ be the projective completion of $\Omega_C$, which is a smooth compactification of $S$, and let $p: \overline{S} \to C$ be the projection.
We consider the (derived) stack $\mathfrak{Higgs}_{r,d}$ classifying Higgs sheaves on $C$ of rank $r$ and degree $d$, or equivalently via the BNR correspondence, coherent sheaves $\mathcal{E}$ on $\overline{S}$ whose support does not intersect $D_\infty= \overline{S}\backslash S$ and for which $p_*(\mathcal{E}) \in Coh(C)$ is of rank $r$ and degree $d$. When $r=0$ we recover the stacks $\RCoh_{d\delta}(S)$. We will sometimes view a Higgs sheaf as a pair $(\mathcal{F},\theta)$ where $\mathcal{F} \in Coh(C)$ and $\theta \in \Hom(\mathcal{F},\mathcal{F} \otimes \Omega_C)$. The correspondence is given by $\mathcal{E} \mapsto \mathcal{F}:=p_*(\mathcal{E})$ and reads as follows on the Chern classes:
$$\mathrm{ch}_0(\mathcal{E})=0, \qquad c_1(\mathcal{E})=r[C],\qquad c_2(\mathcal{E})=r(r+1)(1-g)-d.$$ 
Note that a Higgs sheaf $\mathcal{E}$ is of dimension $\geq 1$ is and only if it is pure of dimension $1$ if and only if the associated sheaf $\mathcal{F}$ on $C$ is a vector bundle.
We will denote by $\mathcal{H}iggs_{r,d}$ the classical truncation of $\mathfrak{Higgs}_{r,d}$.

\smallskip

As opposed to the case of the moduli stack of coherent sheaves on a curve, the stack $\mathfrak{Higgs}_{r,d}$ is not irreducible. Luckily, as soon as $g>1$, the stack $\mathfrak{Higgs}^\circ_{r,d}$ parametrizing Higgs \textit{bundles} of rank $r$ and degree $d$ is irreducible. More precisely, denote by $\mathfrak{Higgs}_{r,d}^{tor=l}$ the locally closed substack of $\mathfrak{Higgs}_{r,d}$ parametrizing Higgs sheaves whose maximal zero-dimensional subsheaf is of degree $l$. 

\medskip

\begin{proposition} Assume that $g >1$. Then the Zariski closures of $\mathfrak{Higgs}_{r,d}^{tor=l}$ for $l \geq 0$ form a complete set of irreducible components of $\mathfrak{Higgs}_{r,d}$. Moreover  the stack ${\mathcal{H}iggs_{r,d}^{tor=l}}$ is of dimension $2r^2(g-1)+l+1$. 
In particular, we have
$\dim \mathcal{H}iggs_{r,d}^\circ=2r^2(g-1) +1.$
\end{proposition}

\begin{proof}
We begin by recalling some standard facts about Harder-Narasimhan filtrations. Let us fix $l>0$. Projecting a Higgs sheaf to its vector bundle quotient and torsion subsheaf yields a morphism $\Higgs_{r,d}^{tor=l} \to \Higgs_{r,d-l}^\circ \times \Higgs_{0,l}$ which is a stack vector bundle of rank $0$; indeed, the fiber over a pair of Higgs sheaves $(\mathcal{V},\mathcal{T})$ is equal to $\mathbf{R}\Hom(\mathcal{V},\mathcal{T})[1]$ which is of perfect amplitude $[-1,0]$ and of virtual rank $\langle [\mathcal{V}],l\delta\rangle=0$ by~\eqref{E:eulerdeltaO}. More generally, for any Harder-Narasimhan strata $HN_{\alpha_1, \ldots, \alpha_s}$ with $\sum_i \alpha_i=(r,d)$ and $s>1$, the morphism
$$p_{\alpha_1, \ldots, \alpha_s}: HN_{\alpha_1, \ldots, \alpha_s} \to \prod_i \mathcal{H}iggs^{ss}_{\alpha_i}$$
is an iterated stack vector bundle of rank $\sum_{i<j} \langle \alpha_i, \alpha_j\rangle=\sum_{i<j}2(g-1)r_ir_j$, we have
\begin{equation}\label{eq:proofHiggs1}
\dim(HN_{\alpha_1, \ldots, \alpha_s})=\sum_i \left(2(g-1)r_i^2+1\right) + \sum_{i<j} 2(g-1)r_ir_j,
\end{equation} 
where $\alpha_i=(r_i,d_i)$.

We now turn to the proof of the Proposition.
Let us first show that the stack $\Higgs_{r,d}^{\circ}$ is irreducible\footnote{This can be deduced from \cite[Section~2]{BDHitchin} but we provide an argument for the sake of completeness.}. Let $\mathfrak{Bun}_{r,d}$ be the stack classifying vector bundles on $C$ of rank $r$ and degree $d$. It is a smooth, irreducible stack of dimension $(g-1)r^2$, with a non-empty open substack $\mathfrak{Bun}_{r,d}^{st}$ parametrizing stable vector bundles. The canonical morphism $\pi: \Higgs_{r,d}^{\circ} \to \mathfrak{Bun}_{r,d}$ identifies $\Higgs^{\circ}_{r,d}$ with the cotangent stack of $\mathfrak{Bun}_{r,d}$. It follows that any irreducible component of $ \mathcal{H}iggs_{r,d}^\circ$ is of dimension at least $2r^2(g-1) +1$ (one can see this, for instance, by locally realizing $ \mathcal{H}iggs_{r,d}^\circ$ as a symplectic quotient). On the other hand, the morphism $\pi$ is representable and $\pi^{-1}(\mathcal{F}) \simeq \Ext^1(\mathcal{F},\mathcal{F})^*$. By Serre duality, $\dim\; \pi^{-1}(\mathcal{F})=(g-1)r^2+\dim\;\End(\mathcal{F})$. In particular, $\mathcal{H}iggs_{r,d}^{st}$ is irreducible, of dimension $2(g-1)r^2 +1$. It is known that the open substack of stable Higgs bundles is dense in the stack of semistable Higgs bundles (see e.g. \cite[Prop. 5.3, 5.4]{DHSM2} and use the fact that when $g>1$, $\mathcal{H}iggs^{st}_{r,d}$ is nonempty).  To conclude, it is enough to show that the dimension of any Harder-Narasimhan strata $HN_{\alpha_1, \ldots, \alpha_s}$ of $\mathcal{H}iggs^\circ_{r,d}$ is of dimension strictly less than $2(g-1)r^2+1$. This now follows easily from induction and  formula \eqref{eq:proofHiggs1}. Note that for any strata $HN_{\alpha_1, \ldots, \alpha_s}$ of $\mathcal{H}iggs^\circ_{r,d}$ we have $r_i>0$ for any $i$.

Next, for any $l>0$, the same argument shows that $\mathcal{H}iggs_{r,d}^{tor=l}$ is irreducible, of dimension $2(g-1)r^2+l+1$. 
Finally, it is easy to see that the union $\Higgs_{r,d}^{tor \leq l}=\bigsqcup_{n \leq l} \Higgs_{r,d}^{tor=n}$ is an open substack for any $n$. Since the dimensions of the irreducible locally closed substacks $\Higgs_{r,d}^{tor=l}$ increase with $l$, we deduce that the Zariski closures of each $\Higgs_{r,d}^{tor=l}$ is an irreducible component of $\Higgs_{r,d}$. 
\end{proof}

\medskip

\begin{remark}
\leavevmode\nolisttopbreak
\begin{enumerate}[label=$\mathrm{(\alph*)}$,leftmargin=8mm,itemsep=1.2mm]
\item
For $g=0$ or $1$ the situation is quite different. When $g=0$, the stack $\mathfrak{Higgs}^\circ_{r,d}$ has infinitely many irreducible components whose classical truncations are all of dimension $-r^2$. In that case $\mathfrak{Higgs}_{r,d}^\circ$ coincides with the global nilpotent cone. When $g=1$, the classical stack $\mathcal{H}iggs_{r,d}^\circ$ is also not irreducible, but the dimensions of the irreducible components may vary between $1$ and $r$. Similar results hold for $\mathfrak{Higgs}^{tor=l}_{r,d}$ for any positive $l$.
\item
For any $g$ and $l$, the stack $\mathfrak{Higgs}_{r,d}^{tor=l}$ is of (virtual) dimension $2(g-1)r^2$. In particular, this dimension is independent of $l$.
\end{enumerate}
\end{remark}

\medskip

\subsection{Regularity of the Hecke pattern}\label{sec:conditionHiggs} For any $r \geq 1$, let us put $\Higgs_r^\circ=\bigsqcup_{d} \Higgs_{r,d}^\circ$. The aim of this paragraph is to prove the following result:

\smallskip

\begin{proposition}\label{T:HiggsregularHP} Assume that $g>1$. Then the substack $\Higgs_{r}^\circ$ is a regular $S$-strong two-sided Hecke pattern on $S=T^*C$.
\end{proposition}
\begin{proof} The fact that $\Higgs^\circ_r$ is an $S$-strong two-sided Hecke pattern follows from Example~\ref{ex:Hecke-patterns}. Since $C$ is of genus $g >1$, $\mathcal{H}iggs_{r,d}^\circ$ is irreducible and of 
dimension $2(g-1)r^2 + 1$ for any $d$. Fix $d \in \Z$ and set $\alpha=(r,d), $ $\gamma=\alpha+\delta=(r,d+1)$. 
Let $\mathcal{E}_{\eta}$ be the tautological sheaf on $\Higgs_{\eta} \times S$. Let $\mathfrak{U}$ be any finite 
type open substack of $\Higgs^\circ_{\gamma} \times S$. Let $\mathcal{E}_{-1} \to \mathcal{E}_0$ be a presentation 
of $\mathcal{E}_{\gamma}$ as a perfect complex. Let $s$ be the associated section. The virtual rank of 
$\mathcal{E}_{\gamma}$ being zero, the virtual dimension of the map 
$\pi_{\delta,\alpha}: \widetilde{\Higgs}_{\delta;\alpha} \to \Higgs_\gamma$ is $1$. 
It follows that 
$$\dim(\widetilde{\mathcal{H}iggs}^\circ_{\delta;\alpha}) \geq \dim(\mathcal{H}iggs^\circ_{\gamma}) +1.$$
We will show that $s$ is regular by proving that 
$$\dim(\widetilde{\mathcal{H}iggs}^\circ_{\delta;\alpha}) \leq \dim(\mathcal{H}iggs^\circ_{\gamma}) +1,$$ 
which will in fact imply equality. 
This will also prove the regularity of the section $s'$, see (\ref{E:assumption}). 
For this, let 
$$\mu: \mathcal{H}iggs^\circ_\gamma \to \mathcal{A}
,\quad
\mathcal{A}=\bigoplus_{i=1}^r H^0(C,\Omega_C^{\otimes i})$$
be the Hitchin map. If $a \in \mathcal{A}$ we write $C_a$ for the corresponding spectral curve.
Let 
$\mathcal{C} \subset \mathcal{A} \times S$ be the universal spectral curve. For $i \geq 1$, we set 
$$\mathcal{R}_{i}=\{(a,(x,\xi)) \in \mathcal{C}\,:\, x\in C,\, \xi\in T^*_x C,\, C_a \xrightarrow{p} C\;\text{is ramified of order}\;i\;\mathrm{at}\;(x,\xi)\}. $$
Let $\mathcal{A}_{i}^{(x,\xi)} \subset \mathcal{A}$ be the fiber over $(x,\xi)$ of the projection $\mathcal{R}_{i} \to S$. 
Note that $\mathcal{R}_{\geq i}=\bigsqcup_{j \geq i} \mathcal{R}_j$ is closed and $\mathcal{R}_{\geq 1}=\mathcal{C}$. We set $\overline{\mathcal{R}}_i=(\mu \times \Id)^{-1}(\mathcal{R}_i)$.

\smallskip

\begin{lemma}\label{L:dimest} For any $i \geq 1$ and $(x,\xi) \in S$ we have $\codim_{\mathcal{A}} (\mathcal{A}_{i}^{(x,\xi)}) = i$.
\end{lemma}
\begin{proof} Since $\Omega_C$ is base-point free, the evaluation at $x \in C$ yields a surjective linear map $i_x^* :\mathcal{A} \to \bigoplus_{i=1}^r T^*_x(C)^{\otimes i}$. The locally closed subset $\mathcal{A}_i^{(x,\xi)}$ is the inverse image of the subset of degree $r-1$ polynomials in one variable $y \in T^*_xC$ vanishing with order $i$ at $\xi$. This condition defines a subset of codimension $i$.
\end{proof}

Since $g>1$ the map $\mu$ is flat by~\cite[Cor. 9]{Ginzburg}. From this we deduce that $\codim_{\mathcal{H}iggs^\circ_\gamma}(\mu^{-1}(\mathcal{A}_{i}^{(x,\xi)})) = i$ hence $\codim_{\mathcal{H}iggs^\circ_\gamma \times S}(\overline{\mathcal{R}}_i) = i$.
Recall the morphism 
$$\pi: \widetilde{\mathcal{H}iggs}^\circ_{\delta,\alpha} \to \mathcal{H}iggs^\circ_\gamma \times S
,\quad
\mathcal{H} \subset \mathcal{E}\mapsto (\mathcal{G}, \supp(\mathcal{E}/\mathcal{H})).$$ 
The map $\pi$ lands in the closed substack $\overline{\mathcal{R}}_{\geq 1}$. For any Higgs bundle $\mathcal{F}$ whose support is ramified at $(x,\xi)$ of order $i$ we have $\dim(\text{Hom}_{\mathcal{O}_S}(\mathcal{F},\mathcal{O}_{(x,\xi)})) \leq i$. It follows that $\dim(\pi^{-1}(\mathcal{E},(x,\xi))) \leq i-1$ if $(\mathcal{E},(x,\xi))) \in \overline{\mathcal{R}}_i$. Hence, we have
$$\dim(\pi^{-1}(\overline{\mathcal{R}}_{i})) \leq \dim(\overline{\mathcal{R}}_{i}) + i-1 = \dim(\mathcal{H}iggs^\circ_\gamma \times S)-i+i-1=\dim(\mathcal{H}iggs^\circ_\gamma)+1.$$
Since this is true for all $i=1, \ldots, r$, we get the desired dimension estimate.
\end{proof}

\begin{remark}\leavevmode\nolisttopbreak
  \begin{enumerate}[label=$\mathrm{(\alph*)}$,leftmargin=8mm,itemsep=1.2mm]
        \item Refining the above dimension estimates, one can show that for $g>1$ the Hecke correspondence $\widetilde{\mathcal{H}iggs}^\circ_{\delta,\alpha}$ is irreducible.
        \item When $g =0,1$, $\Higgs_r^\circ$ is \textit{not} a regular Hecke pattern. 
        For instance, if $gcd(r,d)=1$ and $r>1$ then a generic stable Higgs sheaf 
        $\mathcal{E}=(\mathcal{F},\theta)$ has a scalar Higgs field. Hence the fiber of $p_{\delta,\alpha}$ over $\mathcal{E}$ is of dimension $r-1$. Thus $\widetilde{\mathcal{H}iggs}^\circ_{\delta,\alpha}$ has a component of dimension $r+\dim(\mathcal{H}iggs^{st}_{r,d})$ lying over $\mathcal{H}iggs^{st}_{r,d}$.
        \item A similar argument also shows that the stacks of $\Higgs^{\mathcal{L},\circ}$ of $\mathcal{L}$-twisted Higgs sheaves with $\deg(\mathcal{L}) > 2(g-1)$ form a regular two-sided Hecke pattern on $\Tot(\mathcal{L})$.
    \end{enumerate}
\end{remark}

\subsection{Action on tautological classes}
Unlike the case of the stack of coherent sheaves on $C$~\cite{Heinloth}, the homology of the stack of Higgs sheaves on $C$ is not generated by tautological classes. 
Note, however that by Markman's Theorem \cite{Markman} this is the case after restriction to the stack of stable 
Higgs sheaves in the case $(r,d)=1$. 
This motivates the study of Hecke operators on the subspace of tautological classes of $H_*(\Higgs^\circ,\Q)$. 
By Propostion~\ref{P:regularHP} there is an action of $\Wnn{0}(T^*C)$ on $\mathbf{V}^\taut(\Higgs^\circ)$.
Moreover since $H^*(S,\Q) = H^*(C,\Q)$, by degree reasons we have $c_1^2=c_2=0$ and $c_1\Delta_S=0$.
Analogously to Corollary~\ref{cor:nondef-bbW} we conclude that $\Wnn{0}(T^*C) \simeq U(\mathfrak{w}_{T^*C})$ where 
$$\mathfrak{w}_{T^*C}=\bigoplus_{\substack{m,n \geq 0\\ (m,n) \neq 0}}\bigoplus_{\gamma \in \Pi} \Q D_{m,n}(\gamma), \qquad \Pi=\{1, \gamma_1, \ldots, \gamma_{2g}, \omega\}$$
with $\gamma_1, \ldots, \gamma_{2g}$ a symplectic basis of $H^1(C,\Q)$, with relations
$$[D_{m,n}(\gamma), D_{m',n'}(\gamma')]=(mn'-m'n)D_{m+m',n+n'-1}(\gamma\gamma')$$
for all tuples $(m,m',n,n',\gamma,\gamma')$.

We should stress that Proposition~\ref{cor:big-W-faithful} does not apply in this case, because Higgs bundles have zero rank as sheaves on $T^*C$.
This is also evidenced by the fact that, contrary to the case of Hilbert scheme, the action of both $W^0$ and $\Wn^+$ on $\mathbf{V}^\taut(\Higgs^\circ)$ is not faithful.
However, this non-faithfulness is a feature; it is related to the existence of a certain ``rational'' degeneration of $\Wn^\geq(T^*C)$, which was used in~\cite{HMMS} to prove the $P=W$ conjecture of de Cataldo-Hausel-Migliorini.
See also \S\ref{ssec:K3} below.

\begin{example}\label{ex:Higgs-degen}
    Consider Higgs bundles of rank $1$ on a curve $C$ of genus $g>1$.
    In this case 
    $$\Higgs^\circ_{1,d}\simeq \mathrm{Jac}^d C\times B\mathbb{G}_m\times H^0(C,\Omega_C).$$
    We have the Poincar\'e line bundle $\mathcal{P}$ on $(\mathrm{Jac}^dC\times B\mathbb{G}_m)\times C$, and it is known that 
    \[
        c_1(\mathcal{P}) = u\otimes 1 + \sum_\gamma b_\gamma\otimes \overline{\gamma} + d\otimes\pt,
    \]
    where $u$ is the generator of $H^*(B\mathbb{G}_m,\Q)$, and $b_\gamma$ form the basis of $H^1(C,\Q)\subset \Lambda^\bullet H^1(C,\Q)\simeq H^*(\mathrm{Jac}_d,\Q)$.
    Starting from this observation, an easy but tedious computation shows that
    \[
        \psi_1(\pt) = u,\quad \psi_1(\gamma) = b_\gamma, \quad \psi_0(\pt) = d,
    \]
    where the classes $\psi_i$ are obtained from the universal sheaf on a natural compactification of $\Higgs^\circ_{1,d}\times T^*C$ in accordance with \S\ref{sec:taut}.
    This means that the cohomology of $\Higgs^\circ_{1,d}$ is generated by the classes $\psi_1$.
    In particular, $\psi_2(1)$ can be expressed in terms of these classes for each $d$.
    Moreover, since the tautological sheaf for different $d$ differs by tensoring with a power of $\mathcal{O}(c)$, $c\in C$, the dependence on $d$ can be seen to be polynomial in $d$.
    As $\psi_0(\pt) = d$, we conclude that $\psi_2(1)$ is expressed in $\psi_1$'s and $\psi_0(\pt)$ independently of $d$, so that the action of $W^0(T^*C)$ on $\mathbf{V}^\taut(\Higgs^\circ_1)$ is not faithful.
    Applying $\Ad_{D_{1,0}(1)}$ twice to such an expression, we see that $\mathfrak{q}_2(1)$ expresses in terms of $\mathfrak{q}_1$'s (and $d$), so the action of $\langle\Wn^+(T^*C), \psi_0(\pt)\rangle$ is not faithful either.
\end{example}

\medskip

\subsection{One-dimensional sheaves on K3 surfaces}\label{ssec:K3}
Let us now briefly consider another example.
Let $S$ be a K3 surface, and fix a smooth curve $C\subset S$ of genus $g>1$ which is a very ample divisor.
We denote by $\Mukai_r$ the moduli stack of coherent sheaves $\mathcal{E}$ of pure 1-dimensional sheaves on $S$, such that $c_1(\mathcal{E})=r[C]$.
It is well known~\cite{Mukai,DEL} that $\Mukai_r$ behaves similarly to $\Higgs^\circ_r$. In particular, it admits a morphism $\Mukai_r\to \mathbb{P}^{r^2(g-1)+1}$ with properties analogous to the Hitchin map $\Higgs_r^\circ\to \mathcal{A}$. 

It follows again from Example~\ref{ex:Hecke-patterns} that $\Mukai_r$ is a two-sided Hecke pattern.
While we expect that $\Mukai_r$ is in fact regular, let us for simplicity's sake restrict our attention to the subspace $\mathbf{V}^{\mathrm{vtaut}}_{\Mukai_r}\subset H_*(\Mukai_r,\Q)$ of classes obtained by capping the virtual fundamental class $[\Mukai_r]_\mathrm{vir}$ with tautological classes.
Invoking Proposition~\ref{prop:Negutlemmavirtual}, we can apply an analogue of Proposition~\ref{P:regularHP} to obtain an action of $\DW{0}(S)$ on $\mathbf{V}^{\mathrm{vtaut}}_{\Mukai_r}$.
While $c_1=0$, it is known that $s_2 = -24\pt$.
Thus we are placed in the semi-deformed situation of \S\ref{sec:semidef}, where we can take $q=(-\int_S C^2/24)^{-1/2}C$ up to passing to a finite extension of $\Q$ in the coefficients.

Similarly to Example~\ref{ex:Higgs-degen}, it is easy to check that the action of $\DW{0}(S)$, or indeed of $W^+(S)$ on $\mathbf{V}^{\mathrm{vtaut}}_{\Mukai_r}$ is not faithful.
However, with an argument analogous to the one found in~\cite[Section 6]{HMMS}, one can show that the action of $W^{\geq}(S)$ degenerates to the action of the algebra $U(\mathfrak{w}_S^{\mathrm{log}})$, where $\mathfrak{w}_S^{\mathrm{log}}$ has basis $x^m\partial^n\lambda$, $m,n\in\N$, $\lambda\in H^*(S,\Q)$, and the Lie bracket is given by
\[
    [x^m\partial^a\lambda,x^n\partial^b\mu] = \sum_{i\geq 1} i!\left( \binom{a}{i}\binom{n}{i}-\binom{b}{i}\binom{m}{i} \right)x^{m+n-i}\partial^{a+b-i}q^{i-1}\lambda\mu.
\]
One way to think of this is that $\mathfrak{w}_S$ looks like the Lie algebra of differential operators on $\mathbb{C}^*$ with coefficients in $H^*(S,\Q)$, and so rational degeneration should look like the Lie algebra of differential operators on $\mathbb{C}$.

Note that the defining relations of $\mathfrak{w}_S^{\mathrm{log}}$ imply that $\{ x^2/2, x\partial+2q, \partial^2/2 \}$ is an $\mathfrak{sl}_2$-triple, which should control the perverse filtration on the cohomology of the stable locus of $\Mukai_r$. We plan to return to this in the future work.

\medskip

\section{Some conjectures}\label{sec:conjectures}
\subsection{Action beyond tautological classes}
Let us informally summarize what we proved, omitting most adjectives.
First, given any two-sided Hecke pattern $X$, we have two actions of $W^+(S)$ on the Borel-Moore homology of $X$.
Second, for a regular Hecke pattern these two actions glue to an action of $\DW{\c}(S)$ \emph{on the subspace of tautological classes}.
Somewhat frustratingly, our methods do not prove the relation~\eqref{eq:comm-pm} for non-tautological classes.
This reflects the fact that $\DW{\c}(S)$ is supposed to be related to a Drinfeld double of $\mathbf{H}_0(S)$, which has not yet been defined.
In view of Proposition~\ref{prop:Negutlemmavirtual}, it also seems natural to drop the regularity condition.

\begin{conjecture}
For any two-sided Hecke pattern $X$ of rank $r$, there is an action of $W^{(r)}_{\star,\star}(S)$ on $H_*(X)$ for appropriate $\star\in\{\uparrow,\downarrow\}$.
\end{conjecture}

Our main theorems were proved under the assumption that $S$ has pure cohomology, but we conjecture that they hold without this assumption.
The reason for this assumption is that we need $\mathbf{V}^\taut(\RHilb)$ to be a faithful representation of $U(\Heis_{\oS})$ in order to compare the deformed $W$-algebra with the COHA as subalgebras of $\End(\mathbf{V}^\taut(\RHilb))$.
We plan to address this problem in future work by considering another family of tautological classes on $\Hilbrm(S)$, obtained from the correspondence
\[
    \begin{tikzcd}
        S & Q_n\ar[l]\ar[r]& \Hilbrm_n(S)
    \end{tikzcd}
\] 
where $Q_n$ is the universal subscheme.

\subsection{$W$-algebras for 3-dimensional varieties}
Let $S$ be a smooth surface. The Borel-Moore homology of the stack $\Coh^{0}(S)$ is isomorphic to the critical cohomology $H_{\mathrm{crit}}^*(\Coh^0(M))$ of the stack of finite length sheaves $\Coh^0(M)$ on the Calabi-Yau \emph{threefold} $M = \mathrm{Tot}_S (K_S)$ by dimensional reduction~\cite{Kinjo}.
For any 3-Calabi-Yau $M$, the space $\mathbf{H}^{\mathrm{crit}}_0(M) = H_{\mathrm{crit}}^*(\Coh^0(M))$ is an associative algebra thanks to the recent work of Kinjo, Park and Safronov~\cite{KPS}.
These algebras, called critical COHAs, were first introduced and studied for 3CY dg-algebras in~\cite{KS}.
We expect that an isomorphism similar to the one of Theorem~\ref{thmB} should hold for $\mathbf{H}^{\mathrm{crit}}_0(M)$ as well.

\begin{conjecture}
Let $W^+(M)$ be an algebra generated by $T_n(\lambda)$, $n\geq 0$, $\lambda\in H^*(M)$, modulo the relations \eqref{W:b}, \eqref{W:e}-\eqref{W:g} of Definition~\ref{def:W}, where we replace $s_2$ by $c_2(TM)$, and $c_1\Delta_S$ by $\Delta_M$.
Then $\mathbf{H}^{\mathrm{crit}}_0(M)\simeq W^+(M)$. 
\end{conjecture}

One may wonder if a similar statement can be made without Calabi-Yau condition.
Our preliminary computations suggest that the analogues of relations \eqref{W:f}, \eqref{W:g} become significantly more complicated.
We hope to return to this in future work.

\section*{Acknowledgements}

It is a great pleasure to thank B.~Davison, J.~Heinloth, A.~Khan, A.~Negut, M.~Porta and F.~Sala for very useful discussions and correspondences, as well as the anonymous referee for a thorough reading and feedback. We are especially grateful to B.~Hennion for his patient explanations of derived algebraic geometry and to L.~Hennecart for pointing out some simplification to an argument in a previous version, using the results of \cite{DHSM}. The intellectual debt that this paper owes to~\cite{Negut} will be obvious to the expert reader. Anton Mellit was supported by the ERC consolidator grant ``Refined invariants in combinatorics, low-dimensional topology and geometry of moduli spaces'' No. 101001159. Alexandre Minets was in part supported by the ERC starter grant ``Categorified Donaldson-Thomas Theory'' No. 759967, and also likes to thank MPIM Bonn for the excellent working conditions.
Olivier Schiffmann was in part supported by the PNRR grant CF 44/14.11.2022 \textit{Cohomological Hall algebras of smooth surfaces and applications}.

\medskip

\appendix

\section{Borel-Moore homology for derived stacks}
\label{sec:appA}
A derived stack $X$ is a functor $R\mapsto X(R)$ assigning an $\infty$-groupoid
to every simplicial commutative ring, that satisfies \'etale hyperdescent.
We say that $X$ is 1-Artin if its diagonal is representable by an algebraic space
and if there is a scheme $Y$ with a smooth and surjective morphism $Y\to X$.
This morphism is called a smooth atlas for $X$.
Unless specified otherwise, in this work all derived stacks are supposed to be 1-Artin, and locally quotient stacks of finite type. We will work over the ground field $\C$.

\smallskip

\subsection{Dualizing complexes and virtual classes}\label{sec:relBMhomology} We will use the formalism introduced in~\cite{Khan} to which we refer for details, see also \cite{Porta-Yu}. For any derived stack $X$ there is an 
$\infty$-category $\mathrm{Sh}_{\Q}(X)$ of constructible $\Q$-sheaves on $X$ which satisfies 
a six-functor formalism, see~\cite[Thm. A.5]{Khan}. The dualizing complex is defined as $\mathbb{D}_X=p^!\mathbb{Q}$ where $p: X \to \Spec(\mathbb{C})$ is the map to the point. The sheaf of Borel-Moore chains on $X$ is 
$$\mathrm{C}_\bullet^{\mathrm{BM}}(X,\Q)=p_*p^!\Q=p_*\mathbb{D}_X \in \mathrm{Sh}_\Q(\Spec(\C)) = D(\Q\text{-mod}).$$
The Borel-Moore homology is obtained as usual by taking derived global sections 
$$H_i(X,\Q)=H_i^{\mathrm{BM}}(X,\Q)=H^{-i}(\mathrm{C}_\bullet^{\mathrm{BM}}(X,\Q)).$$ 
Likewise, the sheaf of cochains and the cohomology groups are defined as
$$\mathrm{C}^\bullet(X,\Q)=p_*p^*\Q=p_*\Q_X, \qquad H^i(X,\Q)=H^{i}(\mathrm{C}^\bullet(X,\Q)). $$
These satisfy all the usual properties, see \cite[Section 2]{Khan}. 
The Borel-Moore homology is insensitive to the derived structure, in the sense that the direct image map 
$H_i(X^{cl},\Q) \to H_i(X,\Q)$ is an isomorphism for any $X$ and $i$. Of crucial importance for us are the notions of 
Gysin pullback and virtual fundamental classes for quasi-smooth morphisms. Let $f: X \to Y$ be a quasi-smooth 
morphism of dimension $d$. There is a map $f^!: H_i(Y,\Q) \to H_{i+2d}(X,\Q)$ which is 
induced by a morphism
$$[f]_{vir}: f^*\mathbb{D}_Y \to f^!\mathbb{D}_Y[-2d]=\mathbb{D}_X[-2d].$$
When $X$ is itself quasi-smooth then the (virtual) fundamental class of $X$ is defined as
$[X]=p^!(1) \in H_{2\dim_X}(X,\Q)$. Note that if $X$ is smooth and 
classical then $[X]$ is just the usual fundamental 
class.
In general, the classes $[X]$ and $[X^{cl}]$ differ.
In fact they live in different homological degrees of $H_*(X,\Q)=H_*(X^{cl},\Q)$.

\smallskip

 We collect here some of the basic properties of derived pullbacks which we will use.

\smallskip

\begin{proposition}\label{Prop:appA} Let $f: X \to Y$ be a quasi-smooth morphism of derived stacks. 
\begin{enumerate}[label=$\mathrm{(\alph*)}$,leftmargin=8mm,itemsep=.5em]
\item If $f$ is an open embedding then $f^!=f^*$,
\item If $g : Y \to Z$ is quasi-smooth then $[g]_{vir}[f]_{vir}=[gf]_{vir} : g^*f^*\mathbb{D}_Z \to \mathbb{D}_X[-2\dim(f)-2\dim(g)]$. We have 
$$(gf)^!=g^!f^! : H_*(Z,\Q) \to H_{*+2\dim(f)+2\dim(g)}(X,\Q)
,\quad
f^!([Y])=[X].$$
\item Compatibility with cap product: for any classes $c \in H^*(Y,\Q)$ and $ c' \in H_*(Y,\Q)$ we have 
$$f^!(c \cap c')=f^*(c) \cap f^!(c').$$
\item Assume that $f$ is proper, representable, of finite Tor-amplitude. 
Hence $f_*: H^*(X,\Q) \to H^*(Y,\Q)$ is well-defined. 
The following projection formula holds: for any classes $c \in H^*(X,\Q)$ and $\alpha \in H_*(Y,\Q)$, 
we have $$f_*(c \cap f^!(\alpha))=f_*(c) \cap \alpha.$$
\item Consider the following Cartesian diagram of derived stacks with $f$ quasi-smooth and $g$ proper
$$\xymatrix{X \ar[r]^-{\overline{f}} \ar[d]^-{\overline{g}}&Y \ar[d]^-{{g}} \\ X' \ar[r]^-{f} & Y'}$$
The proper base change formula $\overline{g}_!\overline{f}^!=f^!g_! : H_*(Y,\Q) \to H_*(X',\Q)$ holds.
\end{enumerate}
\end{proposition}

\medskip

\subsection{Relative and hyperbolic homology} Let $X$ be a derived stack and $S$ a scheme.
Let $\pi: X \to S$ be a morphism. Let $p: S \to \mathrm{pt}$ be the projection. 
We define the space of relative Borel-Moore chains on $X/S$ to be
$$C^{BM}_*(X/S,\Q)=\pi_{*}\mathbb{D}_X \in D^b_c(S),$$
We define the $S$-hyperbolic, or simply hyperbolic homology, to be
$$H_i(X/S,\Q)=H^{-i}(p_{!}\pi_{*}\mathbb{D}_X).$$
Note that by definition, for any proper map $f:S\to S'$ we have, denoting $p':S'\to \mathrm{pt}$,
\begin{equation}\label{eq:hyp-change-of-S}
    H_i(X/S',\Q) = H^{-i}(p'_{!}(f\pi)_{*}\mathbb{D}_X) = H^{-i}((p'f)_{!}\pi_{*}\mathbb{D}_X) = H_i(X/S,\Q).
\end{equation}

\begin{example} Taking $X=S$ we get that $H_*(S/S,\Q)$ is the usual homology of $S$. On the other hand, if $S$ is a point then $H_*(X/S,\Q)$ is the Borel-Moore homology of $X$.
\end{example}
When $S$ is understood from the context, we abbreviate 
$$H^c_i(X,\Q)=H_i(X/S,\Q).$$ 
There is an isomorphism $H_*^c(X^{cl},\Q) \simeq H^c_*(X,\Q)$. 
Dually, the space of relative cochains on $X/S$ is
$$C^*(X/S,\Q)=\pi_{*}\mathbb{Q}_X \in D^b(S),$$
and we define the $S$-hyperbolic, or simply hyperbolic cohomology, to be
$$H^i(X/S,\Q)=H^{i}(p_{!}\pi_{*}\mathbb{Q}_X).$$ 
There is a natural morphism $\pi^*:H^*_c(S,\Q) \to H^*(X/S,\Q)$. 

\smallskip

\begin{lemma}\label{lem:actioncohrelBM} 
There is a natural action 
of $H^*(X,\Q)$ on $H_*(X/S,\Q)$ given by
$$\cap: H^i(X,\Q) \otimes H_j(X/S,\Q)\to H_{j-i}(X/S,\Q)
,\quad
i,j\in\Z.$$
There is a natural cap product
$$\cap : H^i(X/S,\Q) \otimes H_j(X,\Q)\to H_{j-i}(X/S,\Q)
,\quad
i,j\in\Z.$$
\end{lemma}

\begin{proof} Consider the following diagram
$$\xymatrix{
X \ar[d]^-{\Delta} \ar[r]^-{\pi} & S \ar[d]^-{\Delta'} \ar[dr]^-{\Id} &&\\
X \times X \ar[r]^{\pi \times \pi} & S \times S\ar[r]^-{pr_2} & S \ar[r]^-{p} & \{pt\}
}$$
where $pr_2: S \times S\to S$ is the projection onto the second factor.
We have $\Delta^*(\mathbb{Q}_X \boxtimes \,\mathbb{D}_X) = \mathbb{D}_X$. The adjunction $\Id \to \Delta_*\Delta^*$ yields a map
$\mathbb{Q}_X \boxtimes\, \mathbb{D}_X \to \Delta_*\mathbb{D}_X$. Applying $(\pi \times \pi)_*$ we get a morphism $(\pi \times \pi)_*( \mathbb{Q}_X \boxtimes \mathbb{D}_X ) \to \Delta'_*\pi_{*}\mathbb{D}_X$, hence a map $p_{!}pr_{2*}(\pi \times \pi)_*( \mathbb{Q}_X \boxtimes \mathbb{D}_X ) \to p_{!}\pi_{*}\mathbb{D}_X$. The construction of the second cap product is similar.
\end{proof}

\smallskip

\begin{lemma}\label{L:projformula} Let $X,Y$ be derived stacks, $S$ a scheme. Let $\pi_X: X \to S, \pi_Y: Y \to S$ and $f: X \to Y$ be morphisms such that $\pi_Y \circ f = \pi_X$.  Then
\begin{enumerate}[label=$\mathrm{(\alph*)}$,leftmargin=8mm,itemsep=.5em]
\item Assume that $f$ is quasi-smooth of relative dimension $d$. 
There is a map $f^! :H_i(Y/S,\Q) \to H_{i+2d}(X/S,\Q)$.
\item Assume that $f$ is proper. There is a map $f_* :H_i(X/S,\Q) \to H_{i}(Y/S,\Q)$. 
The projection formula holds, i.e., for any classes $c \in H^*(Y,\Q)$ and $x \in H_*(X/S)$ 
we have $$f_*(f^*(c) \cap x)=c \cap f_*(x).$$
\item For any cartesian diagram of $S$-stacks
$$\xymatrix{X \ar[r]^-{\overline{f}} \ar[d]^-{\overline{g}}&Y \ar[d]^-{{g}} \\ X' \ar[r]^-{f} & Y'}$$
with $f$ smooth and $g$ proper we have $\overline{g}_*\overline{f}^!=f^!g_* : H_*(Y/S,\Q) \to H_*(X'/S,\Q)$. 
\end{enumerate}
\end{lemma}
\begin{proof} Let $p: S \to \mathrm{pt}$ be the projection. Assume that $f$ is quasi-smooth of dimension $d$. The virtual fundamental class gives a morphism $f^*\mathbb{D}_Y \to \mathbb{D}_X[-2d]$. Applying $p_{!}\pi_{X*}$ and using the adjunction $\Id \to f_*f^*$ yields a canonical morphism $p_{!}\pi_{Y*}\mathbb{D}_Y \to p_{!}\pi_{X*}\mathbb{D}_X[-2d]$, proving (a). The construction of the direct image morphism follows directly from the adjunction $f_!f^! \to \Id$. We leave the proof of the projection formula to the reader. It boils down to the commutativity of the following diagram
\[
    \begin{tikzcd}[column sep=large]
        f_*\mathbb{D}_X\boxtimes \mathbb{Q}_Y\ar[r,"\mathbb{Q}_Y \to f_*\mathbb{Q}_X"]\ar[d,"f_*\mathbb{D}_X \to \mathbb{D}_Y"'] & f_*\mathbb{D}_X\boxtimes f_*\mathbb{Q}_X\ar[r,"1 \to\Delta_*\Delta^*"] & (f\times f)_*\Delta_*\mathbb{D}_X\ar[r,Equal]\ar[d,"f_*\mathbb{D}_X \to \mathbb{D}_Y"] &[-30pt] \Delta_*f_*\mathbb{D}_X \\
        \mathbb{D}_Y \boxtimes \mathbb{Q}_Y\ar[rr,"1 \to \Delta_*\Delta^*"] &  & \Delta_*\mathbb{D}_Y & 
    \end{tikzcd}    
\]
The proper base change statement (c) is obtained by taking compactly supported cohomology of the proper base change over the base $S$.
\end{proof}

\medskip

We will need base changes relative to an open immersion $\iota: S^\circ \to S$. For $\pi_X:X \to S$ a derived $S$-stack we set $X^\circ=X \times_{S} S^\circ$. Let $\iota_X: X^\circ \to X$ and $\pi_{X}^\circ: X^\circ \to S^\circ$ 
be the induced maps. We define a pushforward map $i_{X!}: H_*(X^\circ/S^\circ,\Q) \to H_*(X/S,\Q)$ 
to be the composition of the chain of maps
$$p^\circ_!\pi_{X*}^{\circ}\mathbb{D}_{X^\circ}=p^\circ_!\pi_{X*}^{\circ}i_X^!\mathbb{D}_{X}=p^\circ_!i_{S}^!\pi_{X*}\mathbb{D}_{X}=p_!i_{S!}i_{S}^!\pi_{X*}\mathbb{D}_{X} \stackrel{i_{S!}i_S^! \to 1}{\longrightarrow} p_!\pi_{X*}\mathbb{D}_{X}.$$

\begin{example}\label{ex:relbmhomology} Assume that $S$ is proper. Then the composed map
$$\iota^*\iota_{!} : H^c_*(S^\circ,\Q)=H_*(S^\circ/S^\circ,\Q) \to H_*(S/S,\Q) = H_*(S,\Q) \to H_*(S^\circ,\Q)$$ 
is the canonical morphism from usual homology to Borel-Moore homology.
\end{example}

\smallskip

\begin{proposition}\label{P:basechange} Let $S^\circ, S$ be as above. Let $\pi_X: X \to S, \pi_Y : Y \to S$ be two derived $S$-stacks, $f: X \to Y$ a morphism of $S$-stacks and let $f^\circ: X^\circ \to Y^\circ$ be the base change of $f$. 
\begin{enumerate}[label=$\mathrm{(\alph*)}$,leftmargin=8mm,itemsep=.5em]
\item If $f$ is quasi-smooth then 
$$f^{!}i_{Y!}=i_{X!}f^{\circ !} : H_*(Y^\circ/S^\circ,\Q) \to H_{*+2\dim(f)}(X/S,\Q).$$
\item If $f$ is proper then $f_*i_{X!}=i_{Y!}f^\circ_*: H_*(X^\circ/S^\circ,\Q) \to H_{*}(Y/S,\Q)$. 
\end{enumerate}
\end{proposition}
\begin{proof} This follows from some tedious but unimaginative diagram chasing (see also \cite[\S 3, 4]{Porta-Yu} where similar results are proven in a dual context). 
\end{proof}


\medskip

\begin{thebibliography}{99}

\bibitem{AGT}
Alday, L., Gaiotto, D., Tachikawa, Y., \emph{Liouville correlation functions from four dimensional gauge theories}, Lett. Math. Phys., \textbf{91} (2010), 167--197. 

\bibitem{AS}
Arbesfeld, N., Schiffmann, O., \emph{A presentation of the deformed $W_{1+\infty}$ algebra}, Symmetries, integrable systems and representations, 1--13.
Springer Proc. Math. Stat., \textbf{40}, (2013).

\bibitem{BDHitchin}
Beilinson, A., Drinfeld, V., \emph{Quantization of Hitchin's integrable system and Hecke eigensheaves}, available at \href{https://math.uchicago.edu/~drinfeld/langlands/QuantizationHitchin.pdf}{https://math.uchicago.edu/$\sim$drinfeld/langlands/QuantizationHitchin.pdf}

\bibitem{Coutinho}
Coutinho, Severino C., \emph{A primer of algebraic D-modules}, No. 33. Cambridge University Press, 1995.

\bibitem{DavisonCY2}
Davison, B., \emph{The critical CoHA of a quiver with potential}, Q. J. Math. \textbf{68} (2017), no.2, 635--703.

\bibitem{Davison22}
Davison, B., \emph{Affine BPS algebras, W algebras, and the cohomological Hall algebra of $\mathbb{A}^2$}, \href{https://arxiv.org/abs/2209.05971}{arXiv:2209.05971}.

\bibitem{Davisonpurity2CY}
Davison, B., \emph{Purity in 2CY categories}, Invent. Math. \textbf{238} (2024), no. 1, 69--173.

\bibitem{DEL}
Donagi, R., Ein, L., Lazarsfeld, R., \emph{Nilpotent cones and sheaves on K3 surfaces}, Contemp. Math. \textbf{207} (1997), 51--61.

\bibitem{DHSM}
Davison, B., Hennecart, L., Schlegel Mejia, S., \emph{BPS Lie algebras for totally negative 2-Calabi-Yau categories and nonabelian Hodge theory for stacks}, \href{https://arxiv.org/abs/2212.07668}{arXiv:2212.07668}.

\bibitem{DHSM2}
Davison, B., Hennecart, L., Schlegel Mejia, S., \emph{BPS Lie algebras and generalised Kac--Moody algebras from 2-Calabi-Yau categories}, \href{https://arxiv.org/abs/2303.12592}{arXiv:2303.12592}.

\bibitem{D71} Deligne, P., Th\'eorie de Hodge II, Publ. Math. IHES \textbf{40} (1971), 5--57.

\bibitem{DiacoPortaSala}
Diaconescu, E., Porta, M., Sala, F., \emph{Cohomological Hall algebras and their representations via torsion pairs}, \href{https://arxiv.org/abs/2207.08926}{arXiv:2207.08926}.

\bibitem{FeiginTsymbaliuk}
Feigin, B., Tsymbaliuk, A., \emph{Heisenberg action in the equivariant K-theory of Hilbert schemes via Shuffle Algebra}, Kyoto J. Math. \textbf{51} (2011), no. 4, 831--854.

\bibitem{Ginzburg}
Ginzburg, V., \emph{The global nilpotent variety is lagrangian}, Duke Math. J. \textbf{109} (2001), no. 3, 511--519.

\bibitem{GKM}
Goresky, M., Kottwitz, R., MacPherson R., \emph{Equivariant cohomology, Koszul duality, and the localization theorem}, Invent. Math. \textbf{131} (1998), no. 1, 25--83.

\bibitem{GS93} 
Gottsche, L., Soergel, W., \emph{Perverse sheaves and the cohomology of Hilbert schemes of smooth algebraic surfaces}, Math. Ann. \textbf{296} (1993), 235--245.

\bibitem{Heinloth}
Heinloth, J., \emph{Cohomology of the moduli stack of coherent sheaves on a curve}, Geometry and arithmetic, EMS Series of Congress Reports (2012), 165--171. 

\bibitem{HMMS}
Hausel, T., Mellit, A., Minets, A., Schiffmann, O., \emph{$P=W$ via $\mathcal{H}_2$}, \href{https://arxiv.org/abs/2209.05429}{arXiv:2209.05429}.

\bibitem{Huybrechts-Lehn}
Huybrechts, D., Lehn, M., \emph{The geometry of moduli spaces of sheaves}, 2nd edition, Cambridge University Press 2010, ISBN 978-0-521-13420-0.

\bibitem{QJiang}
Jiang, Q., \emph{Derived projectivization of complexes}, Mem. Amer. Math. Soc. \textbf{316} (2025), no. 1604, v+131 pp.

\bibitem{KPS}
T. Kinjo, H. Park, P. Safronov, \emph{Cohomological Hall algebras for 3-Calabi-Yau categories}, \href{https://arxiv.org/abs/2406.12838}{arXiv:2406.12838}.

\bibitem{Khan}
Khan, A.A., \emph{Virtual fundamental classes of derived stacks I}, \href{https://arxiv.org/abs/1909.01332}{arXiv:1909.01332}.

\bibitem{Kinjo}
Kinjo, T., \emph{Dimensional reduction in cohomological Donaldson-Thomas theory},
Compos. Math. \textbf{158}, No. 1, 123--167 (2022).

\bibitem{KS}
Kontsevich, M., Soibelman, Y., \emph{Cohomological Hall algebra, exponential Hodge structures and motivic Donaldson-Thomas invariants}, Commun. Number Theory Phys. \textbf{5}, No. 2, 231--352 (2011). 

\bibitem{KV}
Kapranov, M., Vasserot, E., \emph{The cohomological Hall algebra of a surface and factorization cohomology}, J. Eur. Math. Soc. \textbf{25} (2023), no. 11, pp. 4221--4289.

\bibitem{Lehn}
Lehn, M., \emph{Chern classes of tautological sheaves on Hilbert schemes of points on surfaces}, Invent. Math. \textbf{136} (1999), no. 1, 157--207.

\bibitem{Lehnsorger}
Lehn, M., Sorger, C., \emph{Symmetric groups and the cup product on the cohomology of Hilbert schemes},
Duke Math. J. \textbf{110} (2001), no. 2, 345--357.

\bibitem{LQW}
Li, W.-P., Qin, Z., Wang, W., \emph{Hilbert schemes and W-algebras}, Intern. Math. Res. Notices \textbf{27} (2002), 1427--1456.

\bibitem{Macdonald}
Macdonald, I., \emph{Symmetric functions and Hall polynomials}, Second edition. With contributions by A. Zelevinsky. Oxford Mathematical Monographs. Oxford Science Publications. The Clarendon Press, Oxford University Press, New York, 1995.

\bibitem{Markman}
Markman, E., \emph{Generators of the cohomology ring of moduli spaces of sheaves on symplectic surfaces}, J. Reine Angew. Math. \textbf{544} (2002), 61--82. 

\bibitem{MO}
Maulik, D., Okounkov, A., \emph{Quantum Groups and Quantum Cohomology}, Ast\'erisque(2019), no. \textbf{408}, ix+209 pp.

\bibitem{Minets}
Minets, A., \emph{Cohomological Hall algebras for Higgs torsion sheaves, moduli of triples and sheaves on surfaces}, Sel. Math., New Ser. \textbf{26}, No. 2, Paper No. 30, 67 p. (2020).

\bibitem{Mukai}
Mukai, S., \emph{Symplectic structure of the moduli space of sheaves on an abelian or K3 surface}, Invent. Math. \textbf{77} (1984), 101--116.

\bibitem{NakLectures}
Nakajima, H., \emph{Lectures on Hilbert schemes of points on surfaces}, University Lecture Series, 18. American Mathematical Society, Providence, RI, 1999.

\bibitem{NegutShuffle}
Negut, A., \emph{Shuffle algebras associated to surfaces}, Sel. Math. New Ser. \textbf{25}, 36 (2019).

\bibitem{Negut}
Negut, A., \emph{Hecke correspondences for smooth moduli spaces of sheaves},  Publ. Math. Inst. Hautes \'Etudes Sci. \textbf{135} (2022), 337--418. 

\bibitem{Porta-Sala}
Porta, M., Sala, F., \emph{Two-dimensional categorified Hall algebras}, J. Eur. Math. Soc. \textbf{25} (2023), no. 3, pp. 1113--1205.

\bibitem{Porta-Yu}
Porta, M., Yue Yu, T., \emph{Non-archimedean Gromov-Witten invariants}, \href{https://arxiv.org/abs/2209.13176}{arXiv:2209.13176}.

\bibitem{Q18} Qin, Z., \emph{Hilbert Schemes of Points and Infinite Dimensional Lie Algebras},  Mathematical Surveys and Monographs, 228. American Mathematical Society, Providence, RI, 2018.

\bibitem{SalaSchiffmann}
Sala, F., Schiffmann, O., \emph{Cohomological Hall algebra of Higgs sheaves on a curve}, Algebr. Geom. \textbf{7} (2020), no. 3, 346--376.

\bibitem{SVDuke}
Schiffmann, O., Vasserot, E., \emph{The elliptic Hall algebra and the K-theory of the Hilbert scheme of $\mathbb{A}^2$}, Duke Math. J. \textbf{162} (2013), no. 2, 279--366. 

\bibitem{SVIHES}
Schiffmann, O., Vasserot, E., \emph{Cherednik algebras, W-algebras and the equivariant cohomology of the moduli space of instantons on $\mathbb{A}^2$}, Publ. Math. Inst. Hautes Études Sci. \textbf{118} (2013), 213--342.

\bibitem{STV15}
Schürg, T, Toën B., Vezzosi G., \emph{Derived algebraic geometry, determinants of perfect complexes, and applications to obstruction theories for maps and complexes}, J. reine angew. Math. \textbf{702} (2015), 1--40.

\bibitem{SVV}
Shan, P., Varagnolo, M., Vasserot, E., \emph{On the center of quiver Hecke algebras}, Duke Math. J. \textbf{166} (2017), no.6, 1005--1101.

\bibitem{Toda}
Toda, Y., \emph{ Hall-type algebras for categorical Donaldson-Thomas theories on local surfaces}, Selecta Math. (N.S.) \textbf{26} (2020), no. 4, Paper No. 62.

\bibitem{Tsymbaliuk}
Tsymbaliuk, A., \emph{The affine Yangian of $\mathfrak{gl}_1$ revisited}, Adv. Math. \textbf{304} (2017), 583--645.

\bibitem{Vasserot}
Vasserot, E., \emph{Sur l'anneau de cohomologie du schéma de Hilbert de $\mathbb{C}^2$}, C. R. Acad. Sci. Paris S\'er. I Math. \textbf{332} (2001), no. 1, 7--12.

\bibitem{Yang-Zhao}
Yang, Y., Zhao, G., \emph{The cohomological Hall algebra of a preprojective algebra}, Proc. Lond. Math. Soc. (3) \textbf{116}, No. 5, 1029--1074 (2018).

\bibitem{Zhao}
Zhao, Y., \emph{On the K-theoretic Hall algebra of a surface}, Int. Math. Res. Not. IMRN 2021, no. \textbf{6}, 4445--4486.

\end{thebibliography}
\end{document}